\documentclass[a4paper, reqno]{amsart}
\usepackage{etex}  

\usepackage[T1]{fontenc}
\usepackage[latin1]{inputenc}
\usepackage{newlfont, verbatim, indentfirst, enumerate}

\usepackage{amsmath, amsthm, amssymb, amscd, mathrsfs}
\usepackage{latexsym, amsfonts, amsbsy, mathrsfs}
\usepackage{array, pst-all, graphicx, rotating}
\usepackage{tikz} \usetikzlibrary{arrows} \usetikzlibrary{patterns}
\usepackage[unicode]{hyperref}

\allowdisplaybreaks
\numberwithin{equation}{section}

\theoremstyle{plain}
\newtheorem{theorem}{Theorem}[section]
\newtheorem{theo}[theorem]{Theorem}
\newtheorem{prop}[theorem]{Proposition}
\newtheorem{lem}[theorem]{Lemma}   
\newtheorem{cor}[theorem]{Corollary}

\theoremstyle{definition}
\newtheorem{defin}[theorem]{Definition}
\newtheorem{rem}[theorem]{Remark}
\newtheorem{construction}[theorem]{Construction}
\newtheorem{descrip}[theorem]{Description}

\renewenvironment{itemize}{\begin{list}{$\bullet$}{\leftmargin=0.5cm}\parindent=0pt}{\end{list}}

\newcommand {\functor} [1] {\mathcal{#1}}  
\newcommand {\atlas} [1] {\mathcal{#1}}  
\newcommand {\atlasmaprep} [2] {\hat{#1}^{#2}}  
\newcommand {\atlasmap} [2] {[\hat{#1}^{#2}]}  
\newcommand {\groupoid} [1] {\mathscr{#1}}  
\newcommand {\groupoidtot} [2] {\groupoid{#1}^{#2}_{\bullet}}  
\newcommand {\groupname} [2] {\groupoid{#1}^{#2}_1\rightrightarrows^{\hspace{-0.25 cm}^{s}}_{\hspace{-0.25 cm}_{t}}\groupoid{#1}^{#2}_0}  
\newcommand {\groupoidmap} [2] {({#1}^{#2}_0,{#1}^{#2}_1)}  
\newcommand {\groupoidmaptot} [2] {{#1}^{#2}_{\bullet}}  
\newcommand {\stack} [1] {\mathfrak{#1}}  
\newcommand {\fiber} [4] {{#1}\,_{#2}\times_{#3}\, {#4}}  
\newcommand {\unifX} [1] {(\widetilde{X}_{#1},G_{#1},\pi_{#1})}  
\newcommand {\unifY} [1] {(\widetilde{Y}_{#1},H_{#1},\chi_{#1})} 
\newcommand {\emphatic} [1] {\emph{#1}}

\def \tX {\widetilde{X}}
\def \tY {\widetilde{Y}}
\def \tZ {\widetilde{Z}}

\def \tx {\widetilde{x}}
\def \ty {\widetilde{y}}

\def \overf {\overline{f}}  
\def \overg {\overline{g}}
\def \overh {\overline{h}}
\def \tf {\widetilde{f}}
\def \tg {\widetilde{g}}

\def \change {\mathcal{C}h}  
\def \dom {\operatorname{dom}}  
\def \cod {\operatorname{cod}}  
\def \germ {\operatorname{germ}}  
\def \red {\,\operatorname{red}}  
\def \reg {\,\operatorname{reg}}  
\def \ind {\operatorname{ind}}  
\def \Stab {\operatorname{Stab}}  
\def \max {\operatorname{max}}  
\def \id {\operatorname{id}}  
\def \pr {\operatorname{pr}}  
 
\def \RedAtl {(\mathbf{\mathcal{R}ed\,\mathcal{A}tl})}  
\def \RedOrb {(\mathbf{\mathcal{R}ed\,\mathcal{O}rb})}  

\def \LieGpd {(\mathbf{\mathcal{L}ie\,\mathcal{G}pd})}  
\def \EGpd {(\mathbf{\mathcal{\acute{E}}\,\mathcal{G}pd})}  
\def \PEGpd {(\mathbf{\mathcal{P\acute{E}}\,\mathcal{G}pd})}  
\def \PEEGpd {(\mathbf{\mathcal{PE\acute{E}}\,\mathcal{G}pd})} 

\def \DiffStacks {(\mathbf{C}^{\infty}\textrm{-}\mathbf{Stacks})}  
\def \Orb {(\mathbf{Orb})}  
\def \OrbEff {(\mathbf{Orb}^{\mathbf{eff}})}  

\def \CATA {\mathbf{\mathscr{A}}}  
\def \CATB {\mathbf{\mathscr{B}}}
\def \CATC {\mathbf{\mathscr{C}}}

\def \SETW {\mathbf{W}}  
\def \SETWinv {\mathbf{W}^{-1}} 
\def \SETWsat{\mathbf{W}_{\mathbf{\operatorname{sat}}}}  

\def \SETWA {\mathbf{W}_{\mathbf{\mathscr{A}}}}  
\def \SETWAsat{\mathbf{W}_{\mathbf{\mathscr{A}},\mathbf{\operatorname{sat}}}}
\def \SETWAinv {\mathbf{W}^{-1}_{\mathbf{\mathscr{A}}}}

\def \SETWB {\mathbf{W}_{\mathbf{\mathscr{B}}}}  
\def \SETWBinv {\mathbf{W}^{-1}_{\mathbf{\mathscr{B}}}}
\def \SETWBsat{\mathbf{W}_{\mathbf{\mathscr{B}},\mathbf{\operatorname{sat}}}}
\def \SETWBsatinv{\mathbf{W}^{-1}_{\mathbf{\mathscr{B}},\mathbf{\operatorname{sat}}}}

\def \WRedAtl {\mathbf{W}_{\mathbf{\mathcal{R}ed\,\mathcal{A}tl}}}  
\def \WRedAtlsat {\mathbf{W}_{\mathbf{\mathcal{R}ed\,\mathcal{A}tl},\operatorname{sat}}} 
\def \WRedAtlinv {\mathbf{W}_{\mathbf{\mathcal{R}ed\,\mathcal{A}tl}}^{-1}}  

\def \WEGpd {\mathbf{W}_{\mathbf{\mathcal{\acute{E}}\,\mathcal{G}pd}}}  
\def \WEGpdinv {\mathbf{W}_{\mathbf{\mathcal{\acute{E}}\,\mathcal{G}pd}}^{-1}}

\def \WPEGpd {\mathbf{W}_{\mathbf{\mathcal{P\acute{E}}\,\mathcal{G}pd}}}  
\def \WPEGpdinv {\mathbf{W}_{\mathbf{\mathcal{P\acute{E}}\,\mathcal{G}pd}}^{-1}}

\def \WPEEGpd {\mathbf{\mathbf{W}_{\mathcal{PE\acute{E}}\,\mathcal{G}pd}}} 
\def \WPEEGpdinv {\mathbf{W}_{\mathbf{\mathcal{PE\acute{E}}\,\mathcal{G}pd}}^{-1}}

\hypersetup{unicode=true, pdftoolbar=true, pdfmenubar=true, pdffitwindow=false,
  pdfstartview={FitH},pdftitle={A bicategory of reduced orbifolds from the point of view of
  differential geometry - I}, pdfauthor={Matteo Tommasini}, pdfkeywords={reduced orbifolds, Lie
  groupoids, differentiable stacks, 2-categories, bicategories}, pdfnewwindow=true,
  colorlinks=false, linkcolor=red, citecolor=red, filecolor=magenta, urlcolor=cyan}

\numberwithin{equation}{section}

\begin{document}

\title[A bicategory of reduced orbifolds - I]
{A bicategory of reduced orbifolds from the point of view of differential geometry - I}

\author{Matteo Tommasini}

\address{\flushright
Mathematics Research Unit\newline
University of Luxembourg\newline
6, rue Richard Coudenhove-Kalergi\newline
L-1359 Luxembourg\newline\newline
website: \href{http://matteotommasini.altervista.org/}
{\nolinkurl{http://matteotommasini.altervista.org/}}\newline\newline
email: \href{mailto:matteo.tommasini2@gmail.com}{\nolinkurl{matteo.tommasini2@gmail.com}}}

\date{\today}

\thanks{I would like to thank Dorette Pronk and Anke Pohl for help and useful suggestions when I
was writing this paper, and Barbara Fantechi for suggesting this problem to me. I would
also like to acknowledge the Institute of Algebraic Geometry and the Riemann Center of the Leibniz
Universit\"{a}t Hannover, where the initial part of this research was performed. This research was
mainly done at the Mathematics Research Unit of the University of Luxembourg, thanks to the grant
4773242 by Fonds National de la Recherche Luxembourg. I would also like to thank Carnegie Mellon
University for hospitality in March 2013, when part of this paper was written.}

\begin{abstract}
We describe a bicategory $\RedOrb$ of reduced orbifolds in the framework of classical differential
geometry (i.e.\ without any explicit reference to notions of Lie groupoids or differentiable
stacks, but only using orbifold atlases, local lifts and changes of charts). In order to construct
such a bicategory, we first define a $2$-category $\RedAtl$ whose objects are reduced orbifold
atlases (on any paracompact, second countable, Hausdorff topological space). The definition of
morphisms is obtained as a slight modification of a definition by A.~Pohl, while the definitions
of $2$-morphisms and compositions of them is new in this setup. Using the bicalculus of fractions
described by D.~Pronk, we are able to construct the bicategory $\RedOrb$ from the $2$-category
$\RedAtl$. We prove that $\RedOrb$ is equivalent to the bicategory of reduced orbifolds
described in terms of proper, effective, \'etale Lie groupoids by D.~Pronk and I.~Moerdijk and to
the $2$-category of reduced orbifolds described by several authors in the past in terms of a
suitable class of differentiable Deligne-Mumford stacks.
\end{abstract}

\maketitle

\begingroup{\hypersetup{linkbordercolor=white}
\tableofcontents}\endgroup

\section*{Introduction}
A well-known issue in mathematics is that of modeling geometric objects where points have
non-trivial groups of automorphisms. In topology and differential geometry the standard approach
to these objects (when each point has a finite group of automorphisms) is through orbifolds. This
concept was formalized for the first time by Ikiro Satake in 1956 in~\cite{Sa} with some different
hypotheses than the current ones, although the informal idea dates back at least to Henri
Poincar\'e (for example, see~\cite{Poi}). Currently there are at least $3$ main approaches to
orbifolds:

\begin{enumerate}[(1)]
 \item via orbifold atlases and ``good maps'' between them, as described in~\cite{CR},
 \item via a class of Lie groupoids, namely proper, \'etale groupoids
  (see for example~\cite{Pr}, \cite{M} and~\cite{MM}),
 \item via a class of $C^{\infty}$-Deligne-Mumford stacks (see for example~\cite{J1}
  and~\cite{J2}).
\end{enumerate}

On the one hand, the approach in (1) gives rise to a $1$-category. On the other hand, the approach in
(2) gives rise to a bicategory (i.e.\ almost a $2$-category, except that compositions of
$1$-morphisms is associative only up to canonical $2$-morphisms) and the approach in (3) gives
rise to a $2$-category. It was proved in~\cite{Pr} that (2) and (3) are equivalent
bicategories. Since (2) and (3) are compatible approaches, then one might argue that:

\begin{enumerate}[(i)]
 \item there should also exist a non-trivial structure of $2$-category or bicategory having as
  objects orbifold atlases or equivalence classes of them (i.e.\ orbifold structures);
 \item the structure of (i) should be compatible with the approaches of (2) and (3) and it
  should replace the approach of (1) (since (1) gives rise only to a category instead of a
  $2$-category or bicategory).
\end{enumerate}

In the present paper we manage to prove both (i) and (ii) for the family of all \emph{reduced}
orbifolds, i.e.\ orbifolds that are locally modeled on open connected sets of some $\mathbb{R}^n$,
modulo finite groups acting smoothly and \emph{effectively} on them. In order to do that, we
proceed as follows.

\begin{itemize}
 \item We describe a $2$-category $\RedAtl$ whose objects are reduced orbifold atlases on any
  paracompact, second countable, Hausdorff topological space. The definition of morphisms is
  obtained as a slight modification of an analogous definition given by Anke Pohl in~\cite{Po},
  while the notion of $2$-morphisms (and compositions of them) is new in this setup (see
  Definitions~\ref{def-03} and~\ref{def-06}). Such notions are useful for differential
  geometers mainly because they don't require any previous knowledge of Lie groupoids and/or
  differentiable stacks. In Proposition~\ref{prop-01} we will prove that $\RedAtl$ is a 
  $2$-category, but it is
  still not the structure that we want to get in (i); indeed in $\RedAtl$ different orbifold
  atlases that represent the same orbifold structure are not related by an isomorphism neither by
  an internal equivalence.
 \item We recall briefly the definition of the $2$-category $\PEEGpd$, whose objects are proper,
  effective, \'etale groupoids and we describe in Theorem~\ref{theo-01} a $2$-functor
  
  \[\functor{F}^{\red}:\RedAtl\longrightarrow\PEEGpd.\]
  
 \item In~\cite{Pr} Dorette Pronk proved that the set $\WPEEGpd$ of all Morita equivalences (also
  known as weak equivalences or essential equivalences) in $\PEEGpd$ admits a right bicalculus
  of fractions. Roughly speaking, this amounts to saying that it is possible to construct a
  bicategory $\PEEGpd\left[\WPEEGpdinv\right]$ and a pseudofunctor
  
  \[\functor{U}_{\WPEEGpd}:\PEEGpd\longrightarrow\PEEGpd\left[\WPEEGpdinv\right]\]
  that sends each weak equivalence to an internal equivalence and that is universal with respect to
  this property (see Proposition~\ref{prop-03}). The bicategory obtained in this way is the
  bicategory that we mentioned in (b) above if we restrict to the case of reduced orbifolds.
 \item In $\RedAtl$ we consider a class $\WRedAtl$ of morphisms (that we call ``refinements''
  of reduced orbifold atlases, see
  Definition~\ref{def-10}) and we prove that such a class
  admits a right bicalculus of fractions. Therefore, we are able to construct a bicategory
  $\RedOrb$ and a pseudofunctor
  
  \[\functor{U}_{\WRedAtl}:\RedAtl\longrightarrow\RedOrb:=\RedAtl\left[\WRedAtlinv\right]\]
  that sends each refinement to an internal equivalence and that is universal with respect
  to this property (see Proposition~\ref{prop-05}). Objects in this new bicategory are again reduced
  orbifold atlases; a morphism from an atlas $\atlas{X}$ to an atlas $\atlas{Y}$ is a triple
  consisting of a reduced orbifold atlas $\atlas{X}'$, a refinement $\atlas{X}'
  \rightarrow\atlas{X}$ and a morphism $\atlas{X}'\rightarrow\atlas{Y}$. In other terms, a morphism
  from $\atlas{X}$ to $\atlas{Y}$ is given firstly by replacing $\atlas{X}$ with a ``refined''
  atlas $\atlas{X}'$ (keeping track of the refinement),
  then by considering a morphism from $\atlas{X}'$ to $\atlas{Y}$. We
  refer to Lemma~\ref{lem-25} for the description of $2$-morphisms in this bicategory.
 \item Using the axiom of choice and the results about bicategories of fractions that we
  proved in our previous papers~\cite{T4} and~\cite{T5}, we are able to prove that:\\
  
  \noindent\textbf{Theorem A} (Proposition~\ref{prop-06} and Theorem~\ref{theo-05}).
  \emph{There is an equivalence of bicategories
  $\functor{G}^{\red}$ making the following diagram commute:}
 
  \[\begin{tikzpicture}[xscale=2.5,yscale=-0.8]
    \node (A0_0) at (0, 0) {$\RedAtl$};
    \node (A0_2) at (2, 0) {$\PEEGpd$};
    \node (A2_0) at (0, 2) {$\RedOrb$};
    \node (A2_2) at (2, 2) {$\PEEGpd\left[\WPEEGpdinv\right]$.};
    
    \node (B1_1) at (1, 1) {$\curvearrowright$};

    \path (A0_0) edge [->]node [auto,swap] {$\scriptstyle{\functor{U}_{\WRedAtl}}$} (A2_0);
    \path (A2_0) edge [->,dashed]node [auto,swap] {$\scriptstyle{\functor{G}^{\red}}$} (A2_2);
    \path (A0_2) edge [->]node [auto] {$\scriptstyle{\functor{U}_{\WPEEGpd}}$} (A2_2);
    \path (A0_0) edge [->]node [auto] {$\scriptstyle{\functor{F}^{\red}}$} (A0_2);
  \end{tikzpicture}\]
  This proves that the approach described in $\RedOrb$ is compatible with the approach
  (2) to reduced orbifolds in terms of proper, effective, \'etale Lie groupoids. Since (2) and (3)
  are equivalent approaches by~\cite{Pr}, this implies at once that:\\
  
  \noindent\textbf{Theorem B} (Theorem~\ref{theo-03}). \emph{$\RedOrb$ is
  equivalent to the $2$-category
  $\OrbEff$ of effective orbifolds described as a full $2$-subcategory of the $2$-category of
  $C^{\infty}$-Deligne-Mumford stacks.}\\
\end{itemize}

Even if we will not use explicitly the language of stacks in all this paper, we think that
it is important to remark the following $2$ facts:

\begin{itemize}
 \item in the language of (differentiable) stacks, the notion of objects is long and complicated
  to be stated precisely: it requires the notion of pseudofunctor (or the notion of
  category fibered in
  groupoids), Grothendieck topology and descent conditions. Having described that, morphisms
  (and $2$-morphisms) are almost straightforward to define and the resulting structure is that of
  a $2$-category;
 \item in the language used in the present paper (that is mostly intended to be used by
  differential geometers), objects are very easy to describe: they are simply reduced orbifold
  atlases; as we mentioned above, also morphism are easy to describe. On the contrary,
  the definitions of $2$-morphisms between such objects will be a bit longer
  (see Lemma~\ref{lem-25}) and the resulting structure will be that of a bicategory (hence composition
  of morphisms is associative only up to canonical $2$-morphisms).
\end{itemize}

To summarize, this paper provides a suitable bicategory of reduced orbifolds, that is equivalent
to the already known bicategories of reduced orbifolds that are standard in the literature. Its
main advantage is that its objects are reduced atlases, so it gives a description that is closer to
classical differential geometry than the descriptions (2) and (3) given in terms of Lie
groupoids or differentiable stacks.\\

In the literature there are already other attempts to define morphisms of (reduced) orbifolds in
terms that are useful for differential geometers:

\begin{enumerate}[(a)]
 \item the ``smooth maps'' defined for example in~\cite[Definition~1.3]{ALR};
 \item the ``good maps'' described by Weimin Chen and Yongbin Ruan in~\cite{CR};
 \item the orbifold maps described by Anke Pohl (only for the reduced case) in~\cite{Po}.
\end{enumerate}

The maps in (a) were the first ones to be defined, but it turned out that they were
not good enough: in general one could not pullback orbifold vector bundles along such maps (and in
the case when this was possible, the pullback was not unique up to isomorphism). This led Chen and
Ruan to introduce the concept of good maps. Such maps proved to be good enough in order to define
pullbacks of orbifold vector bundles (and fiber products under some assumptions), and they are
currently frequently used in mathematical physics and differential geometry when
dealing with orbifolds. However, they are bad-behaved for the following $2$ reasons:

\begin{itemize}
 \item not all smooth maps of manifolds are good maps;
 \item fiber products (when they exist) do not have the universal property of fiber products in
  a category; in particular, pullbacks of orbifold vector bundles do not have the universal property.
\end{itemize}

The first problem is just a technical mistake in the definition of good maps, and it can be
corrected without much trouble by simply relaxing a bit the technical assumptions on good maps.
However, the second problem is much worse and it cannot be corrected easily. We will exhibit
examples of both bad behaviors in the next paper~\cite{T10}.\\

The definition given in (c) solves the first problem (but not the second one). However,
composition of morphisms is not well-defined in~\cite{Po} (we will show also
this fact in~\cite{T10}).\\

Both in case (b) and in case (c), the bad problems quickly mentioned above are
a consequence of completely ignoring the fact that
orbifolds have much more structure than that of a usual $1$-category. Actually, Theorem B
proves that $\RedOrb$ has a non-trivial structure of bicategory because
it is equivalent to $\OrbEff$ (that
has a non-trivial structure of $2$-category). In~\cite{T10} we will prove that the bad
behaviors of (b) and (c) are given by the following reasons.

\begin{itemize}
 \item The category (b) of reduced orbifolds with good maps is equivalent to 
  the \emph{homotopy category} of $\RedOrb$, i.e.\ the $1$-category obtained by identifying any pair of
  $1$-morphisms of $\RedOrb$ whenever there is an invertible $2$-morphism between them. Now the
  problem is the following: given any weak fiber product in a non-trivial bicategory $\CATB$,
  the corresponding
  commutative square in the homotopy category $\operatorname{Ho}(\CATB)$
  not necessarily have the universal property of
  fiber products in a $1$-category. This leads to all the problems mentioned above
  for fiber products of good maps in (b).
  On the contrary, we will show in~\cite{T9}
  that weak pullbacks of vector bundles in $\RedOrb$ have the universal
  property of weak fiber products in that bicategory.
 \item The definition of maps according to Pohl is obtained in $2$ steps. First of all, one gets
  a notion of ``representative'' of map, that corresponds to the notion of $1$-morphism in $\PEEGpd$;
  the $1$-category $\CATC$ obtained in this way has objects given by reduced orbifold atlases. Since Pohl
  wants to identify orbifold atlases that give the same orbifold structure, she has to construct
  a new $1$-category $\overline{\CATC}$, where morphisms are equivalence classes of the ``representatives''
  mentioned above. The
  way used by Pohl for identifying morphisms takes into account a certain class of
  commutative diagrams of $\PEEGpd$, without considering the existence of
  $2$-commutative diagrams. Because of that, one gets that composition of morphisms in
  $\overline{\CATC}$ is not
  well-defined. The problem can be solved by quotienting the set of morphisms in $\CATC$
  by a bigger equivalence
  relation (taking into account the role played by $2$-morphisms).
  In this way one would get a $1$-category $\widetilde{\CATC}$, that is
  equivalent to the homotopy category of $\RedOrb$. As such, $\widetilde{\CATC}$ would
  solve the bad-behaved definition of
  composition given by Pohl, but it would still carry
  the problems about fiber products mentioned before for (b).
\end{itemize}

To summarize, \emph{both the category constructed by Chen-Ruan \emphatic{(}in the reduced
case\emphatic{)} and the one
defined by Pohl have some serious drawbacks, mainly induced by ignoring the crucial role played by
$2$-morphisms. On the contrary, the bicategory $\RedOrb$ constructed in the present paper solves
such problems} (and it is equivalent to the standard approach to reduced orbifolds
because of Theorems A and B).\\

Apart from that, only one important problem remains open: we have described a bicategory structure
that solves problems (i) and (ii) by \emph{restricting} to the case of reduced orbifolds. Is it
possible to give an analogous description of a bicategory $(\mathbf{\functor{O}rb})$ also in the
more general case of (possibly) non-reduced orbifolds? Since the bicategories of (2) and (3) are
defined (and equivalent) also in this more general setup, in principle this should be
possible, but it seems that this will require much more work.

\section{Reduced orbifold atlases}
Let us review some basic definitions about reduced orbifolds.

\begin{defin}
\cite[\S~1]{MP} Let $X$ be a paracompact, second countable, Hausdorff topological space and let
$X'\subseteq X$ be open and non-empty. Then a \emph{reduced orbifold chart} (also known as
\emph{reduced uniformizing system}) \emph{of dimension} $n$ for $X'$ is the datum of:

\begin{itemize}
 \item a \emph{connected} open subset $\tX$ of $\mathbb{R}^n$;
 \item a \emph{finite} group $G$ of smooth automorphisms of $\tX$;
 \item a continuous, surjective and $G$-invariant map $\pi:\tX\rightarrow X'$, which induces an
  homeomorphism between $\tX/G$ and $X'$, where we give to $\tX/G$ the quotient topology (in
  particular, $\pi$ is an open map).
\end{itemize}
For every point $\tx\in\tX$, we denote by $\Stab(G,\tx)$ the stabilizer of $\tx$ in $G$.
\end{defin}

\begin{rem}
We will always assume that $G$ \emph{acts effectively}; the orbifolds that have this property are
usually called \emph{reduced} or \emph{effective}. Some of the current literature on orbifolds
assumes that $\tX$ is only a connected smooth manifold of dimension $n$ instead of an open
connected subset of $\mathbb{R}^n$. This makes a difference for the definition of charts, but the
arising notion of orbifold is not affected by that. To be more precise, to any orbifold atlas (see
below) where the $\tX$'s are connected smooth manifolds of dimension $n$, one can associate easily
another orbifold atlas where the $\tX$'s are open connected subsets of $\mathbb{R}^n$ and the $2$
orbifold atlases give rise to the same orbifold structure (see below).
\end{rem}

The following definition is a special case of~\cite[\S~2.1]{Po}.

\begin{defin}\label{def-14}
Let us fix any pair of reduced charts $\unifX{1}$ and $\unifX{2}$ for subsets $X_1,X_2$ of
$X$. Then a \emph{change of charts} from $\unifX{1}$ to $\unifX{2}$ is any diffeomorphism
$\lambda:\tY_1\stackrel{\sim}{\rightarrow}\tY_2$ such that:

\begin{itemize}
 \item $\tY_1$ is any connected component of $\pi_1^{-1}(Y)$ for some open non-empty
  subset $Y$ of $X_1$ (since the action of $G_1$ on $\tX_1$ permutes such connected
  components, then $\pi_1(\tY_1)=Y$);
 \item $\tY_2$ is an open subset of $\tX_2$;
 \item $\pi_2\circ\lambda={\pi_1}|_{\tY_1}$.
\end{itemize}

Using~\cite[Lemma~A.2]{MP} and the fact that $\lambda$ is a diffeomorphism, it turns out that
$Y$ is contained also in $X_2$ and that $\tY_2$ is a connected component of
$\pi_2^{-1}(Y)$. So the inverse of any change of charts is again a change of charts. If
$\lambda$ is any change of charts, we denote by $\dom\lambda$ its domain and by $\cod\lambda$ its
codomain. If $\tx\in\dom\lambda$, then $\germ_{\tx}\lambda$ denotes the germ of $\lambda$ at $\tx$.
An \emph{embedding} is any change of charts $\lambda$ as before, such that $\dom\lambda=\tX_1$. $2$
charts as before are called \emph{compatible} if for each pair $\tx_1\in\tX_1$, $\tx_2\in\tX_2$ with
$\pi_1(\tx_1)=\pi_2(\tx_2)$, there exists a change of charts $\lambda$ 
from $\unifX{1}$ to $\unifX{2}$, with $\tx_1\in\dom\lambda$.
Up to composing $\lambda$ with an element of $G_2$, this is the same as
requiring that there exists a change of charts $\lambda$ such that $\tx_1\in\dom\lambda$ and $\lambda
(\tx_1)=\tx_2$.
\end{defin}

\begin{rem}\label{rem-01}
Let us suppose that we have any change of charts $\lambda:\tY_1\rightarrow\tY_2$ from
$\unifX{1}$ to $\unifX{2}$. Then let us fix any point $g_1\in G_1$ and let us suppose that
$g_1(\tY_1)\cap\tY_1\neq\varnothing$; by the hypothesis on $\tY_1$, we conclude that
necessarily $g_1(\tY_1)=\tY_1$. Therefore we can consider the subgroup of $G_1$:

\[G_1(\tY_1):=\{g_1\in G_1\,\,\textrm{s.t.}\,\,g_1(\tY_1)\cap\tY_1\neq\varnothing\}=\{g_1\in G_1
\,\,\textrm{s.t.}\,\,g_1(\tY_1)=\tY_1\}.\]

By~\cite[Lemma~2.10]{MM}, we have that the group $G_1(\tY_1)$ acts \emph{effectively}
on $\tY_1$, so the triple $(\tY_1,G_1(\tY_1),{\pi_1}|_{\tY_1})$ is a reduced orbifold chart;
moreover, $\lambda$ can be considered as an embedding
from $(\tY_1,G_1(\tY_1),{\pi_1}|_{\tY_1})$ to $\unifX{2}$.
\end{rem}

Using Remark~\ref{rem-01}, the following definition is equivalent to~\cite[\S~1]{MP}.

\begin{defin}\label{def-01}
Let $X$ be a paracompact, second countable, Hausdorff topological space; a \emph{reduced orbifold
atlas of dimension} $n$ on $X$ is any family $\atlas{X}=\{\unifX{i}\}_{i\in I}$ of reduced
orbifolds charts of dimension $n$, such that:

\begin{enumerate}[(i)]
 \item the family $\{X_i:=\pi_i(\tX_i)\}_{i\in I}$ is an open cover of $X$;
 \item every pair of charts in $\atlas{X}$ is compatible.
\end{enumerate}

Given any orbifold atlas $\atlas{X}$ as before and any pair $(i,i')\in I\times I$, we denote by
$\change(\atlas{X},i,i')$ the set of all changes of charts $\lambda$ from $\unifX{i}$ to
$(\tX_{i'},G_{i'},\pi_{i'})$ and we set $\change(\atlas{X}):=\coprod_{(i,i')\in I\times I}
\change(\atlas{X},i,i')$.
\end{defin}

\begin{defin}\label{def-02}
\cite[\S~1]{MP} Let $\atlas{X}$ and $\atlas{X}'$ be reduced orbifold atlases for the same
topological space $X$. We say that they are \emph{equivalent} if their union is again an orbifold
atlas for $X$, i.e.\ if and only if any chart of $\atlas{X}$ is compatible with any chart of
$\atlas{X}'$; a \emph{reduced orbifold structure} of dimension $n$ on $X$ is any equivalence
class with respect to compatibility of atlases. A
\emph{reduced orbifold} of dimension $n$ is any pair $(X,[\atlas{X}])$ consisting of a
paracompact, second countable, Hausdorff topological space $X$ and a reduced orbifold structure
$[\atlas{X}]$ on $X$. Any atlas in $[\atlas{X}]$ is called a reduced orbifold atlas for $(X,
[\atlas{X}])$. The notion of being compatible gives also rise to a partial order
on the set of reduced orbifold atlases for $X$; it turns out that given any reduced orbifold
atlas there is exactly one maximal atlas associated to it with respect of this definition,
so a reduced orbifold structure can be equivalently defined as a maximal reduced orbifold atlas.
\end{defin}

\begin{defin}
\cite[\S~4.1]{CR} Let $f: X\rightarrow Y$ be any continuous map between topological spaces and
let $X'\subseteq X$ and $Y'\subseteq Y$ be open subsets such that $f(X')\subseteq Y'$. Let us suppose
that there are reduced orbifold
charts $\unifX{}$ for $X'$ and $\unifY{}$ for $Y'$. Then a \emph{local lift} of $f$
with respect to these $2$ charts is any smooth map $\tf:\tX\rightarrow\tY$, such that $\chi\circ
\tf=f\circ\pi$.
\end{defin}

\begin{defin}
Let us fix any reduced orbifold atlas $\atlas{X}$ as before and let $P$ be any subset of $\change(
\atlas{X})$. We say that $P$ is \emph{a good subset of} $\change(\atlas{X})$ if the following
property holds:

\begin{enumerate}[({GS})]
 \item\label{GS} for each $\lambda\in\change(\atlas{X})$ and for each $\tx\in\dom\lambda$ there exists
  $\hat{\lambda}\in P$ such that $\tx\in\dom\hat{\lambda}$ and $\germ_{\tx}\lambda=\germ_{\tx}
  \hat{\lambda}$.
\end{enumerate}

Since $P$ is a subset of $\change(\atlas{X})$, for each $(i,i')\in I\times I$ we write for
simplicity $P(i,i'):=P\cap\change(\atlas{X},i,i')$ and $P(i,-):=\coprod_{i'\in I}P(i,i')$.
\end{defin}

\begin{rem}\label{rem-02}
In the notations of~\cite{Po}, (\hyperref[GS]{GS}) is the condition that $P$ \emph{generates} the
pseudogroup $\change(\atlas{X})$ inside the larger pseudogroup $\Psi(\atlas{X})$ defined and used
in~\cite{Po}; such a pseudogroup is obtained by taking into account all changes of charts of
$\atlas{X}$ with a more general definition than the one used in the present paper. In~\cite{Po}
there are other $2$ technical conditions (axioms of ``quasi-pseudogroup''), but they are implied
by (\hyperref[GS]{GS}) in our case, so we can omit them. Under this remark, our definition of
morphism of orbifold atlases $\atlas{X}\rightarrow\atlas{Y}$ (Definition~\ref{def-11} and~\ref{def-03}
below) is equivalent to the definition of ``orbifold map with domain atlas $\atlas{X}$ and
range atlas $\atlas{Y}$'' stated in~\cite[Definitions~4.4 and~4.10]{Po}. To be more precise,
our ``representatives of morphisms'' (Definition~\ref{def-11} below) are a subset of the
representatives given in~\cite[Definitions~4.4]{Po}, but the sets of equivalence classes described in
Definition~\ref{def-03} and in~\cite[Definition~4.10]{Po} will be the same.
\end{rem}

\begin{defin}\label{def-11}
Let us fix any pair of reduced orbifold atlases $\atlas{X}=\{\unifX{i}\}_{i\in I}$ and $\atlas{Y}=
\{\unifY{j}\}_{j\in J}$ for $X$ and $Y$ respectively. Then a \emph{representative of a morphism}
from $\atlas{X}$ to $\atlas{Y}$ is any tuple $\atlasmaprep{f}{}:=(f,\overf,\{\tf_i\}_{i\in I},P_f,
\nu_f)$ that satisfies the following conditions:

\begin{enumerate}[({M}1)]
 \item\label{M1} $f:X\rightarrow Y$ is any continuous map;
 \item\label{M2} $\overf:I\rightarrow J$ is any set map such that $f(\pi_i(\tX_i))\subseteq
  \chi_{\overf(i)}(\tY_{\overf(i)})$ for each $i\in I$;
 \item\label{M3} for each $i\in I$, the map $\tf_i$ is a local lift of $f$ with respect to the
  orbifold charts $(\widetilde{U}_i,G_i,\pi_i)\in\atlas{X}$ and $\left(\widetilde{V}_{\overf(i)},
  H_{\overf(i)},\chi_{\overf(i)}\right)\in\atlas{Y}$;
 \item\label{M4} $P_f$ is any good subset of $\change(\atlas{X})$;
 \item\label{M5} $\nu_f:P_f\rightarrow\change(\atlas{Y})$ is any set map that assigns to each
  $\lambda\in P_f(i,i')$ a change of charts $\nu_f(\lambda)\in\change(\atlas{Y},\overf(i),
  \overf(i'))$, such that:

  \begin{enumerate}[(a)]
   \item\label{M5a} $\dom\nu_f(\lambda)$ is an open set containing $\tf_i(\dom\lambda)$,
   \item\label{M5b} $\cod\nu_f(\lambda)$ is an open set containing $\tf_{i'}(\cod\lambda)$,
   \item\label{M5c} $\tf_{i'}\circ\lambda=\nu_f(\lambda)\circ {\tf_i}|_{\dom\lambda}$,
   \item\label{M5d} for all $i\in I$, for all $\lambda,\lambda'\in P_f(i,-)$ and for all $\tx_i\in
    \dom\lambda\cap\dom\lambda'$ with $\germ_{\tx_i}\lambda=\germ_{\tx_i}\lambda'$, we have 

    \[\germ_{\tf_i(\tx_i)}\nu_f(\lambda)=\germ_{\tf_i(\tx_i)}\nu_f(\lambda'),\]

   \item\label{M5e} for all $(i,i',i'')\in I^3$, for all $\lambda_1\in P_f(i,i')$, for all
    $\lambda_2\in P_f(i',i'')$ and for all $\tx_i\in\lambda_1^{-1}(\cod\lambda_1\cap\dom\lambda_2)$,
    we have 

    \[\germ_{\tf_{i'}(\lambda_1(\tx_i))}\nu_f(\lambda_2)\cdot\germ_{\tf_i(\tx_i)}\nu_f(\lambda_1)=
    \germ_{\tf_i(\tx_i)}\nu_f(\lambda_3),\]
    where $\lambda_3$ is any element of $P_f(i,i'')$ such that $\germ_{\tx_i}\lambda_3=
    \germ_{\tx_i}\lambda_2\circ\lambda_1$ (it exists by (\hyperref[GS]{GS})),  
   \item\label{M5f} for all $i\in I$, for all $\lambda\in P_f(i,i)$ and for all $\tx_i\in\dom
   \lambda$ such that $\germ_{\tx_i}\lambda=\germ_{\tx_i}\id_{\tX_i}$, we have

   \[\germ_{\tf_i(\tx_i)}\nu_f(\lambda)=\germ_{\tf_i(\tx_i)}\id_{\tY_{\overf(i)}}.\]
\end{enumerate}
\end{enumerate}
\end{defin}

\begin{defin}\label{def-03}
Given $2$ representatives of morphisms from $\atlas{X}=\{\unifX{i}\}_{i\in I}$ to
$\atlas{Y}=\{\unifY{j}\}_{j\in J}$ as follows:

\[\atlasmaprep{f}{}=\left(f,\overf,\left\{\tf_i\right\}_{i\in I},P_f,\nu_f\right)\quad\textrm{and}
\quad\atlasmaprep{f}{\prime}:=\left(f',\overf',\left\{\tf'_i\right\}_{i\in I},P_{f'},\nu_{f'}\right),
\]
we say that $\atlasmaprep{f}{}$ is \emph{equivalent} to $\atlasmaprep{f}{\prime}$ if and only if $f=
f'$, $\overf=\overf'$, $\tf_i=\tf'_i$ for all $i\in I$, and 

\begin{equation}\label{eq-26}
\germ_{\tf_i(\tx_i)}\nu_f(\lambda)=\germ_{\tf_i(\tx_i)}\nu_{f'}(\lambda')
\end{equation}
for all $i\in I$, for all $\lambda\in P_f(i,-)$, $\lambda'\in P_{f'}(i,-)$ and for all $\tx_i\in\dom
\lambda\cap\dom\lambda'$ with $\germ_{\tx_i}\lambda=\germ_{\tx_i}\lambda'$. This defines an
equivalence relation (it is reflexive by (\hyperref[M5d]{M5d})). The equivalence class of
$\atlasmaprep{f}{}$ will be denoted by

\begin{equation}\label{eq-41}
\atlasmap{f}{}=\left(f,\overf,\left\{\tf_i\right\}_{i\in I},\left[P_f,\nu_f\right]\right):\,
\atlas{X}\longrightarrow\atlas{Y}
\end{equation}
and it is called a \emph{morphism of reduced orbifold atlases from $\atlas{X}$ to $\atlas{Y}$
over the continuous map $f:X\rightarrow Y$}.
\end{defin}

\begin{lem}
\hspace{0.1em}
\begin{enumerate}[\emphatic{(}i\emphatic{)}]
 \item given any reduced orbifold chart $(\tX,G,\pi)$ on any topological space and any change of charts
 $\lambda$ from $(\tX,G,\pi)$ to itself, there is a unique $g\in G$ such that
  $\lambda=g|_{\dom\lambda}$;
 \item let us fix any pair of reduced orbifold atlases $\atlas{X}:=\{\unifX{i}\}_{i\in I}$ and
  $\atlas{Y}:=\{\unifY{j}\}_{j\in J}$ and any morphism as in \eqref{eq-41}; for each
  $i\in I$ there is a group homomorphism $\dot{f}_i:G_i\rightarrow H_{\overf(i)}$, such
  that for each $\lambda\in P_f(i,i)$ we have that $\nu_f(\lambda)=\dot{f}_i(g_i)|_{\dom\nu_f
  (\lambda)}$, where $g_i$ is the unique element of $G_i$ as in \emphatic{(}i\emphatic{)};
 \item for each $i\in I$ and for each $g_i\in G_i$ we have $\tf_i\circ g_i=\dot{f}_i(g_i)\circ\tf_i$.
\end{enumerate}
\end{lem}

\begin{proof}
Since $\dom\lambda$ is connected, claim (i) is a straightforward consequence
of~\cite[Proposition~A.1]{MP}.
Now let us fix any representative $(P_f,\nu_f)$ for $[P_f,\nu_f]$, any $g_i
\in G_i$ and any point $\tx_i\in\tX_i$; since $P_f$ satisfies condition (\hyperref[M4]{M4}), then
there exists a (in general non-unique) $\lambda\in P_f(i,i)$ such that
$\tx_i\in\dom\lambda$ and $\lambda=g_i|_{\dom\lambda}$ (a priori, the second condition holds only in
a neighborhood $\tX'$ of $\tx_i$ in $\dom\lambda$; by (i) we conclude that the same relation holds
everywhere on $\dom\lambda$). By (i) applied on $\atlas{Y}$, we get that $\nu_f(\lambda)$ is the
restriction of a unique object $\dot{f}_i(g_i,\tx_i,\lambda)$ in $H_{\overline{f}(i)}$.\\

We claim that $\dot{f}_i(g_i,\tx_i,\lambda)$ does not depend on $\tx_i$ or $\lambda$
(but only on $g_i$ and on $(P_f,\nu_f)$). So let us fix
another point $\tx'_i\in\tX_i$ and another $\lambda'\in P_f$ such that $\tx'_i\in\dom\lambda'$
and $\lambda'=g_i|_{\dom\lambda'}$. Since $\tX_i$ is connected by definition of reduced orbifold
chart, then there exists a continuous path $\gamma:[0,1]\rightarrow\tX_i$, which joins $\tx_i$ and
$\tx'_i$. For any $t\in\,]0,1[\,$, we choose $\lambda_t\in P_f$ such that $\gamma(t)\in\dom\lambda_t$
and $\lambda_t=g_i|_{\dom\lambda_t}$. By compactness, we can cover $\gamma([0,1])$ by a finite
number of open sets $\{\dom\lambda_{t^l}\}_{l=1,\cdots,r}$ and we can also assume that $\dom
\lambda_{t^l}$ intersects $\dom\lambda_{t^{l+1}}$ for each $l=1,\cdots,r-1$. For each $l$,
using (i) we get that $\nu_f(\lambda_{t^l})$ is the restriction of a unique object $\dot{f}(g_i,
\gamma(t^l),\lambda_{t^l})$ in $H_{\overf(i)}$. This proves that for each $l=1,\cdots,r-1$ we have $\dot{f}_i(g_i,
\gamma(t^l),\lambda_{t^l})=\dot{f}(g_i,\gamma(t^{l+1}),\lambda_{t^{l+1}})$, which implies that
$\dot{f}_i(g_i,\tx_i,\lambda)=\dot{f}_i(g_i,\tx'_i,\lambda')$. Therefore, we have proved that
$\dot{f}$ is a well-defined set map from $G_i$ to $H_{\overline{f}(i)}$, which depends
only on $(P_f,\nu_f)$. The fact that $\dot{f}_i$ is
a group homomorphism is a simple consequence of conditions (\hyperref[M5e]{M5e}) and
(\hyperref[M5f]{M5f}). Using \eqref{eq-26},
$\dot{f}_i$ does not depend on the representative $(P_f,\nu_f)$ chosen for $[P_f,\nu_f]$. Claim (iii)
is a direct consequence of (\hyperref[M5c]{M5c}) and (ii).
\end{proof}

\begin{defin}\label{def-04}
Let us fix any reduced orbifold atlas $\atlas{X}:=\{\unifX{i}\}_{i\in I}$ on any topological
space $X$ and any set $\{\unifX{i'}\}_{i'\in I'}$ of reduced orbifold charts on $X$, that are
compatible with the charts of $\atlas{X}$ (with the set $I'$ disjoint from $I$). Then the
family

\[\atlas{X}':=\left\{\Big(\tX_i,G_i,\pi_i\Big)\right\}_{i\in I}\coprod\left\{
\Big(\tX_{i'},G_{i'},\pi_{i'}\Big)\right\}_{i'\in I'}\]
is again a reduced orbifold atlas for $X$. Moreover, there is an obvious inclusion
$\nu_{\atlas{X},\atlas{X}'}:\change(\atlas{X})\hookrightarrow\change(\atlas{X}')$.
Therefore we can consider the following morphism:

\[\iota_{\atlas{X},\atlas{X}'}:=\left(\id_X,I\hookrightarrow(I\amalg I'),\left
\{\id_{\tX_i}\right\}_{i\in I},[\change(\atlas{X}),\nu_{\atlas{X},\atlas{X}'}]
\right):\,\atlas{X}\,{\ensuremath{\lhook\joinrel\relbar\joinrel\rightarrow}}\,\atlas{X}'.\]

We call any such morphism an \emph{inclusion} of reduced orbifold atlases. In particular,
for every reduced orbifold atlas $\atlas{X}$ as before, we denote by

\[\iota_{\atlas{X}}:\,\atlas{X}\,{\ensuremath{\lhook\joinrel\relbar\joinrel\rightarrow}}\,
\atlas{X}^{\max}\]
the inclusion $\iota_{\atlas{X},\atlas{X}^{\max}}$ of $\atlas{X}$ into the maximal atlas
$\atlas{X}^{\max}$ associated to it.
\end{defin}

Now we need to define the composition of morphisms of reduced orbifold atlases. In order to do that, we
follow~\cite[Construction~5.9]{Po}, with the only differences due to Remark~\ref{rem-02}.

\begin{construction}\label{cons-01}
Let us fix any triple of orbifold atlases

\[\atlas{X}=\left\{\Big(\tX_i,G_i,\pi_i\Big)\right\}_{i\in I},\quad\atlas{Y}:=\left\{
\Big(\tY_i,H_j,\chi_j\Big)\right\}_{j\in J},
\quad\atlas{Z}=\left\{\Big(\tZ_l,K_l,\eta_l\Big)\right\}_{l\in L}\]
for $3$ topological spaces $X$, $Y$ and $Z$ respectively. Let us also fix $2$ morphisms 

\begin{gather}
\nonumber \atlasmap{f}{}=\left(f,\overf,\Big\{\tf_i\Big\}_{i\in I},\left[P_f,\nu_f\right]\right):
 \,\atlas{X}\longrightarrow\atlas{Y}, \\
\label{eq-45} \atlasmap{g}{}=\left(g,\overg,\Big\{\tg_j\Big\}_{j\in J},\left[P_g,\nu_g\right]
 \right):\,\atlas{Y}\longrightarrow \atlas{Z}.  
\end{gather}

Then we define a composition

\[\atlasmap{g}{}\circ\atlasmap{f}{}:=\left(g\circ f,\overg\circ\overf,\left\{
\tg_{\overf(i)}\circ\tf_i\right\}_{i\in I},\left[P_{g\circ f},\nu_{g\circ f}\right]\right):\,\atlas{X}\longrightarrow
\atlas{Z}.\]

Here we construct the class $[P_{g\circ f},\nu_{g\circ f}]$ as follows: first of all, let us fix representatives
$(P_f,\nu_f)$ for $[P_f,\nu_f]$ and $(P_g,\nu_g)$ for $[P_g,\nu_g]$. Then let us fix any $i\in I$,
any $\lambda\in P_f(i,-)$ and any point $\overline{x}_i\in\dom\lambda$. Since $P_g$ is a good
subset of $\change(\atlas{Y})$, then by condition
(\hyperref[GS]{GS}) there are a (non-unique) $\omega_{\tf_i
(\overline{x}_i),\nu_f(\lambda)}\in P_g(\overf(i),-)$ and an open set

\[\tY_{\tf_i(\overline{x}_i),\nu_f(\lambda)}\subseteq\dom\nu_f(\lambda)\cap\dom\omega_{\tf_i
(\overline{x}_i),\nu_f(\lambda)}\subseteq\tY_{\overf(i)},\]
such that $\tf_i(\overline{x}_i)\in\tY_{\tf_i(\overline{x}_i),\nu_f(\lambda)}$ and 

\[\left.\Big(\nu_f(\lambda)\Big)\right|_{\tY_{\tf_i(\overline{x}_i),\nu_f(\lambda)}}=\left.\left(
\omega_{\tf_i(\overline{x}_i),\nu_f(\lambda)}\right)\right|_{\tY_{\tf_i(\overline{x}_i),\nu_f(
\lambda)}}.\]

For each pair $(\lambda,\overline{x}_i)$ as before, there exists a (non-unique) open connected
subset $\tX_{\overline{x}_i,\lambda}\subseteq\tf_i^{-1}
\left(\tY_{\tf_i(\overline{x}_i),\nu_f(\lambda)}
\right)\cap\dom\lambda$, such that:

\begin{itemize}
 \item $\overline{x}_i\in\tX_{\overline{x}_i,\lambda}$;
 \item $\tX_{\overline{x}_i,\lambda}$ is invariant under the action of $\Stab(G_i,\overline{x}_i)$;
 \item for each $g_i\in G_i\smallsetminus\Stab(G_i,\overline{x}_i)$ we have $g_i
  (\tX_{\overline{x}_i,\lambda})\cap\tX_{\overline{x}_i,\lambda}=\varnothing$
\end{itemize}
(in this way, $\lambda$ is still a change of charts if restricted to $\tX_{\overline{x}_i,
\lambda}$). Then for each $i\in I$ and for each $\lambda\in P_f(i,-)$ we choose any set of points
$\{\overline{x}^e_i\}_{e\in E(\lambda)}\subseteq\tX_i$ such that the set $\{\tX_{\overline{x}_i^e,
\lambda}\}_{e\in E(\lambda)}$ is a covering for $\dom\lambda$ and such that if $e\neq e'$, then
$\tX_{\overline{x}_i^e,\lambda}\neq\tX_{\overline{x}_i^{e'},\lambda}$. Then we consider the set:

\[P_{g\circ f}:=\Big\{\lambda|_{\tX_{\overline{x}_i^e,\lambda}}\quad\forall\,i\in I,\,\,\,\forall\,\lambda\in
P_f(i,-),\,\,\,\forall\,e\in E(\lambda)\Big\}.\]

In general, given any element $\lambda'\in P_{g\circ f}$, there can be more than one $\lambda\in P_f$,
such that $\lambda|_{\tX_{\overline{x}_i^e,\lambda}}=\lambda'$ (for some $e\in E(\lambda)$);
therefore for any such $\lambda'\in P_f$, using the axiom of choice
we make an arbitrary choice of $(\lambda,e)$
with that property. For that choice, we fix also a choice of $\tY_{\tf_i(\overline{x}^e_i),
\nu_f(\lambda)}$ and of $\omega_{\tf_i(\overline{x}^e_i),\nu_f(\lambda)}$ as before.\\

Since $P_f$ is a good subset of $\change(\atlas{X})$, then also $P_{g\circ f}$ is a good subset of
$\change(\atlas{X})$. Then for each $\lambda|_{\tX_{\overline{x}_i^e,\lambda}}\in P_{g\circ f}$ we set

\[\nu_f^{\ind}\left(\lambda|_{\tX_{\overline{x}_i^e,\lambda}}\right):=\omega_{\tf_i(
\overline{x}_i^e),\nu_f(\lambda)}.\]

So we have defined a set map $\nu_f$ from $P_{g\circ f}$ to $\change(\atlas{Y})$; a direct computation
proves that $(P_{g\circ f},\nu_f^{\ind})\in[P_f,\nu_f]$. Then we simply define

\[\nu_{g\circ f}\Big(\lambda|_{\tX_{\overline{x}_i^e,\lambda}}\Big):=\nu_g\Big(\omega_{\tf_i(
\overline{x}_i^e),\nu_f(\lambda)}\Big)=\nu_g\circ\nu_f^{\ind}\Big(
\lambda|_{\tX_{\overline{x}_i^e,\lambda}}\Big)\]
for every $\lambda|_{\tX_{\overline{x}_i^e,\lambda}}\in P_{g\circ f}$ and it is easy to verify that $\nu_{g\circ f}$
satisfies properties (\hyperref[M5a]{M5a}) -- (\hyperref[M5d]{M5d}). The construction of $P_{g\circ f}$ and
$\nu_{g\circ f}$ depends on some choices, but it can be proved that the equivalence class $[P_{g\circ f},
\nu_{g\circ f}]$ does
not depend on such choices. In this way we have defined a notion of composition of morphisms of
reduced orbifold atlases.
\end{construction}

\begin{lem}
The composition of morphisms of reduced orbifold atlases is associative.
\end{lem}

The proof is obvious for what concerns the composition of maps of the form $f,\overf$ and $\tf_i$;
the proof of the associativity on the pairs of the form $[P_f,\nu_f]$ is straightforward, so we
omit it.

\begin{rem}
Since reduced orbifold structures are given by equivalence classes of reduced orbifold atlases
(see Definition~\ref{def-02}), then one would like to define compositions also for
every pair of morphisms $\atlasmap{f}{}:\atlas{X}\rightarrow\atlas{Y}$ and $\atlasmap{g}{}:\atlas{Y}'
\rightarrow\atlas{Z}$, whenever $[\atlas{Y}]=[\atlas{Y}']$. We will provide such a
definition in the paper~\cite{T10} (at the cost of quotienting out by an equivalence relation
induced by the $2$-morphisms that we are going to define below). As we mentioned in the
Introduction, currently
in the literature there are also other $2$ definitions
of morphisms between orbifolds. One is given by ``good maps'' defined by Weimin Chen and Yongbin
Ruan in~\cite{CR} (such a definition is given also for the (possibly) non-reduced case); the other one
is given by Anke Pohl in~\cite{Po}. We will describe in~\cite{T10} the relations between our
definition of morphism and those of~\cite{CR} and~\cite{Po}.
\end{rem}

\begin{defin}\label{def-13}
Let us fix any pair of paracompact, second countable, Hausdorff topological spaces $X$ and $Y$,
any open embedding $f:X\hookrightarrow Y$ and any reduced orbifold
atlas $\atlas{X}=\{\unifX{i}\}_{i\in I}$ on $X$. Then we set

\begin{equation}\label{eq-30}
f_{\ast}(\atlas{X}):=\left\{\left(\tX_i,G_i,f\circ\pi_i\right)\right\}_{i\in I}.
\end{equation}

Since $f$ is an open embedding, then $f_{\ast}(\atlas{X})$ is a family of compatible reduced
orbifold charts over $Y$, hence $f_{\ast}(\atlas{X})$ is a reduced orbifold atlas over $f(X)$
(with the topology given by the fact that $f(X)$ is open in $Y$). Moreover, $f$ induces
a morphism $[\hat{f}^{\ind}]:\atlas{X}\rightarrow f_{\ast}(\atlas{X})$: on the topological level it is
simply
given by $f$, while the rest of the structure is trivial: $\overline{f}$ is the identity of $I$,
each $\tf_i$ is an identity, $P_f$ is the entire set $\change(\atlas{X})$ and
$\nu_f$ associates to any change of charts in $\atlas{X}$ the same change of charts in
$f_{\ast}(\atlas{X})$. In particular, if $f$ is an homeomorphism, then $f_{\ast}(\atlas{X})$
is a reduced orbifold atlas on $Y$ and $[\hat{f}^{\ind}]$
is an isomorphism (with respect to the definition of composition given above).
\end{defin}

As we said in the introduction, our first aim is to construct a $2$-category $\RedAtl$
of reduced orbifold
atlases. Roughly speaking, a $2$-category is the datum of objects, morphisms and ``morphism between
morphisms'' (known as $2$-\emph{morphisms} or, sometimes, \emph{natural transformations}),
together with identities and compositions of morphisms
and $2$-morphisms (for more details we refer e.g.\ to~\cite{Lei}). First of all, we have to 
define a notion of $2$-morphism in this setup.

\begin{defin}\label{def-05}
Let us fix any pair of reduced orbifold atlases $\atlas{X}=\{\unifX{i}\}_{i\in I}$
and $\atlas{Y}:=\{\unifY{j}\}_{j\in J}$ over $X$ and $Y$
respectively. Moreover, let us fix $2$ morphisms from $\atlas{X}$ to $\atlas{Y}$ over
\emph{the same} continuous function $f:X\rightarrow Y$:

\[\atlasmap{f}{m}:=\left(f,\overf^m,\left\{\tf^m_i\right\}_{i\in I},\left[P_{f^m},\nu_{f^m}\right]
\right)\quad\textrm{for}\,\,m=1,2.\]

Then a \emph{representative of a $2$-morphism} from $\atlasmap{f}{1}$ to $\atlasmap{f}{2}$ is any set
of data:

\[\delta:=\left\{\left(\tX_i^a,\delta_i^a\right)\right\}_{i\in I, a\in A(i)},\]
such that:

\begin{enumerate}[({2M}a)] 
 \item\label{2Ma} for all $i\in I$ the set $\{\tX_i^a\}_{a\in A(i)}$ is an open covering of $\tX_i$;
 \item\label{2Mb} for all $i\in I$ and for all $a\in A(i)$, $\delta_i^a$ is a change of charts in
  $\atlas{Y}$ with

  \[\tf^1_i\left(\tX_i^a\right)\subseteq\dom\delta_i^a\subseteq\tY_{\overf^1(i)},\quad\tf^2_i
  \left(\tX_i^a\right)\subseteq\cod\delta_i^a\subseteq\tY_{\overf^2(i)};\]

 \item\label{2Mc} for all $i\in I$, for all $a\in A(i)$ and for all $\tx_i\in\tX_i^a$ we have
  
  \begin{equation}\label{eq-01}
  \tf_i^2(\tx_i)=\delta_i^a\circ\tf_i^1(\tx_i);
  \end{equation}
 
 \item\label{2Md} for all $i\in I$, for all $a,a'\in A(i)$ and for all $\tx_i\in\tX_i^a\cap\tX_i^{a'}$
  we have

  \[\germ_{\tf_i^1(\tx_i)}\delta_i^a=\germ_{\tf_i^1(\tx_i)}\delta_i^{a'};\]
 \item\label{2Me} for all $(i,i')\in I\times I$, for all $(a,a')\in A(i)\times A(i')$, for all $\lambda
  \in\change(\atlas{X},i,i')$ and for all $\tx_i\in\dom\lambda\cap\tX_i^a$ such that $\lambda(\tx_i)
  \in\tX_{i'}^{a'}$, there exist
 
  \begin{equation}\label{eq-02}
  \left(P_{f^m},\nu_{f^m}\right)\in\left[P_{f^m},\nu_{f^m}\right],\quad\lambda^m\in P_{f^m}(i,i')\quad
  \textrm{for}\,\,m=1,2,
  \end{equation}
  such that

  \begin{equation}\label{eq-03}
  \tx_i\in\dom\lambda^m,\quad\germ_{\tx_i}\lambda^m=\germ_{\tx_i}\lambda\quad\textrm{for~}m=1,2
  \end{equation}
  and

  \begin{equation}\label{eq-04}
  \germ_{\tf_i^2(\tx_i)}\nu_{f^2}(\lambda^2)\cdot\germ_{\tf_i^1(\tx_i)}\delta_i^a=\germ_{\tf_{i'}^1
  (\lambda(\tx_i))}\delta_{i'}^{a'}\cdot\germ_{\tf_i^1(\tx_i)}\nu_{f^1}(\lambda^1).
  \end{equation}
\end{enumerate}
\end{defin}

\begin{rem}
Given any $i\in I$ and any pair $a,a'\in A(i)$, by
(\hyperref[2Mc]{2Mc}) we get that $\delta_i^a$ coincides with $\delta_i^{a'}$ on the
set $\tf_i^1(\tX_i^a\cap\tX_i^{a'})$, that in general is not an open set; actually, by
(\hyperref[2Md]{2Md}) $\delta_i^a$ and $\delta_i^{a'}$
coincide on some open set containing $\tf_i^1(\tX_i^a\cap\tX_i^{a'})$.
We remark also that
both the left hand side and the right hand side of \eqref{eq-04} are well-defined. Indeed,

\[\delta_i^a\circ\tf_i^1(\tx_i)=\tf_i^2(\tx_i)\]
by \eqref{eq-01} and

\[\nu_{f^1}(\lambda^1)\Big(\tf_i^1(\tx_i)\Big)=\tf_{i'}^1\Big(\lambda^1(\tx_i)\Big)=\tf_{i'}^1\Big(
\lambda(\tx_i)\Big)\]
by (\hyperref[M5c]{M5c}) and \eqref{eq-03}.
\end{rem}

\begin{rem}
Let us suppose that there exist data as in \eqref{eq-02} that satisfy conditions \eqref{eq-03} and
\eqref{eq-04}. Let us suppose that $(P'_{f^m},\nu'_{f^m},\lambda^{\prime m})$ for
$m=1,2$ is another set of data as \eqref{eq-02} that satisfies condition \eqref{eq-03}. Then by
Definition~\ref{def-03} we conclude that 

\[\germ_{\tf_i^m(\tx_i)}\nu_{f^m}(\lambda^m)=\textrm{germ}_{\tf_i^m(\tx_i)}\nu'_{f^m}
(\lambda^{\prime m})\quad\textrm{for}\,\,m=1,2,\]
so \eqref{eq-04} is verified also by the new set of data. Therefore, (\hyperref[2Me]{2Me}) is
equivalent to:

\begin{enumerate}[({2M}a)$'$]
\setcounter{enumi}{4}
 \item\label{2Meprime} for all $(i,i')\in I\times I$, for all $(a,a')\in A(i)\times A(i')$, for all
  $\lambda\in\change(\atlas{X},i,i')$, for all $\tx_i\in\dom\lambda\cap\tX_i^a$ such that
  $\lambda(\tx_i)\in\tX_{i'}^{a'}$ and for all data \eqref{eq-02} that satisfy \eqref{eq-03}, we have
  that \eqref{eq-04} holds.
\end{enumerate}
\end{rem}

\begin{defin}\label{def-06}
Let us fix any pair of reduced orbifold atlases $\atlas{X},\atlas{Y}$ and any pair of morphisms
$\atlasmap{f}{1}$, $\atlasmap{f}{2}$ from $\atlas{X}$ to $\atlas{Y}$ over the same continuous map
(as in Definition~\ref{def-05}). Moreover, let us fix any pair of representatives of
$2$-morphisms from $\atlasmap{f}{1}$ to $\atlasmap{f}{2}$:

\[\delta:=\left\{\left(\tX_i^a,\delta_i^a\right)\right\}_{i\in I, a\in A(i)}\quad\textrm{and}\quad
\overline{\delta}:=\left\{\left(\tX_i^{\overline{a}},\overline{\delta}_i^{\overline{a}}\right)
\right\}_{i\in I, \overline{a}\in\overline{A}(i)}.\]

Then we say that $\delta$ is \emph{equivalent} to $\overline{\delta}$ if and only if for all $i\in
I$, for all pairs $(a,\overline{a})\in A(i)\times\overline{A}(i)$ and for all $\tx_i\in\tX_i^a\cap
\tX_i^{\overline{a}}$ (if non-empty) we have

\[\germ_{\tf_i^1(\tx_i)}\delta_i^a=\germ_{\tf_i^1(\tx_i)}\overline{\delta}_i^{\overline{a}}.\]

This definition gives rise to an equivalence relation (it is reflexive by (\hyperref[2Md]{2Md})). We
denote by $[\delta]:\atlasmap{f}{1}\Rightarrow\atlasmap{f}{2}$ the class of any $\delta$ as before 
and we say that $[\delta]$ is a \emph{$2$-morphism} from $\atlasmap{f}{1}$ to $\atlasmap{f}{2}$. 
\end{defin}

\section{Vertical and horizontal compositions of 2-morphisms}
\begin{construction}\label{cons-02}
Let us fix $2$ reduced orbifold atlases $\atlas{X}=\{\unifX{i}\}_{i\in I}$, $\atlas{Y}=
\{\unifY{j}\}_{j\in J}$ for $X$ and $Y$ respectively, any continuous map $f:X\rightarrow Y$ and
any triple of morphisms from $\atlas{X}$ to $\atlas{Y}$ over $f$:

\[\atlasmap{f}{m}:=\left(f,\overf^m,\left\{\tf_i^m\right\}_{i\in I},\left[P_{f^m},\nu_{f^m}\right]
\right)\quad\textrm{for}\quad m=1,2,3.\]

In addition, let us fix any $2$-morphism $[\delta]:\atlasmap{f}{1}\Rightarrow\atlasmap{f}{2}$ and any
$2$-morphism $[\sigma]:\atlasmap{f}{2}\Rightarrow\atlasmap{f}{3}$. We want to define a \emph{vertical
composition} $[\sigma]\odot[\delta]:\atlasmap{f}{1}\Rightarrow\atlasmap{f}{3}$; in order to do that,
let us fix any representative

\[\sigma=\left\{\left(\tX_i^b,\sigma_i^b\right)\right\}_{i\in I,\,b\in B(i)}\]
for $[\sigma]$. As we did in Construction~\ref{cons-01}, we can always choose a (non-unique)
representative 

\[\delta=\left\{\left(\tX_i^a,\delta_i^a\right)\right\}_{i\in I,a\in A(i)}\]
for $[\delta]$, such that for each $i\in I$ and for each $(a,b)\in A(i)\times B(i)$, the set map
$\delta_i^a$ restricted to the set

\begin{equation}\label{eq-05}
\tY_i^{a,b}:=\left(\delta_i^a\right)^{-1}\Big(\cod\delta_i^a\cap\dom\sigma_i^b\Big) 
\end{equation}
(if non-empty) is again a change of charts, so that also $\theta_i^{a,b}:=\sigma_i^b\circ
\delta_i^a|_{\tY_i^{a,b}}$ is a change of charts of $\atlas{Y}$. Then for each $i\in I$ and for
each $(a,b)\in A(i)\times B(i)$ we set $\tX_i^{a,b}:=\tX_i^a\cap\tX_i^b$; if $\tX_i^{a,b}$ is
non-empty, then also $\tY_i^{a,b}$ is non-empty. Then we define:

\[\theta:=\left\{\left(\tX_i^{a,b},\theta_i^{a,b}\right)\right\}_{i\in I,\,(a,b)\in A(i)\times B(i)
\textrm{~s.t.~}\tX_i^{a,b}\neq\varnothing}.\]
\end{construction}

A straightforward proof shows that:

\begin{lem}
The collection $\theta$ so defined is a representative of a $2$-morphism from $\atlasmap{f}{1}$ to
$\atlasmap{f}{3}$. Moreover, the class of $\theta$ does not depend on the choices of representatives
$\delta$ for $[\delta]$ and $\sigma$ for $[\sigma]$. 
\end{lem}

Therefore, it makes sense to give the following definition.

\begin{defin}
Given any pair $[\delta]$, $[\sigma]$ as before, we define their \emph{vertical composition} as:

\[[\sigma]\odot[\delta]:=[\theta]:\atlasmap{f}{1}\Longrightarrow\atlasmap{f}{3}.\]
\end{defin}

\begin{construction}\label{cons-03}
Let us fix any triple of reduced orbifold atlases

\[\atlas{X}=\left\{\Big(\tX_i,G_i,\pi_i\Big)\right\}_{i\in I},\quad\atlas{Y}:=\left\{
\Big(\tY_j,H_j,\chi_j\Big)\right\}_{j\in J},
\quad\atlas{Z}=\left\{\Big(\tZ_l,K_l,\eta_l\Big)\right\}_{l\in L}\]
for $X$, $Y$ and $Z$ respectively. Let us also fix any set of morphisms

\begin{gather*}
\atlasmap{f}{m}:=\left(f,\overf^m,\Big\{\tf_i^m\Big\}_{i\in I},\left[P_{f^m},\nu_{f^m}\right]\right):
 \,\atlas{X}\longrightarrow\atlas{Y}\quad\textrm{for}\quad m=1,2, \\
\atlasmap{g}{m}:=\left(g,\overg^m,\Big\{\tg_j^m\Big\}_{j\in J},\Big[P_{g^m},\nu_{g^m}\Big]\right):\,
 \atlas{Y}\longrightarrow\atlas{Z}\quad\textrm{for}\quad m=1,2. 
\end{gather*}

Moreover, let us suppose that we have fixed any $2$-morphism $[\delta]:\atlasmap{f}{1}\Rightarrow
\atlasmap{f}{2}$ and any $2$-morphism $[\xi]:\atlasmap{g}{1}\Rightarrow\atlasmap{g}{2}$. Our aim is
to define
an \emph{horizontal} composition $[\xi]\ast[\delta]:\atlasmap{g}{1}\circ\atlasmap{f}{1}\Rightarrow
\atlasmap{g}{2}\circ\atlasmap{f}{2}$. In order to do that, we fix any representative $(P_{g^1},
\nu_{g^1})$ for $[P_{g^1},\nu_{g^1}]$ and any representative

\[\overline{\delta}:=\left\{\left(\tX_i^{\overline{a}},\overline{\delta}_i^{\overline{a}}\right)
\right\}_{i\in I,\,\overline{a}\in\overline{A}(i)}\]
for $[\delta]$. For any $i\in I$ and any $\overline{a}\in\overline{A}(i)$ we have that
$\overline{\delta}_i^{\overline{a}}\in\change(\atlas{Y})$. Since $P_{g^1}$ satisfies condition
(\hyperref[GS]{GS}), then as in the previous constructions we can use $\overline{\delta}$ in
order to get another representative

\[\delta:=\left\{\left(\tX_i^a,\delta_i^a\right)\right\}_{i\in I,\,a\in A(i)}\]
for $[\delta]$, such that for each $i\in I$ and $a\in A(i)$, the change of charts $\delta_i^a$ is
the restriction of a change of charts $\widetilde{\delta}_i^a\in P_{g^1}$ (in general, such a change
of charts is not unique, we fix any arbitrary choice of $\widetilde{\delta}_i^a$'s with this
property). We choose also any representative $\xi:=\{(\tY_j^c,\xi_j^c)\}_{j\in J,c\in C(j)}$ for
$[\xi]$. Let us fix any $i\in I,a\in A(i),c\in C(\overf^2(i))$ and any point $\overline{x}_i\in
\tX_i$ such that the point $\overline{z}_i:=\tg_{\overf^1(i)}^1\circ\tf_i^1(\overline{x}_i)$
belongs to the set

\[\tZ_i^{a,c}:=\left(\nu_g^1(\delta_i^a)\right)^{-1}\left(\cod\nu_g^1(\delta_i^a)\cap\dom
\xi_{\overf^2(i)}^c\right).\]

Then there exists a (non-unique) open connected subset $\tZ_i^{a,c,\overline{x}_i}\subseteq
\tZ_i^{a,c}$, such that:

\begin{itemize}
 \item $\overline{z}_i\in\tZ_i^{a,c,\overline{x}_i}$;
 \item $\tZ_i^{a,c,\overline{x}_i}$ is invariant under the action of
  $\Stab(H_{\overline{g}^1\circ\overline{f}^1(i)},\overline{z}_i)$;
 \item for all $h\in H_{\overline{g}^1\circ\overline{f}^1(i)}\smallsetminus
  \Stab(H_{\overline{g}^1\circ\overline{f}^1(i)},\overline{z}_i)$
  we have $h(\tZ_i^{a,c,\overline{x}_i})\cap\tZ_i^{a,c,\overline{x}_i}=\varnothing$
\end{itemize}
(in this way, $\nu_g^1(\widetilde{\delta}_i^a)$ is again a change of charts if restricted to
$\tZ_i^{a,c,\overline{x}_i}$). We define also

\begin{gather*}
\tX_i^{a,c}:=\tX_i^a\cap\left(\tg^1_{\overf^1(i)}\circ\tf_i^1\right)^{-1}\Big(\tZ_i^{a,c}
 \Big), \\
\tX_i^{a,c,\overline{x}_i}:=\tX_i^a\cap\left(\tg^1_{\overf^1(i)}\circ\tf_i^1\right)^{-1}\Big(
 \tZ_i^{a,c,\overline{x}_i}\Big).
\end{gather*}

For each $(i,a,c)$ as before, we choose any set of points $\left\{\overline{x}_i^e\right\}_{e\in
E(i,a,c)}$ such that the corresponding sets of the form $\tX_i^{a,c,\overline{x}^e_i}$ are a covering
for $\tX_i^{a,c}$ (if $\tX_i^{a,c}$ is empty, we set $E(i,a,c):=\varnothing$). For simplicity, we rename
each $\tX_i^{a,c,\overline{x}^e_i}$ as $\tX_i^{a,c,e}$ and analogously for
the sets of the form $\tZ_i^{a,c,\overline{x}^e_i}$. Then for each $i\in I$, $(a,c)
\in A(i)\times C(\overf^2(i))$ and $e\in E(i,a,c)$ we define

\[\gamma_i^{a,c,e}:=\left.\xi_{\overf^2(i)}^c\circ\nu^1_g(\widetilde{\delta}_i^a)
\right|_{\tZ_i^{a,c,e}}\]
and we set

\[\gamma:=\left\{\left(\tX_i^{a,c,e},\gamma_i^{a,c,e}\right)\right\}_{i\in I,\,(a,c)\in A(i)\times
C(\overf^2(i)),\,e\in E(i,a,c)}.\]
\end{construction}

A direct check proves that:

\begin{lem}
The collection $\gamma$ so defined is a representative of a $2$-morphism from $\atlasmap{g}{1}\circ
\atlasmap{f}{1}$ to $\atlasmap{g}{2}\circ\atlasmap{f}{2}$. Moreover, the class of $\gamma$ does not
depend on the representatives $(P_{g^1},\nu_{g^1})$, $\delta$ and $\xi$ chosen for $[P_g^1,\nu_g^1]$,
$[\delta]$ and $[\xi]$ respectively.
\end{lem}

So it makes sense to give the following definition.

\begin{defin}
Given any pair $[\delta]$, $[\xi]$ as before, we define their \emph{horizontal composition} as:

\[[\xi]\ast[\delta]:=[\gamma]:\atlasmap{g}{1}\circ\atlasmap{f}{1}\Longrightarrow\atlasmap{g}{2}\circ
\atlasmap{f}{2}.\]
\end{defin}

\section{The 2-category $\RedAtl$}
\begin{defin}
Given any reduced orbifold atlas $\atlas{X}=\{\unifX{i}\}_{i\in I}$ on a topological space $X$,
we define the \emph{identity} of $\atlas{X}$ as the morphism

\[\id_{\atlas{X}}:=\left(\id_{X},\id_I,\left\{\id_{\tX_i}\right\}_{i\in I},\left[\change
(\atlas{X}),\nu_{\id}\right]\right):\atlas{X}\longrightarrow\atlas{X}\]
where $\nu_{\id}$ is the identity on $\change(\atlas{X})$ (this is a special case of
Definition~\ref{def-04}). Given any pair of reduced orbifold
atlases $\atlas{X}$ and $\atlas{Y}$ and any morphism $\atlasmap{f}{}=(f,\overf,\{\tf_i\}_{i\in I},
[P_f,\nu_f])$ from $\atlas{X}$ to $\atlas{Y}$, we define the \emph{$2$-identity}
$i_{\atlasmap{f}{}}$ as the class of

\[\left\{\left(\tX_i,\id_{\tX_{\overf(i)}}\right)\right\}_{i\in I}.\]

Moreover, for each reduced orbifold atlas $\atlas{X}$, we set $i_{\atlas{X}}:=
i_{\id_{\atlas{X}}}$.
\end{defin}

A direct check proves that:

\begin{lem}\label{lem-01}
\hspace{0.1em}
\begin{enumerate}[\emphatic{(}a\emphatic{)}]
 \item The morphisms and $2$-morphisms of the form $\id_{-}$ and $i_{-}$ are the identities with respect
  to $\circ$ and $\odot$ respectively. Moreover, any $2$-morphism is invertible with respect
  to $\odot$. 
 \item Let us fix any pair of morphisms of reduced orbifolds $\atlasmap{f}{1},\atlasmap{f}{2}:
  \atlas{X}\rightarrow\atlas{Y}$ and any $2$-morphism $[\delta]:\iota_{\atlas{Y}}\circ\atlasmap{f}{1}
  \Rightarrow\iota_{\atlas{Y}}\circ\atlasmap{f}{2}$ \emphatic{(}where $\iota_{\atlas{Y}}$ is the
  inclusion $\atlas{Y}\hookrightarrow\atlas{Y}^{\max}$, see \emphatic{Definition~\ref{def-04})}.
  Then there is a unique $2$-morphism $[\delta']:\atlasmap{f}{1}\Rightarrow\atlasmap{f}{2}$ such
  that $[\delta]=i_{\iota_{\atlas{Y}}}\ast[\delta']$.
 \end{enumerate}
\end{lem}

The following proof is long but completely straightforward, so we omit it.

\begin{lem}\label{lem-02}
Given any diagram as follows

\[\begin{tikzpicture}[scale=0.8]
  \def\x{2.0}
  \def\y{-1.2}
  \node (A0_0) at (0*\x, 0*\y) {$\atlas{X}$};
  \node (B0_1) at (0.6*\x, -0.3) {$\scriptstyle{\atlasmap{f}{2}}$};
  \node (B0_2) at (2.6*\x, -0.3) {$\scriptstyle{\atlasmap{g}{2}}$};
  \node (A0_2) at (2*\x, 0*\y) {$\atlas{Y}$};
  \node (A0_4) at (4*\x, 0*\y) {$\atlas{Z}$,};
  
  \node (A0_1) at (1.1*\x, 0.6) {$\Downarrow\,[\delta]$};
  \node (A0_1) at (1.1*\x, -0.6) {$\Downarrow\,[\sigma]$};
  \node (A0_3) at (3.1*\x, 0.6) {$\Downarrow[\xi]$};
  \node (A0_3) at (3.1*\x, -0.6) {$\Downarrow[\eta]$};
  
  \path (A0_0) edge [->,bend left=70] node [auto] {$\scriptstyle{\atlasmap{f}{1}}$} (A0_2);
  \path (A0_0) edge [->] node [auto] {} (A0_2);
  \path (A0_0) edge [->,bend right=70] node [auto,swap] {$\scriptstyle{\atlasmap{f}{3}}$} (A0_2);
  \path (A0_2) edge [->,bend left=70] node [auto] {$\scriptstyle{\atlasmap{g}{1}}$} (A0_4);
  \path (A0_2) edge [->] node [auto] {} (A0_4);
  \path (A0_2) edge [->,bend right=70] node [auto,swap] {$\scriptstyle{\atlasmap{g}{3}}$} (A0_4);
\end{tikzpicture}\]
we have

\[\Big([\eta]\odot[\xi]\Big)\ast\Big([\sigma]\odot[\delta]\Big)=\Big([\eta]\ast[\sigma]\Big)
\odot\Big([\xi]\ast[\delta]\Big).\]
\end{lem}

\begin{prop}\label{prop-01}
The definitions of \emph{reduced} orbifold atlases, morphisms and $2$-morphisms, compositions
$\circ,\odot,\ast$ and identities give rise to a $2$-category, that we denote by $\RedAtl$.
\end{prop}

\begin{proof}
In order to construct a $2$-category, we define some data as follows.

\begin{enumerate}[(1)]
 \item The class of objects is the set of all the reduced orbifold atlases $\atlas{X}$ for any
  paracompact, second countable, Hausdorff topological space $X$.
 \item If $\atlas{X}$ and $\atlas{Y}$ are reduced atlases for $X$ and $Y$ respectively, we
  define a small category $\RedAtl(\atlas{X},\atlas{Y})$ as follows: the space of objects is
  the set of all morphisms $\atlasmap{f}{}:\atlas{X}\rightarrow\atlas{Y}$ over any
  continuous map $f:X\rightarrow Y$; for any pair of morphisms $\atlasmap{f}{}$ and
  $\atlasmap{g}{}$ over $f$ and $g$ respectively, using Definition~\ref{def-05} and~\ref{def-06}
  we set:

  \[\Big(\RedAtl(\atlas{X},\atlas{Y})\Big)(\atlasmap{f}{},\atlasmap{g}{}):=\left\{\begin{array}{c c}
  \textrm{all}\,\,2\textrm{-morphisms}\,\atlasmap{f}{}\Rightarrow\atlasmap{g}{} & \textrm{if~}f=g,\\
  \varnothing & \textrm{otherwise.}\end{array}\right.\]
  The composition in any such category is the vertical composition $\odot$, that is clearly
  associative; the identity over any object $\atlasmap{f}{}$ is given by $i_{\atlasmap{f}{}}$. By
  Lemma~\ref{lem-01}(a) we get that actually any such category is an internal groupoid in
  $(\operatorname{Sets})$, i.e.\ a category
  where all the morphisms are invertible.
 \item For every triple of reduced atlases $\atlas{X},\atlas{Y},\atlas{Z}$, we define a functor
  ``composition''

  \[\RedAtl(\atlas{X},\atlas{Y})\times\RedAtl(\atlas{Y},\atlas{Z})\longrightarrow
  \RedAtl(\atlas{X},\atlas{Z})\]
  as $\circ$ on any pair of morphisms and as $\ast$ on any pair of $2$-morphisms. We want to
  prove that this gives rise to a functor. It is easy to see that identities are preserved, so one
  needs only to prove that compositions are preserved, i.e.\ that the \emph{interchange law}
  (see~\cite[Proposition~1.3.5]{Bo}) is satisfied. This is exactly the statement of
  Lemma~\ref{lem-02}.
\end{enumerate}

All the other necessary proofs that $\RedAtl$ is a $2$-category are trivial, so we omit them.
\end{proof}

\section{From reduced orbifold atlases to proper, effective, \'etale groupoids}
The aim of this section is to define a $2$-functor $\functor{F}^{\red}$ from $\RedAtl$ to the
$2$-category of proper, effective, \'etale Lie groupoids. We recall briefly the necessary
definitions and notations.

\begin{defin}
\cite[Definition~2.11]{Ler} A \emph{Lie groupoid} is the datum of $2$ smooth (Hausdorff, paracompact)
manifolds $\groupoid{X}_0,\groupoid{X}_1$ and five smooth maps:

\begin{itemize}
 \item $s,t:\groupoid{X}_1\rightrightarrows\groupoid{X}_0$, such that both $s$ and $t$ are submersions
  (so that the fiber products of the form $\fiber{\groupoid{X}_1}{t}{s}{\cdots\,_t\times_s}
  {\groupoid{X}_1}$ (for finitely many terms) \emph{are manifolds}); these $2$ maps are usually called
  \emph{source} and \emph{target} of the Lie groupoid;
 \item $m:\fiber{\groupoid{X}_1}{t}{s}{\groupoid{X}_1}\rightarrow\groupoid{X}_1$, called
  \emph{multiplication};
 \item $i:\groupoid{X}_1\rightarrow\groupoid{X}_1$, known as \emph{inverse} of the Lie groupoid;
 \item $e:\groupoid{X}_0\rightarrow\groupoid{X}_1$, called \emph{identity};
\end{itemize}
which satisfy the following axioms:

\begin{enumerate}[({LG}1)]
 \item\label{LG1} $s\circ e=1_{\groupoid{X}_0}=t\circ e$;
 \item\label{LG2} if we denote by $\pr_1$ and $\pr_2$ the $2$ projections from
  $\fiber{\groupoid{X}_1}{t}{s}{\groupoid{X}_1}$ to $\groupoid{X}_1$, then we have $s\circ m=s\circ
  \pr_1$ and $t\circ m=t\circ\pr_2$;
 \item\label{LG3} the $2$ morphisms $m\circ (1_{\groupoid{X}_1}\times m)$ and $m\circ(m\times
  1_{\groupoid{X}_1})$ from $\fiber{\groupoid{X}_1}{t}{s}{\groupoid{X}_1\,_t\times_s\groupoid{X}_1}$
  to $\groupoid{X}_1$ are equal;
 \item\label{LG4} the $2$ morphisms $m\circ(e\circ s,1_{\groupoid{X}_1})$ and $m\circ
  (1_{\groupoid{X}_1},e\circ t)$ from $\groupoid{X}_1$ to $\groupoid{X}_1$ are both equal to the
  identity of $\groupoid{X}_1$;
 \item\label{LG5} $i\circ i=1_{\groupoid{X}_1}$, $s\circ i=t$ (and therefore $t\circ i=s$); moreover,
  we require that $m\circ (1_{\groupoid{X}_1},i)=e\circ s$ and $m\circ (i,1_{\groupoid{X}_1})=e\circ
  t$.
\end{enumerate}
\end{defin}

In other terms, a Lie groupoid is an internal groupoid
in the category of smooth manifolds, such that $s$ and $t$ are submersions. For simplicity, we will
denote any Lie groupoid as before by $(\groupname{X}{})$ or $\groupoidtot{X}{}$.
In the literature one can also find the
notations $(\groupoid{X}_0,\groupoid{X}_1)$, $(U,R,s,t,m,e,i)$ and
$\fiber{R}{s}{t}{R}\stackrel{m}{\rightarrow}R\stackrel{i}{\rightarrow}
{R_1\rightrightarrows^{\hspace{-0.25 cm}^{s}}_{\hspace{-0.25 cm}_{t}}U}\stackrel{e}{\rightarrow}R$
(where $U$ is the set $\groupoid{X}_0$ and $R$ is the set $\groupoid{X}_1$ in our notations).\\

In the following pages, even if we will deal with several Lie groupoids, we will denote by
$s$ the source
morphism of any such object, and analogously for the morphisms $t,m,e$ and $i$. This will not
create any problem, since it will be always clear from the context what is the Lie groupoid we are
working with.

\begin{defin}\label{def-07}
\cite[\S~2.1]{M} Given $2$ Lie groupoids $\groupoidtot{X}{}$ and $\groupoidtot{Y}{}$,
a \emph{morphism}
between them is any pair $\groupoidmaptot{\psi}{}=\groupoidmap{\psi}{}$,
where $\psi_0:\groupoid{X}_0\rightarrow
\groupoid{Y}_0$ and $\psi_1:\groupoid{X}_1\rightarrow\groupoid{Y}_1$ are smooth
maps, which together commute with all structure morphisms of the $2$ Lie groupoids. In other
words, we require that $s\circ\psi_1=\psi_0\circ s$, $t\circ\psi_1=\psi_0\circ t$, $\psi_0\circ e=
e\circ\psi_0$, $\psi_1\circ m=m\circ(\psi_1\times\psi_1)$ and $\psi_1\circ i=i\circ\psi_1$.
\end{defin}

\begin{defin}\label{def-08}
\cite[Definition~2.3]{PS} Let us suppose that we have fixed $2$ morphisms of Lie groupoids
$\groupoidmaptot{\psi}{m}:\groupoidtot{X}{}\rightarrow\groupoidtot{Y}{}$ for $m=1,2$. Then a
\emph{natural transformation} (also known as $2$-\emph{morphism})
$\alpha:\groupoidmaptot{\psi}{1}
\Rightarrow\groupoidmaptot{\psi}{2}$ is the datum of any smooth map $\alpha:\groupoid{X}_0\rightarrow
\groupoid{Y}_1$, such that the following conditions hold:

\begin{enumerate}[({NT}1)]
 \item\label{NT1} $s\circ\alpha=\psi^1_0$ and $t\circ\alpha=\psi^2_0$;
 \item\label{NT2} $m\circ(\alpha\circ s,\psi^2_1)=m\circ(\psi^1_1,\alpha\circ t)$.
\end{enumerate}
\end{defin}

There are well-known notions of identities, compositions of morphisms, vertical
and horizontal compositions of natural transformations, obtained in analogy with the corresponding
notions in the $2$-category of small categories. In particular, we have:

\begin{prop}
\cite[\S~2.1]{PS} The data of Lie groupoids, morphisms, and natural transformations between them
\emphatic{(}together with compositions and identities\emphatic{)} form a $2$-category, known as
$\LieGpd$.
\end{prop}

\begin{defin}
\cite[\S~1.2 and~\S~1.5]{M} A Lie groupoid $\groupoidtot{X}{}$ is called \emph{proper} if the map
$(s,t):\groupoid{X}_1\rightarrow\groupoid{X}_0\times\groupoid{X}_0$ is proper; it is called
\emph{\'etale} if the maps $s$ and $t$ are both \'etale (i.e.\ local diffeomorphisms). Since each
\'etale map is a submersion, in general we will simply write ``\'etale groupoid'' instead of
``\'etale Lie groupoid''.
\end{defin}

\begin{rem}\label{rem-06}
Given any Lie groupoid $(\groupname{X}{})$, we can define an equivalence relation
$\sim_{\groupoid{X}}$ on $\groupoid{X}_0$ by saying that $x_0\,\sim_{\groupoid{X}}\,x'_0$ if and only
if there is $x_1\in\groupoid{X}_1$ such that $s(x_1)=x_0$ and $t(x_1)=x'_0$.
We give to the set
$|\groupoidtot{X}{}|:=\groupoid{X}_0/\sim_{\groupoid{X}}$ the quotient topology and we call
it the \emph{underlying topological space} of $\groupoidtot{X}{}$; we denote by
$\pr_{\groupoidtot{X}{}}:
\groupoid{X}_0\twoheadrightarrow|\groupoidtot{X}{}|$ the quotient map. Given another Lie
groupoid $(\groupname{Y}{})$ and any morphism $\groupoidmap{\psi}{}:(\groupname{X}{})\rightarrow
(\groupname{Y}{})$, there is a unique
set map $|\groupoidmaptot{\psi}{}|:|\groupoidtot{X}{}|\rightarrow
|\groupoidtot{Y}{}|$ (called the \emph{underlying set map of} $\groupoidmaptot{\psi}{}$),
making the following diagram commute

\begin{equation}\label{eq-10}
\begin{tikzpicture}[xscale=1.5,yscale=-0.8]
    \node (A0_0) at (0, 0) {$\groupoid{X}_0$};
    \node (A0_2) at (2, 0) {$\groupoid{Y}_0$};
    \node (A2_0) at (0, 2) {$|\groupoidtot{X}{}|$};
    \node (A2_2) at (2, 2) {$|\groupoidtot{Y}{}|$.};

    \node (A1_1) at (1, 1) {$\curvearrowright$};
    
    \path (A2_0) edge [->,dashed]node [auto,swap] {$\scriptstyle{|\groupoidmaptot{\psi}{}|}$} (A2_2);
    \path (A0_0) edge [->]node [auto,swap] {$\scriptstyle{\pr_{\groupoidtot{X}{}}}$} (A2_0);
    \path (A0_2) edge [->]node [auto] {$\scriptstyle{\pr_{\groupoidtot{Y}{}}}$} (A2_2);
    \path (A0_0) edge [->]node [auto] {$\scriptstyle{\psi_0}$} (A0_2);
\end{tikzpicture}
\end{equation}

Such a map is defined by
$|\groupoidmaptot{\psi}{}|(\pr_{\groupoidtot{X}{}}(x_0)):=\pr_{\groupoidtot{Y}{}}
\circ\psi_0(x_0)$ for all
$x_0\in\groupoid{X}_0$. Then $|\groupoidmaptot{\psi}{}|$ is well-defined by definition of $\sim_{\groupoid{X}}$
and $\sim_{\groupoid{Y}}$. Since
$|\groupoidmaptot{\psi}{}|$ is the unique map making \eqref{eq-10} commute, then given any pair
of composable morphisms $\groupoidmaptot{\psi}{}$ and $\groupoidmaptot{\xi}{}$, we have
$|\groupoidmaptot{\xi}{}\circ\groupoidmaptot{\psi}{}|=
|\groupoidmaptot{\xi}{}|\circ|\groupoidmaptot{\psi}{}|$.\\

If we assume that $\groupoidtot{X}{}$ and $\groupoidtot{Y}{}$ are proper and \'etale, then
\emph{$\pr_{\groupoidtot{X}{}}$ and $\pr_{\groupoidtot{Y}{}}$ are
open maps} as a consequence of~\cite[Proposition~2.23]{Ler} and \emph{the induced map
$|\groupoidmaptot{\psi}{}|$ is continuous}. Indeed, for any
open set $A'\subseteq|\groupoidtot{Y}{}|$, we
have that $\pr_{\groupoidtot{X}{}}^{-1}(|\groupoidmaptot{\psi}{}|^{-1}(A'))=\psi_0^{-1}
((\pr_{\groupoidtot{Y}{}})^{-1}
(A'))$ is open in $\groupoid{X}_0$;
since $\pr_{\groupoidtot{X}{}}$ is surjective and open, then
$|\groupoidmaptot{\psi}{}|^{-1}(A')$ is
equal to the open set $\pr_{\groupoidtot{X}{}}(\pr_{\groupoidtot{X}{}}^{-1}
(|\groupoidmaptot{\psi}{}|^{-1}(A')))$ of $|\groupoidtot{X}{}|$.
\end{rem}

\begin{rem}\label{rem-04}
Let $\groupoidtot{X}{}$ be a proper \'etale groupoid and let us fix any pair of points
$x_0,x'_0\in\groupoid{X}_0$.
Since both $s$ and $t$ are \'etale, for every point $x_1$ in
$\groupoid{X}_1$ such that $s(x_1)=x_0$ and $t(x_1)=x'_0$,
we can find a sufficiently small open neighborhood $W_{x_1}$ of $x_1$ where
both $s$ and $t$ are invertible. Then we can define a set map

\[t\circ(s|_{W_{x_1}})^{-1}:\,s(W_{x_1})\longrightarrow t(W_{x_1}).\]

Such a map is actually a diffeomorphism from an open neighborhood of $x_0$ to an open neighborhood of
$x'_0$; moreover, it is easy to see that it commutes
with the projection $\pr_{\groupoidtot{X}{}}$. So for each pair of points $x_0,x'_0$ as above we
can define a set map:

\begin{gather}
\label{eq-40} \kappa_{\groupoidtot{X}{}}(x_0,x'_0,-):\{x_1\in\groupoid{X}_1\,\textrm{such that}\,s(x_1)=x_0\,\,
 \textrm{and}\,\,t(x_1)=x'_0\}\longrightarrow \\
\nonumber \{\germ_{x_0}f\,\,\,\,\forall\,\textrm{diffeomorphisms $f$ around $x_0$
s.\ t.}\,\,f(x_0)=x'_0\,\,\textrm{and}\,\,\pr_{\groupoidtot{X}{}}\circ f=\pr_{\groupoidtot{X}{}}\}
\end{gather}
by setting

\begin{equation}\label{eq-47}
\kappa_{\groupoidtot{X}{}}(x_0,x'_0,x_1):=\germ_{x_0}\left(t\circ(s|_{W_{x_1}})^{-1}\right)=
\germ_{x_1}t\cdot\left(\germ_{x_1}s\right)^{-1}.
\end{equation}

We claim that $\kappa_{\groupoidtot{X}{}}(x_0,x'_0,-)$ \emph{is surjective}.
For that, we have to consider $2$ cases separately;
if $\pr_{\groupoidtot{X}{}}(x_0)\neq\pr_{\groupoidtot{X}{}}(x'_0)$, then both the first and the second
set in \eqref{eq-40}
are empty, so $\kappa_{\groupoidtot{X}{}}(x_0,x'_0,-)$ is a bijection. If $\pr_{\groupoidtot{X}{}}(x_0)=
\pr_{\groupoidtot{X}{}}(x'_0)$, this means that there is a (in general non-unique) point $x_1\in
\groupoid{X}_1$, such that $s(x_1)=x_0$ and $t(x_1)=x'_0$. Let us fix any diffeomorphism
$f$ as in the second line of \eqref{eq-40}. Let us denote by $W_{x_1}$ any open neighborhood of $x_1$ such that
both $s$ and $t$ are invertible if restricted to such a set. Then the function

\[g:=s\circ(t|_{W_{x_1}})^{-1}\circ f|_{f^{-1}\circ t(W_{x_1})}\]

is a diffeomorphism around $x_0$, it fixes $x_0$ and it commutes with $\pr_{\groupoidtot{X}{}}$.
As a simple consequence of~\cite[Theorem 2.3]{N}, there is a (in general non-unique) point
$\tx_1$ in $\groupoid{X}_1$, such that $s(\tx_1)=x_0=t(\tx_1)$ and
$\kappa_{\groupoidtot{X}{}}(x_0,x_0,\tx_1)=\germ_{x_0}g$.
This implies that

\[\germ_{x_0}f=\kappa_{\groupoidtot{X}{}}(x_0,x'_0,x_1)\cdot\kappa_{\groupoidtot{X}{}}(x_0,x_0,\tx_1)
=\kappa_{\groupoidtot{X}{}}(x_0,x'_0,m(\tx_1,x_1)),\]
so we have proved that $\kappa_{\groupoidtot{X}{}}(x_0,x'_0,-)$ is surjective.
\end{rem}

\begin{defin}\label{def-09}
\cite[example~1.5]{M} Let us fix any proper, \'etale groupoid $\groupoidtot{X}{}$.
We say that $\groupoidtot{X}{}$
is \emph{effective} (or \emph{reduced}) if $\kappa_{\groupoidtot{X}{}}(x_0,x_0,-)$
is injective for all $x_0\in\groupoid{X}_0$.
\end{defin}

\begin{lem}\label{lem-21}
Let us fix any proper, effective, \'etale groupoid $\groupoidtot{X}{}$. Then the set map
$\kappa_{\groupoidtot{X}{}}(x_0,x'_0,-)$ is a bijection for every pair of points $x_0,x'_0$
in $\groupoid{X}_0$.
\end{lem}

\begin{proof}
Let us fix any pair of points $x_1,\tx_1\in\groupoid{X}_1$ such that $s(x_1)=x_0=s(\tx_1)$, 
$t(x_1)=x'_0=t(\tx_1)$ and $\kappa_{\groupoidtot{X}{}}(x_0,x'_0,x_1)
=\kappa_{\groupoidtot{X}{}}(x_0,x'_0,\tx_1)$.
Then:

\begin{gather*}
\kappa_{\groupoidtot{X}{}}(x_0,x_0,m(x_1,i(\tx_1)))
=\left(\kappa_{\groupoidtot{X}{}}(x_0,x'_0,\tx_1)\right)^{-1}
\cdot\kappa_{\groupoidtot{X}{}}(x_0,x'_0,x_1)= \\
=\germ_{x_0}\id=\kappa_{\groupoidtot{X}{}}(x_0,x_0,e(x_0));
\end{gather*}
since $\groupoidtot{X}{}$ is effective, this implies that $m(x_1,i(\tx_1))=e(x_0)$, i.e.\ that
$x_1=\tx_1$, so we have proved that $\kappa_{\groupoidtot{X}{}}(x_0,x'_0,-)$
is injective; we have already said above that
it is surjective (even without the hypothesis of effectiveness), so we conclude.
\end{proof}

\begin{defin}
We define the $2$-categories $\EGpd$, $\PEGpd$ and $\PEEGpd$ as the full $2$-subcategories of
$\LieGpd$ obtained by restricting the objects to \'etale groupoids, respectively to proper, \'etale Lie
groupoids, respectively to proper, effective, \'etale groupoids (morphisms and $2$-morphisms are
simply restricted according to that).
\end{defin} 

\begin{construction}\label{cons-04}
(adapted from~\cite[Construction~2.4]{Po} and from~\cite[\S~4.4]{Pr2}) Let us fix any reduced
orbifold atlas $\atlas{X}=\{\unifX{i}\}_{i\in I}$ of dimension $n$. Then we define
$\functor{F}^{\red}_0(\atlas{X}):=(\groupname{X}{})$ as the following Lie groupoid.

\begin{itemize}
 \item The manifold $\groupoid{X}_0$ is defined as $\coprod_{i\in I}\tX_i$, with the natural smooth
  structure given by the fact that each $\tX_i$ is an open subset of $\mathbb{R}^n$.
 \item As a set, we define

  \[\groupoid{X}_1:=\left\{\germ_{\tx_i}\lambda\quad\forall\,i\in I,\,\,\,\forall\,\lambda\in
  \change(\atlas{X},i,-),\,\,\,\forall\,\tx_i\in\dom\lambda\right\}.\]
  For each $i\in I$ and for each $\lambda\in\change(\atlas{X},i,-)$ we set

  \[\groupoid{X}_1(\lambda):=\left\{\germ_{\tx_i}\lambda\quad\forall\,\tx_i\in\dom\lambda\right\}
  \subseteq\groupoid{X}_1.\]
  Then the topological and differentiable structure on $\groupoid{X}_1$ are given by the germ topology
  and by the germ differentiable structure. In other terms, we choose as charts for $\groupoid{X}_1$
  all the bijections of the form:

  \begin{equation}\label{eq-06}
  \begin{array}{c c c c}
  \tau_{\lambda}: & \groupoid{X}_1(\lambda) & \rightarrow & \dom\lambda
    \subseteq\tX_i\subseteq\mathbb{R}^n \\
  & \germ_{\tx_i}\lambda & \mapsto & \tx_i
  \end{array}
  \end{equation}
  for each $i\in I$ and for each $\lambda\in\change(\atlas{X},i,-)$ (one needs only to show
  that any pair of charts $\tau_{\lambda},\tau_{\lambda'}$ are compatible on their common domain, but
  this is easy). Therefore, by construction any morphism of the form $\tau_{\lambda}$ is a
  diffeomorphism from $\groupoid{X}_1(\lambda)$ to $\dom\lambda$.
 \item The structure maps are defined as follows:

  \begin{gather*}
  s(\germ_{\tx_i}\lambda):=\tx_i,\quad t(\germ_{\tx_i}\lambda):=\lambda(\tx_i), \\
  m(\germ_{\tx_i}\lambda,\germ_{\lambda(\tx_i)}\lambda'):=\germ_{\lambda(\tx_i)}\lambda'\cdot
   \germ_{\tx_i}\lambda, \\
  i(\germ_{\tx_i}\lambda):=\germ_{\lambda(\tx_i)}\lambda^{-1},\quad e(\tx_i):=\germ_{\tx_i}
   \id_{\tX_i}.
  \end{gather*}
\end{itemize}

A direct check proves that $s$ and $t$ are both \'etale, that $m,e,i$ are smooth and that axioms
(\hyperref[LG1]{LG1}) -- (\hyperref[LG5]{LG5}) are satisfied, so $\functor{F}^{\red}_0(\atlas{X})$ is an
\'etale groupoid.
\end{construction}

\begin{lem}\label{lem-03}
For every reduced orbifold atlas $\atlas{X}$, the \'etale groupoid $\functor{F}^{\red}_0
(\atlas{X})$ is proper and effective, i.e.\ it belongs to $\PEEGpd$.
\end{lem}

\begin{proof}
For each $x_0\in\groupoid{X}_0$ we have to prove that the map $\kappa_{\groupoidtot{X}{}}(x_0,x_0,-)$
described in Remark~\ref{rem-04} is injective. By definition of $\groupoid{X}_0$, each such $x_0$ is equal
to $\tx_i\in\tX_i$ for some $i\in I$. Moreover, any $x_1\in\groupoid{X}_1$ such that $s(x_1)=x_0=t(x_1)$
is necessarily equal to $\germ_{\tx_i}\lambda$ for some $\lambda\in\change(\atlas{X},i,i)$ with $\tx_i\in
\dom\lambda$ and $\lambda(\tx_i)=\tx_i$. For any such $\lambda$, we have

\[\kappa_{\groupoidtot{X}{}}(\tx_i,\tx_i,\germ_{\tx_i}\lambda)=\germ_{\tx_i}\left(t\circ
\left(s|_{\groupoid{X}_1(\lambda)}\right)^{-1}\right)=\germ_{\tx_i}\lambda,\]
so $\kappa_{\groupoidtot{X}{}}(\tx_i,\tx_i,-)$ is injective,
hence $\functor{F}^{\red}_0(\atlas{X})$ is effective. So we need only to prove that
$(s,t):\groupoid{X}_1\rightarrow\groupoid{X}_0\times\groupoid{X}_0$ is proper. Let us fix any compact
set $K\subseteq\groupoid{X}_0\times\groupoid{X}_0$ and let $\{q^m\}_{m\in\mathbb{N}}$ be any sequence
in $(s,t)^{-1}(K)\subseteq\groupoid{X}_1$. Up to passing to a
subsequence we can always assume that the sequence $(s,t)(q^m)\in K$ converges to some point
$(\tx_i,\overline{x}_{i'})\in\tX_i\times\tX_{i'}$. So there is $m_1$ such that for each $m>m_1$
we have $(s,t)
(q^m)\in\tX_i\times\tX_{i'}$, hence we can write $(s,t)(q^m)=:(\tx_i^m,\overline{x}_{i'}^m)$. By
definition of $\groupoid{X}_1$, for $m>m_1$ there is a (non-unique) change of charts
$\lambda^m$ from $\unifX{i}$ to $\unifX{i'}$, such that

\[\tx_i^m\in\dom\lambda^m,\quad q^m=\germ_{\tx_i}\lambda^m,\quad \lambda^m(\tx_i^m)=
\overline{x}_{i'}^m.\]

From this, we get

\[\pi_i(\tx_i^m)=\pi_{i'}\circ\lambda^m(\tx_i^m)=\pi_{i'}(\overline{x}_{i'}^m)\quad\forall
m>m_1.\]

By considering the limit for $m\rightarrow\infty$, we get that $\pi_i(\tx_i)=\pi_{i'}
(\overline{x}_{i'})$. Since $\atlas{X}$ is a reduced orbifold atlas, then there exists a change
of charts $\lambda$ from $\unifX{i}$ to $\unifX{i'}$, such that $\tx_i\in\dom\lambda$.
So there is $m_2\geq m_1$, such that for all $m>m_2$
we have $\tx_i^m\in\dom\lambda$. Then for each such $m$ there exists a chart
$(\tX^m,G^m,\pi_i|_{\tX^m})$ such that $\tx_i^m\in\tX^m\subseteq\dom\lambda\cap\dom\lambda^m
\subseteq\tX_i$. For each $m>m_2$ we have that both $\lambda$ and $\lambda^m$ (suitably
restricted) can be considered as embeddings from $(\tX^m,G^m,\pi_i|_{\tX^m})$ to $\unifX{i'}$. So
by~\cite[Lemma~A.1]{MP} there exists a unique $g^m\in G_{i'}$ such that $\lambda^m|_{\tX^m}=g^m
\circ\lambda|_{\tX^m}$. Since $G_{i'}$ is a finite set, then after passing to a subsequence we can
assume that $g^m$ is the same for all $m>m_2$; we denote such a map by $g$. Then by definition of
the differentiable structure on $\groupoid{X}_1$ we have

\[\lim_{m\rightarrow+\infty}q^m=\lim_{m\rightarrow+\infty}\germ_{\tx_i^m}\lambda^m=\lim_{m
\rightarrow+\infty}\germ_{\tx_i^m}g\circ\lambda=\germ_{\tx_i}g\circ\lambda.\]

So this proves that $(s,t)^{-1}(K)$ is compact, so $(s,t)$ is proper.
\end{proof}

Until now we have associated to each object of $\RedAtl$ an object of $\PEEGpd$; we want to do the
same for morphisms and $2$-morphisms.

\begin{construction}\label{cons-05}
(adapted from~\cite[Proposition~4.7]{Po}) Let us fix any pair of reduced orbifold atlases
$\atlas{X}:=\{\unifX{i}\}_{i\in I}$ and $\atlas{Y}:=\{\unifY{j}\}_{j\in J}$ for $X$ and $Y$
respectively and any morphism $\atlasmap{f}{}:\atlas{X}\rightarrow\atlas{Y}$ with representative
given by

\[\atlasmaprep{f}{}:=\left(f,\overf,\left\{\tf_i\right\}_{i\in I},P_f,\nu_f\right).\]

We set

\[\functor{F}^{\red}_0(\atlas{X})=:\left(\groupname{X}{}\right)\quad\textrm{and}\quad
\functor{F}^{\red}_0(\atlas{Y})=:\left(\groupname{Y}{}\right),\]
where

\begin{gather*}
\groupoid{X}_0:=\coprod_{i\in I}\tX_i,\quad\groupoid{X}_1:=\Big\{\germ_{\tx_i}\lambda\quad\forall\,i
 \in I,\,\,\,\forall\,\lambda\in\change(\atlas{X},i,-),\,\,\,\forall\,\tx_i\in\dom\lambda\Big\}, \\
\groupoid{Y}_0:=\coprod_{j\in J}\tY_j,\quad\groupoid{Y}_1:=\Big\{\germ_{\ty_j}\omega\quad\forall\,j
 \in J,\,\,\,\forall\,\omega\in\change(\atlas{Y},j,-),\,\,\,\forall\,\ty_j\in\dom\omega\Big\}.
\end{gather*}

Then we define a set map $\psi_0:\groupoid{X}_0\rightarrow\groupoid{Y}_0$ as

\[\left.\psi_0\right|_{\tX_i}:=\tf_i:\tX_i\longrightarrow\tY_{\overf(i)}\subseteq\groupoid{Y}_0\]
for all $i\in I$. Now let $x_1$ be any point in $\groupoid{X}_1$ and let $\tx_i:=s(x_1)\in\tX_i$ for
some $i\in I$. Since $P_f$ is a good subset of $\change(\atlas{X})$, then there is a (non-unique)
$\lambda\in P_f(i,-)$ such that $x_1=\germ_{\tx_i}\lambda$. We set

\begin{equation}\label{eq-29}
\psi_1(x_1):=\germ_{\tf_i(\tx_i)}\nu_f(\lambda)\in\groupoid{Y}_1.
\end{equation}

If $\lambda'$ is another element of $P_f(i,-)$ such that $x_1=\germ_{\tx_i}\lambda'$, then property
(\hyperref[M5d]{M5d}) for $\atlasmaprep{f}{}$ proves that $\germ_{\tf_i(\tx_i)}\nu_f(\lambda)=
\germ_{\tf_i(\tx_i)}\nu_f(\lambda')$, so $\psi_1$ is
a well-defined set map from $\groupoid{X}_1$ to $\groupoid{Y}_1$. A direct check proves that both
$\psi_0$ and $\psi_1$ are smooth and that the pair $\groupoidmap{\psi}{}$ satisfies
Definition~\ref{def-07}, so it is a morphism of Lie groupoids
from $\functor{F}^{\red}_0(\atlas{X})$ to
$\functor{F}^{\red}_0(\atlas{Y})$. Now let us suppose that 

\[\atlasmaprep{f}{\prime}:=\left(f,\overf,\left\{\tf_i\right\}_{i\in I},P_f',\nu_f'\right)\]
is another representative for $\atlasmap{f}{}$. Then by Definition~\ref{def-03} we get that the
morphism from $\functor{F}^{\red}_0(\atlas{X})$ to $\functor{F}^{\red}_0(\atlas{Y})$ associated to
$\atlasmaprep{f}{\prime}$ coincides with $\groupoidmap{\psi}{}$. Therefore it makes sense to set

\[\functor{F}^{\red}_1(\atlasmap{f}{}):=\groupoidmap{\psi}{}:\,\functor{F}^{\red}_0(\atlas{X})
\longrightarrow\functor{F}^{\red}_0(\atlas{Y}).\]
\end{construction}

\begin{construction}
Now let us fix any pair of atlases $\atlas{X}$ and $\atlas{Y}$ for $X$ and $Y$ respectively
and any pair of morphisms $\atlasmap{f}{1},\atlasmap{f}{2}:\atlas{X}\rightarrow\atlas{Y}$ over a
continuous function $f:X\rightarrow Y$, with representatives

\[\atlasmaprep{f}{m}:=\left(f,\overf^m,\left\{\tf_i^m\right\}_{i\in I},P_{f^m},\nu_{f^m}\right)\quad
\textrm{for}\,\,m=1,2.\]

Let us also fix any $2$-morphism $[\delta]:\atlasmap{f}{1}\Rightarrow \atlasmap{f}{2}$ and any
representative

\[\delta:=\left\{\left(\tX_i^a,\delta_i^a\right)\right\}_{i\in I,\,a\in A(i)}\]
for it. Let us set

\begin{gather*}
\functor{F}^{\red}_0(\atlas{X})=:\left(\groupname{X}{}\right),\quad\functor{F}^{\red}_0(\atlas{Y})=:
 \left(\groupname{Y}{}\right), \\
\functor{F}^{\red}_1(\atlasmap{f}{m})=:\groupoidmap{\psi}{m}\quad\textrm{for}\,\,m=1,2. 
\end{gather*}

Then let us define a set map $\underline{\delta}:\groupoid{X}_0=\coprod_{i\in I}\tX_i\rightarrow
\groupoid{Y}_1$ as

\[\underline{\delta}(\tx_i):=\germ_{\tf_i^1(\tx_i)}\delta_i^a\]
for every $i\in I$, for every $a\in A$ and for every $\tx_i\in\tX_i^a$; this is well-defined by
property (\hyperref[2Md]{2Md}) for $\delta$. We claim that $\underline{\delta}$ is a natural
transformation from
$\groupoidmap{\psi}{1}$ to $\groupoidmap{\psi}{2}$. Clearly $\underline{\delta}$ is smooth, indeed on
each open subset of $\groupoid{X}_0$
of the form $\tX_i^a$ we have that $\underline{\delta}$ coincides with the
composition of $\tf_i^1$ (that is smooth by definition of local lift) and of the inverse of the chart
$\tau_{\delta_i^a}$ for $\groupoid{Y}_1$ (see \eqref{eq-06}). Moreover, let us fix any $i\in I$, any
$a\in A(i)$ and any $\tx_i\in\tX_i^a$. Then

\begin{gather*}
s\circ\underline{\delta}(\tx_i)=s\left(\germ_{\tf_i^1(\tx_i)}\delta_i^a\right)=\tf^1_i(\tx_i)
 =\psi^1_0(\tx_i), \\
t\circ\underline{\delta}(\tx_i)=t\left(\germ_{\tf_i^1(\tx_i)}\delta_i^a\right)=\delta_i^a\circ
 \tf^1_i(\tx_i)\stackrel{(\hyperref[2Mc]{2Mc})}{=}\tf^2_i(\tx_i)=\psi^2_0(\tx_i),
\end{gather*}
so $\underline{\delta}$ satisfies axiom (\hyperref[NT1]{NT1}). Now let us fix any point $x_1\in
\groupoid{X}_1$ and let us set $\tx_i:=s(x_1)$, $\overline{x}_{i'}:=t(x_1)$ for a unique pair $(i,i')
\in I\times I$. Since both $P_{f^1}$ and $P_{f^2}$ are good subsets of $\change(\atlas{X})$, then
for each $m=1,2$ there exists $\lambda^m\in P_{f^m}(i,i')$ such that $\germ_{\tx_i}\lambda^m=x_1$. By
property (\hyperref[2Ma]{2Ma}) there exist $a\in A(i)$ and $a'\in A(i')$ such that $\tx_i\in\tX_i^a$ and
$\overline{x}_{i'}\in\tX_{i'}^{a'}$. Then:

\begin{gather*}
\left(m\circ\left(\underline{\delta}\circ s,\psi^2_1\right)\right)(x_1)=m\left(\underline{\delta}
 (\tx_i),\psi^2_1\left(\germ_{\tx_i}\lambda^2\right)\right)\stackrel{\eqref{eq-29}}{=} \\
\stackrel{\eqref{eq-29}}{=}
 m\left(\germ_{\tf^1_i(\tx_i)}\delta_i^a,\germ_{\tf^2_i(\tx_i)}\nu_f^2(\lambda^2)\right)=
 \germ_{\tf_i^2(\tx_i)}\nu_f^2(\lambda^2)\cdot\germ_{\tf_i^1(\tx_i)}\delta_i^a
 \stackrel{(\hyperref[2Meprime]{2Me})'}{=} \\
\stackrel{(\hyperref[2Meprime]{2Me})'}{=}\germ_{\tf_{i'}^1(\lambda^1(\tx_i))}\delta_{i'}^{a'}\cdot
 \germ_{\tf_i^1(\tx_i)}\nu_f^1(\lambda^1)\stackrel{\eqref{eq-29}}{=}
 m\left(\psi^1_1\left(\germ_{\tx_i}\lambda^1\right),
 \underline{\delta}\left(\lambda^1(\tx_i)\right)\right)= \\
=\left(m\circ\left(\psi^1_1,\underline{\delta}\circ t\right)\right)(x_1). 
\end{gather*}

So $\underline{\delta}$ satisfies also axiom (\hyperref[NT2]{NT2}); therefore $\underline{\delta}$ is a
natural transformation from $\functor{F}^{\red}_1(\atlasmap{f}{1})$ to $\functor{F}^{\red}_1(\atlasmap{f}{2})$.
By Definition~\ref{def-06} we get that $\underline{\delta}$ depends only on $[\delta]$ and not on the
representative $\delta$ chosen for that class. So it makes sense to set:

\[\functor{F}^{\red}_2([\delta]):=\underline{\delta}:\,\functor{F}^{\red}_1(\atlasmap{f}{1})
\Longrightarrow\functor{F}^{\red}_1(\atlasmap{f}{2}).\]
\end{construction}

A direct check proves that:

\begin{lem}\label{lem-04}
For every pair of composable morphisms $\atlasmap{f}{}:\atlas{X}\rightarrow\atlas{Y}$,
$\atlasmap{g}{}:
\atlas{Y}\rightarrow\atlas{Z}$ we have $\functor{F}^{\red}_1(\atlasmap{g}{}\circ\atlasmap{f}{})=
\functor{F}^{\red}_1(\atlasmap{g}{})\circ\functor{F}^{\red}_1(\atlasmap{f}{})$. For every
reduced orbifold atlas $\atlas{X}$ we have
$\functor{F}^{\red}_1(\id_{\atlas{X}})=\id_{\functor{F}^{\red}_0(\atlas{X})}$; for every morphism
$\atlasmap{f}{}$ between reduced orbifold atlases we have $\functor{F}^{\red}_2(i_{\atlasmap{f}{}})=
i_{\functor{F}^{\red}_1(\atlasmap{f}{})}$.
\end{lem}

\begin{lem}\label{lem-05}
Let us fix any diagram in $\RedAtl$ as follows:

\[\begin{tikzpicture}[scale=0.8]
    \def\x{2.5}
    \def\y{-1.2}
    \node (A0_0) at (0*\x, 0*\y) {$\atlas{X}$};
    \node (B0_1) at (0.6*\x, -0.3) {$\scriptstyle{\atlasmap{f}{2}}$};
    \node (A0_2) at (2*\x, 0*\y) {$\atlas{Y}$.};
    
    \node (A0_1) at (1.1*\x, 0.6) {$\Downarrow[\delta]$};
    \node (A0_1) at (1.1*\x, -0.6) {$\Downarrow[\sigma]$};
    
    \path (A0_0) edge [->,bend left=50] node [auto] {$\scriptstyle{\atlasmap{f}{1}}$} (A0_2);
    \path (A0_0) edge [->] node [auto,swap] {} (A0_2);
    \path (A0_0) edge [->,bend right=50] node [auto,swap] {$\scriptstyle{\atlasmap{f}{3}}$} (A0_2);
\end{tikzpicture}\]

Then $\functor{F}^{\red}_2([\sigma]\odot[\delta])=\functor{F}^{\red}_2([\sigma])\odot
\functor{F}^{\red}_2([\delta])$.
\end{lem}

\begin{proof}
Let us set representatives

\[\atlasmaprep{f}{m}:=\left(f,\overf^m,\left\{\tf_i^m\right\}_{i\in I},P_{f^m},\nu_{f^m}\right)\quad
\textrm{for}\quad m=1,2,3\]
for each $\atlasmap{f}{m}$ and representatives

\[\delta=\left\{\left(\tX_i^a,\delta_i^a\right)\right\}_{i\in I,\,a\in A(i)}\quad\textrm{and}\quad
\sigma=\left\{\left(\tX_i^b,\sigma_i^b\right)\right\}_{i\in I,\,b\in B(i)}\]
for $[\delta]$ and $[\sigma]$ respectively, as in Construction~\ref{cons-02}. Then let us set
$\functor{F}^{\red}_2([\delta])=:\underline{\delta}$ and $\functor{F}^{\red}_2([\sigma])=:
\underline{\sigma}$; let us fix any $i\in I$, any $(a,b)\in A(i)\times B(i)$ and any $\tx_i\in
\tX_i^{a,b}:=\tX_i^a\cap\tX_i^b$. Then

\begin{gather*}
\Big(\functor{F}^{\red}_2\left(\left[\sigma\right]\right)\odot\functor{F}^{\red}_2\left(
 \left[\delta\right]\right)\Big)(\tx_i)=\left(\underline{\sigma}\odot\underline{\delta}\right)
 (\tx_i)\stackrel{(\ast)}{=}m\left(\underline{\delta}(\tx_i),\underline{\sigma}(\tx_i)\right)= \\
=m\left(\germ_{\tf_i^1(\tx_i)}\delta_i^a,\germ_{\tf_i^2(\tx_i)}\sigma_i^b\right)=
 \germ_{\tf_i^2(\tx_i)}\sigma_i^b\cdot\germ_{\tf_i^1(\tx_i)}\delta_i^a= \\
=\germ_{\tf_i^1(\tx_i)}\sigma_i^b\circ\delta_i^a|_{\tY_i^{a,b}}=\functor{F}^{\red}_2\Big(\left[
 \sigma\right]\odot\left[\delta\right]\Big)(\tx_i),
\end{gather*}
where $\tY_i^{a,b}$ is defined as in \eqref{eq-05} and $(\ast)$ is the definition of
vertical composition of $2$-morphisms in $\LieGpd$, hence also in $\PEEGpd$.
\end{proof}

\begin{lem}\label{lem-06}
Let us fix any diagram in $\RedAtl$ as follows:

\[\begin{tikzpicture}[scale=0.8]
    \def\x{2.0}
    \def\y{-1.2}
    \node (A0_0) at (0*\x, 0*\y) {$\atlas{X}$};
    \node (A0_2) at (2*\x, 0*\y) {$\atlas{Y}$};
    \node (A0_4) at (4*\x, 0*\y) {$\atlas{Z}$.};

    \node (A0_3) at (3*\x, 0*\y) {$\Downarrow[\xi]$};
    \node (A0_1) at (1*\x, 0*\y) {$\Downarrow[\delta]$};
        
    \path (A0_2) edge [->,bend left=25] node [auto] {$\scriptstyle{\atlasmap{g}{1}}$} (A0_4);
    \path (A0_2) edge [->,bend right=25] node [auto,swap] {$\scriptstyle{\atlasmap{g}{2}}$} (A0_4);
    \path (A0_0) edge [->,bend left=25] node [auto] {$\scriptstyle{\atlasmap{f}{1}}$} (A0_2);
    \path (A0_0) edge [->,bend right=25] node [auto,swap] {$\scriptstyle{\atlasmap{f}{2}}$} (A0_2);
\end{tikzpicture}\]

Then $\functor{F}^{\red}_2([\xi]\ast[\delta])=\functor{F}^{\red}_2([\xi])\ast
\functor{F}^{\red}_2([\delta])$.
\end{lem}

\begin{proof}
Let us set representatives

\begin{gather*}
\atlasmaprep{f}{m}:=\left(f,\overline{f}^m,\Big\{\tf_i^m\Big\}_{i\in I},P_{f^m},\nu_{f^m}\right)\quad
 \textrm{for}\quad m=1,2, \\
\atlasmaprep{g}{m}:=\left(g,\overline{g}^m,\Big\{\tg_j^m\Big\}_{j\in J},P_{g^m},\nu_{g^m}\right)\quad
 \textrm{for}\quad m=1,2
\end{gather*}
for $\atlasmap{f}{m}$ and $\atlasmap{g}{m}$ respectively and representatives

\[\delta=\left\{\left(\tX_i^a,\delta_i^a\right)\right\}_{i\in I,\,a\in A(i)},\quad\xi=\left\{\left(
\tY_j^c,\xi_j^c\right)\right\}_{j\in J,\,c\in C(j)}\]
for $[\delta]$ and $[\xi]$ respectively, as in Construction~\ref{cons-03}; let us denote by 

\[\gamma=\left\{\left(\tX_i^{a,c,e},\gamma_i^{a,c,e}\right)\right\}_{i\in I,\,(a,c)\in A(i)\times
C(\overf^2(i)),\,e\in E(i,a,c)}\]
a representative for $[\xi]\ast[\delta]$, obtained as in the already mentioned construction. Then
let us set:

\begin{gather*}
\functor{F}^{\red}_0(\atlas{X})=:\left(\groupname{X}{}\right),\quad\functor{F}^{\red}_0(\atlas{Y})=:
 \left(\groupname{Y}{}\right),\quad \functor{F}^{\red}_0(\atlas{Z})=:\left(\groupname{Z}{}\right), \\
\functor{F}^{\red}_1(\atlasmap{f}{m})=:\groupoidmap{\psi}{m},\quad\functor{F}^{\red}_1
 (\atlasmap{g}{m})=:\groupoidmap{\phi}{m}\quad\textrm{for}\,\,m=1,2, \\
\functor{F}^{\red}_2([\delta])=:\underline{\delta},\quad \functor{F}^{\red}_2([\xi])=:
 \underline{\xi}.
\end{gather*}

Let us fix any $i\in I$, any $(a,c)\in A(i)\times C(\overf^2(i))$, any $e\in E(i,a,c)$ and any
point $\tx_i\in\tX_i^{a,c,e}$. Then we have

\begin{gather*}
\Big(\functor{F}^{\red}_2\left(\left[\xi\right]\right)\ast\functor{F}^{\red}_2\left(\left[\delta
 \right]\right)\Big)(\tx_i)\stackrel{(\ast)}{=}
 m\circ\left(\phi^1_1\circ\underline{\delta}(\tx_i),\underline{\xi}\circ
 \psi^2_0(\tx_i)\right)= \\
=m\left(\phi^1_1\left(\germ_{\tf_i^1(\tx_i)}\delta_i^a\right),\underline{\xi}\left(\tf_i^2(\tx_i)
 \right)\right)= \\
=m\left(\germ_{\tg_{\overf^1(i)}^1\circ\tf_i^1(\tx_i)}\nu_g^1(\widetilde{\delta}_i^a),
 \germ_{\tg_{\overf^2(i)}^1\circ\tf_i^2(\tx_i)}\xi_{\overf^2(i)}^c\right)= \\
=\germ_{\tg_{\overf^2(i)}^1\circ\tf_i^2(\tx_i)}\xi_{\overf^2(i)}^c\cdot\germ_{\tg_{\overf^1(i)}^1
 \circ\tf_i^1(\tx_i)}\nu_g^1(\widetilde{\delta}_i^a)= \\
=\germ_{\tg_{\overf^1(i)}^1\circ\tf_i^1(\tx_i)}\xi_{\overf^2(i)}^c\circ\nu_g^1(
 \widetilde{\delta}_i^a)|_{\tZ_i^{a,c,e}}=\functor{F}^{\red}_2\Big(\left[\xi\right]\ast\left[
 \delta\right]\Big)(\tx_i),
\end{gather*}
where $\widetilde{\delta}_i^a$ and $\tZ_i^{a,c,e}$ are defined as in Construction~\ref{cons-03}
and $(\ast)$ is the definition of horizontal composition in $\LieGpd$.
\end{proof}

Lemmas~\ref{lem-03},~\ref{lem-04},~\ref{lem-05} and~\ref{lem-06} prove that:

\begin{theo}\label{theo-01}
The data $\functor{F}^{\red}:=(\functor{F}^{\red}_0,\functor{F}^{\red}_1,\functor{F}^{\red}_2)$
define a $2$-functor from $\RedAtl$ to $\PEEGpd$.
\end{theo}

We state some properties of $\functor{F}^{\red}$ that we are going to use soon.

\begin{lem}\label{lem-07}
\emphatic{(adapted from~\cite[Proposition~4.9]{Po})} Let us fix any pair of reduced orbifold atlases
$\atlas{X},\atlas{Y}$ for $X$ and $Y$ respectively. Let us set $\functor{F}^{\red}_0(
\atlas{X})=:(\groupname{X}{})$ and $\functor{F}^{\red}_0(\atlas{Y})=:(\groupname{Y}{})$. Let us also
fix any morphism $\groupoidmap{\psi}{}:(\groupname{X}{})\rightarrow(\groupname{Y}{})$. Then there
is a unique morphism $\atlasmap{f}{}:\atlas{X}\rightarrow\atlas{Y}$ in $\RedAtl$, such that
$\functor{F}^{\red}_1(\atlasmap{f}{})=\groupoidmap{\psi}{}$.
\end{lem}

\begin{proof}
Let us suppose that $\atlas{X}=\{\unifX{i}\}_{i\in I}$ and $\atlas{Y}=\{\unifY{j}\}_{j\in J}$.
Since each $\tX_i$ is connected by definition of orbifold atlas, then the morphism $\psi_0:
\groupoid{X}_0\rightarrow\groupoid{Y}_0$ induces a set map $\overf:I\rightarrow J$ such that $\psi_0
(\tX_i)\subseteq\tY_{\overf(i)}$ for every $i\in I$. For each $i\in I$ we set $\tf_i:=\psi_0|_{\tX_i}:
\tX_i\rightarrow\tY_{\overf(i)}$.\\

Now let us consider the continuous and surjective maps

\[\pi:\,\groupoid{X}_0=\coprod_{i\in I}\tX_i
\,{\ensuremath{\relbar\joinrel\twoheadrightarrow}}\,X
\quad\textrm{and}\quad\quad\chi:\,\groupoid{Y}_0=\coprod_{j\in j}\tY_j\,
{\ensuremath{\relbar\joinrel\twoheadrightarrow}}\,Y\]
defined as $\pi|_{\tX_i}:=\pi_i$ for each $i\in I$ and $\chi|_{\tY_j}:=\chi_j$ for each $j\in J$.
If we consider the equivalence relation
$\sim_{\groupoid{X}}$ on $\groupoid{X}_0$ defined in Remark~\ref{rem-06} and the induced quotient
$\pr_{\groupoidtot{X}{}}:
\groupoid{X}_0\twoheadrightarrow|\groupoidtot{X}{}|=\groupoid{X}_0/\sim_{\groupoid{X}}$,
then we can construct easily
an homeomorphism $\varphi_{\atlas{X}}:|\groupoidtot{X}{}|
\stackrel{\sim}{\rightarrow}X$, such that $\varphi_{\atlas{X}}\circ\pr_{\groupoidtot{X}{}}=\pi$.
Analogously, there is an homeomorphism $\varphi_{\atlas{Y}}:|\groupoidtot{Y}{}|
\stackrel{\sim}{\rightarrow}Y$ such that $\varphi_{\atlas{Y}}\circ\pr_{\groupoidtot{Y}{}}=\chi$.
Then we
define a continuous map $f:X\rightarrow Y$ as $f:=\varphi_{\atlas{Y}}\circ
|\groupoidmaptot{\psi}{}|\circ
\varphi_{\atlas{X}}^{-1}$, where $|\groupoidmaptot{\psi}{}|$ is the continuous map
$|\groupoidtot{X}{}|\rightarrow|\groupoidtot{Y}{}|$ induced by
$\groupoidmaptot{\psi}{}$ as in Remark~\ref{rem-06}. Then we get that

\[f\circ\pi=\varphi_{\atlas{Y}}\circ|\groupoidmaptot{\psi}{}|\circ\varphi_{\atlas{X}}^{-1}\circ\pi=
\varphi_{\atlas{Y}}\circ|\groupoidmaptot{\psi}{}|\circ\pr_{\groupoidtot{X}{}}
\stackrel{\eqref{eq-10}}{=}\varphi_{\atlas{Y}}\circ\pr_{\groupoidtot{Y}{}}\circ\psi_0=
\chi\circ\psi_0,\]
hence

\[f(\pi_i(\tx_i))=\chi\circ\psi_0(\tx_i)=\chi_{\overf(i)}\circ\tf_i(\tx_i)\quad\forall\,i\in I,\,\,
\forall\,\tx_i\in\tX_i,\]
so the collection $\{\tf_i\}_{i\in I}$ satisfies condition (\hyperref[M3]{M3}) (see Definition~\ref{def-11}).
Now let us fix any $i\in I$, any $\overline{x}_i\in\tX_i$ and any $\lambda\in
\change(\atlas{X},i,-)$ with $\overline{x}_i\in\dom\lambda$. Using the construction of $\groupoid{Y}_1$
and the fact that
$\groupoidmap{\psi}{}$ commutes with $s$ (see Definition~\ref{def-07}),
there is a (non-unique) $\omega\in
\change(\atlas{Y},\overf(i),-)$ such that

\begin{equation}\label{eq-07}
\psi_1(\germ_{\overline{x}_i}\lambda)=\germ_{\psi_0(\overline{x}_i)}\omega=\germ_{\tf_i(
\overline{x}_i)}\omega.
\end{equation}

Now $\psi_1$ is continuous, so the set $A:=\psi_1^{-1}(\groupoid{Y}_1(\omega))\cap\groupoid{X}_1(\lambda)
\subseteq\groupoid{X}_1$
is an open set (for the notations used here, see \eqref{eq-06}); $A$ is non-empty since it contains
$\germ_{\overline{x}_i}\lambda$. For each point $\tx_i\in s(A)\subseteq\groupoid{X}_0$, using the
definition of $A$ (and the fact that $s$ commutes with $(\psi_0,\psi_1)$, we have

\begin{equation}\label{eq-42}
\psi_1(\germ_{\tx_i}\lambda)=\germ_{\psi_0(\tx_i)}\omega=\germ_{\tf_i(\tx_i)}\omega
\end{equation}
(so \eqref{eq-07} holds not only for $\overline{x}_i$, but also for every $\tx_i$ in
an open neighborhood of $\overline{x}_i$). Now let us choose any open connected subset
$\tX_{\lambda,\overline{x}_i,\omega}\subseteq s(A)\subseteq\groupoid{X}_0$, such that:

\begin{itemize}
 \item $\overline{x}_i\in\tX_{\lambda,\overline{x}_i,\omega}$;
 \item $\tX_{\lambda,\overline{x}_i,\omega}$ is invariant under the action of
  $\Stab(G_i,\overline{x}_i)$;
 \item for all $g\in G_i\smallsetminus\Stab(G_i,\overline{x}_i)$ we have $g(\tX_{\lambda,
  \overline{x}_i,\omega})\cap\tX_{\lambda,\overline{x}_i,\omega}=\varnothing$
\end{itemize}
(in this way, $\lambda$ is again a change of charts if restricted to $\tX_{\lambda,\overline{x}_i,
\omega}$). Then let us set

\[P_f:=\left\{\lambda|_{\tX_{\lambda,\overline{x}_i,\omega}}\quad\forall\,i\in I,\,\,\,\forall\,
\lambda\in\change(\atlas{X},i,-),\,\,\,\forall\,\overline{x}_i\in\dom\lambda\right\};\]
if we have $2$ (or more) collections $(\lambda,\overline{x}_i,\omega)$ and $(\lambda',
\overline{x}'_i,\omega')$ such that

\begin{equation}\label{eq-08}
\lambda|_{\tX_{\lambda,\overline{x}_i,\omega}}=\lambda'|_{\tX_{\lambda',\overline{x}'_i,\omega'}},
\end{equation}
then we simply make an arbitrary choice of a triple $(\lambda,\overline{x}_i,\omega)$ associated
to the morphism \eqref{eq-08} in $P_f$. Then for each $\lambda|_{\tX_{\lambda,
\overline{x}_i,\omega}}\in P_f$ we set $\nu_f\left(\lambda|_{\tX_{\lambda,\overline{x}_i,
\omega}}\right):=\omega$ and it is easy to see that the collection

\[\atlasmaprep{f}{}:=\left(f,\overf,\left\{\tf_i\right\}_{i\in I},P_f,\nu_f\right)\]
is a representative of a morphism from $\atlas{X}$ to $\atlas{Y}$. The collection $\atlasmaprep{f}{}$
depends on some choices, but the class $\atlasmap{f}{}$ depends only on $\groupoidmap{\psi}{}$.
A direct check using \eqref{eq-42}
proves that $\functor{F}^{\red}_1(\atlasmap{f}{})=\groupoidmap{\psi}{}$; moreover it
is easy to see that if
$\atlasmap{f}{1},\atlasmap{f}{2}:\atlas{X}\rightarrow\atlas{Y}$ are such that $\functor{F}^{\red}_1
(\atlasmap{f}{1})=\functor{F}^{\red}_1(\atlasmap{f}{2})$, then $\atlasmap{f}{1}=\atlasmap{f}{2}$.
This suffices to complete the proof.
\end{proof}

Lemma~\ref{lem-07} proves that $\atlasmap{f}{}$ as above is unique. So
the previous proof implies that:

\begin{cor}\label{cor-05}
Let us fix any reduced orbifold atlas $\atlas{X}=\{\unifX{i}\}_{i\in I}$
on a topological space $X$ and let
us set $\groupoidtot{X}{}:=\functor{F}^{\red}_0(\atlas{X})$. Then there is a canonical homeomorphism
$\varphi_{\atlas{X}}:|\groupoidtot{X}{}|\stackrel{\sim}{\rightarrow} X$ \emphatic{(}sending
each point $[\tx_i]\in|\groupoidtot{X}{}|$ to $\pi_i(\tx_i)$ for each point $\tx_i\in
\tX_i$ and for each $i\in I$\emphatic{)}. Here ``canonical '' means
that given any other reduced orbifold atlas $\atlas{Y}$ on a topological space $Y$ and any 
morphism of reduced orbifold atlases $\atlasmap{f}{}:\atlas{X}\rightarrow\atlas{Y}$ over a continuous
map $f:X\rightarrow Y$, we have

\[f\circ\varphi_{\atlas{X}}=\varphi_{\atlas{Y}}\circ|\functor{F}_1^{\red}(\atlasmap{f}{})|,\]
where $|\functor{F}_1^{\red}(\atlasmap{f}{})|$ is the continuous map from $|\groupoidtot{X}{}|$
to $|\groupoidtot{Y}{}|$ associated to $\functor{F}_1^{\red}(\atlasmap{f}{})$ by
\emphatic{Remark~\ref{rem-06}}.
\end{cor}

\begin{lem}\label{lem-08}
Let us fix any pair of reduced orbifold atlases $\atlas{X},\atlas{Y}$ for $2$ topological
spaces $X$ and $Y$ respectively, and any pair of morphisms $\atlasmap{f}{m}:\atlas{X}\rightarrow
\atlas{Y}$ for $m=1,2$ with representatives

\[\atlasmaprep{f}{m}:=\left(f,\overf,\left\{\tf_i^m\right\}_{i\in I},P_{f^m},\nu_{f^m}\right)\quad
\operatorname{for}\,\,m=1,2.\]

Let us set

\[\functor{F}^{\red}_0(\atlas{X})=:\groupoidtot{X}{},\quad\functor{F}^{\red}_0(\atlas{Y}):=
\groupoidtot{Y}{},\quad\functor{F}^{\red}_1(\atlasmap{f}{m})=:\groupoidmaptot{\psi}{m}
\quad\operatorname{for}\,\,m=1,2.\]

Let us also fix any natural transformation $\alpha:\groupoidmaptot{\psi}{1}\Rightarrow\groupoidmaptot{\psi}{2}$.
Then there exists a unique $2$-morphism $[\delta]:\atlasmap{f}{1}\Rightarrow\atlasmap{f}{2}$ such that
$\functor{F}^{\red}_2([\delta])=\alpha$.
\end{lem}

\begin{proof}
By Definition~\ref{def-08}, $\alpha$ is a smooth map from $\groupoid{X}_0$ to $\groupoid{Y}_1$ such
that $s\circ\alpha=\psi^1_0$; so for each $i\in I$ and for each $\overline{x}_i\in\tX_i\subseteq
\groupoid{X}_0$, we can choose a change of charts $\delta_i^{\overline{x}_i}$ of $\atlas{Y}$, such
that

\begin{equation}\label{eq-09}
\alpha(\overline{x}_i)=\germ_{\psi^1_0(\overline{x}_i)}\delta_i^{\overline{x}_i}=\germ_{\tf_i^1(
\overline{x}_i)}\delta_i^{\overline{x}_i}.
\end{equation}

For each $\overline{x}_i\in\tX_i$ we consider the set

\[\groupoid{Y}_1(\delta_i^{\overline{x}_i}):=\left\{\germ_{\ty_{\overf^1(i)}}\delta_i^{\overline{x}_i}
\quad\forall\,\,\ty_{\overf^1(i)}\in\dom\delta_i^{\overline{x}_i}\right\}\subseteq\groupoid{Y}_1.\]

By Construction~\ref{cons-04} (for $\groupoid{Y}_1$ instead of $\groupoid{X}_1$),
$\groupoid{Y}_1(\delta_i^{\overline{x}_i})$ is
open in $\groupoid{Y}_1$; moreover the map $\tau:\groupoid{Y}_1(\delta_i^{\overline{x}_i})\rightarrow
\dom\delta_i^{\overline{x}_i}$ defined by

\[\tau\left(\germ_{\ty_{\overf^1(i)}}\delta_i^{\overline{x}_i}\right):=\ty_{\overf^1(i)}\]
is a diffeomorphism to an open connected subset of some $\mathbb{R}^n$ (here $n$ is the dimension
of the atlas $\atlas{Y}$). Then we define an open set:

\[\tX_i^{\overline{x}_i}:=\alpha^{-1}(\groupoid{Y}_1(\delta_i^{\overline{x}_i}))\cap\tX_i.\]

For each $i\in I$ we choose any collection $\{\overline{x}_i^a\}_{a\in A(i)}$ such that the family
$\{\tX_i^{\overline{x}^a_i}\}_{a\in A(i)}$ is an open covering of $\tX_i$. For simplicity of
notations, we set $\delta_i^a:=\delta_i^{\overline{x}_i^a}$ and $\tX_i^a:=
\tX_i^{\overline{x}^a_i}$. We claim that the collection $\delta:=\{(\tX_i^a,\delta_i^a)\}_{i\in I,
a\in A(i)}$ is a representative of a $2$-morphism from $\atlasmap{f}{1}$ to $\atlasmap{f}{2}$.\\

In order to prove that, let us fix any $i\in I$, any $a\in A(i)$ and any $\tx_i\in\tX_i^a$. By
definition of $\tX_i^a$ and using condition (\hyperref[NT1]{NT1}) (see Definition~\ref{def-08}), we have: 

\begin{equation}\label{eq-31}
\alpha(\tx_i)=\germ_{\psi^1_0(\tx_i)}\delta_i^a=\germ_{\tf_i^1(\tx_i)}\delta_i^a
\end{equation}
(in other terms, \eqref{eq-09} holds not only for the point $\overline{x}_i^a$, but also for any
$\tx_i$ in an
open neighborhood of $\overline{x}_i$). Again by (\hyperref[NT1]{NT1}) we have $t\circ\alpha=\psi^2_0$;
so for each $\tx_i\in\tX_i^a$ we have

\[\tf_i^2(\tx_i)=\psi^2_0(\tx_i)=t\circ\alpha(\tx_i)=t\left(\germ_{\tf_i^1(\tx_i)}\delta_i^a
\right)=\delta_i^a\circ\tf_i^1(\tx_i),\]
so in particular

\[\tf_i^1(\tX_i^a)\subseteq\dom\delta_i^a\quad\textrm{and}\quad
\tf_i^2(\tX_i^a)\subseteq\cod\delta_i^a;\]
therefore properties (\hyperref[2Ma]{2Ma}), (\hyperref[2Mb]{2Mb}) and (\hyperref[2Mc]{2Mc}) 
(see Definition~\ref{def-05}) are verified
for $\delta$. If $a$ and $a'$ are indices in $A(i)$ and $\tx_i\in\tX_i^a\cap\tX_i^{a'}$, then by
\eqref{eq-31} we have

\[\germ_{\tf_i^1(\tx_i)}\delta_i^a=\alpha(\tx_i)=\germ_{\tf_i^1(\tx_i)}\delta_i^{a'},\]
so (\hyperref[2Md]{2Md}) holds. Now let us fix any $(i,i')\in I\times I$, any $(a,a')\in A(i)\times
A(i')$, any $\lambda\in\change(\atlas{X},i,i')$ and any $\tx_i\in\dom\lambda\cap\tX_i^a$ such that
$\lambda(
\tx_i)\in\tX_{i'}^{a'}$. Since $P_{f^1}$ and $P_{f^2}$ are both good subsets of $\change(\atlas{X})$,
then for each $m=1,2$ there exists $\lambda^m\in P_{f^m}(i,i')$ such that $\tx_i\in\dom\lambda^m$ and
$\germ_{\tx_i}\lambda^m=\germ_{\tx_i}\lambda$. We recall (see Construction~\ref{cons-05}) that 

\[\psi^m_1(\germ_{\tx_i}\lambda)=\germ_{\tf_i^m(\tx_i)}\nu_{f^m}(\lambda^m)\quad\textrm{for}\,\,
m=1,2.\]

Therefore:

\begin{gather*}
\germ_{\tf_i^2(\tx_i)}\nu_{f^2}(\lambda^2)\cdot\germ_{\tf_i^1(\tx_i)}\delta_i^a=m\left(
 \germ_{\tf_i^1(\tx_i)}\delta_i^a,\germ_{\tf_i^2(\tx_i)}\nu_{f^2}(\lambda^2)\right)
  \stackrel{\eqref{eq-31}}{=} \\
\stackrel{\eqref{eq-31}}{=}\left(m\circ\left(\alpha\circ s,\psi^2_1\right)\right)(\germ_{\tx_i}
 \lambda^2)\stackrel{(\hyperref[NT2]{NT2})}{=}\left(m\circ\left(\psi^1_1,\alpha\circ t\right)\right)
 (\germ_{\tx_i}\lambda^1)= \\
=m\left(\germ_{\tf_i^1(\tx_i)}\nu_{f^1}\left(\lambda^1\right),\germ_{\tf_{i'}^1(\lambda^1(\tx_i))}
 \delta_{i'}^{a'}\right)= \\
=\germ_{\tf_{i'}^1(\lambda(\tx_i))}\delta_{i'}^{a'}\cdot\germ_{\tf_i^1(\tx_i)}\nu_{f^1}(\lambda^1).
\end{gather*}

So also property (\hyperref[2Me]{2Me}) holds. Therefore $\delta$ is a representative of a $2$-morphism
from
$\atlasmap{f}{1}$ to $\atlasmap{f}{2}$. Different choices of the sets $\{\overline{x}_i^a\}_a$ and
$\{\delta_i^a\}_a$ give rise to different $\delta$'s, but their equivalence class $[\delta]$ is
the same. Now by \eqref{eq-31} we get that $\functor{F}^{\red}_2([\delta])=\alpha$; moreover, a
direct computation proves that if $[\delta^1]$ and $[\delta^2]$ are such that
$\functor{F}^{\red}_2([\delta^1])=\functor{F}^{\red}_2([\delta^2])$, then $[\delta^1]=[\delta^2]$.
This suffices to conclude.
\end{proof}

So we have proved that for every pair of reduced orbifold atlases $\atlas{X},\atlas{Y}$ the
functor

\[\functor{F}^{\red}(\atlas{X},\atlas{Y}):\,\RedAtl(\atlas{X},
\atlas{Y})\longrightarrow\PEEGpd(\functor{F}^{\red}_0(\atlas{X}),\functor{F}^{\red}_0
(\atlas{Y}))\]
is a bijection on objects and morphisms (i.e.\  on $1$-morphisms and $2$-morphisms of $\RedAtl$ and
of $\PEEGpd$). $\functor{F}^{\red}_0$ is not injective; indeed, given any homeomorphism $f:X
\stackrel{\sim}{\rightarrow}Y$ and any reduced orbifold atlas $\atlas{X}$ on $X$, by
Definition~\ref{def-13} we get that $\functor{F}^{\red}_0(\atlas{X})=\functor{F}^{\red}_0
(f_{\ast}(\atlas{X}))$. It is
not difficult to see that actually this is the only point where $\functor{F}^{\red}_0$ fails to
be injective. We will see in Lemma~\ref{lem-11} below that $\functor{F}^{\red}_0$ is surjective
only up to ``Morita equivalences''.

\section{Morita equivalences between \'etale groupoids}
As we mentioned in the Introduction, the bicategory of reduced orbifold atlases described in terms
of proper, effective and \'etale groupoids is not $\PEEGpd$, but a bicategory obtained from
$\PEEGpd$ by selecting a suitable class of morphisms (Morita equivalences, see below) and by
``turning'' them into internal equivalences. We briefly recall the axioms that are needed for that
construction. We do that firstly because we will need the explicit description of the bicategory
obtained from $\PEEGpd$ by applying such a procedure; secondly, because later we will need to perform an
analogous construction on $\RedAtl$.

\begin{defin}(\cite[\S~2.1]{Pr} in the special case of $2$-categories instead of bicategories)
Let us fix any $2$-category $\CATC$ and any class $\SETW$ of morphisms in $\CATC$. We recall that the
pair $(\CATC,\SETW)$ is said \emph{to admit a right bicalculus of fractions} if the following
conditions are satisfied:

\begin{enumerate}[({BF}1)]\label{BF}
 \item\label{BF1} for every object $A$ of $\CATC$, the $1$-identity $\id_A$ belongs to $\SETW$;
 \item\label{BF2} $\SETW$ is closed under compositions;
 \item\label{BF3} for every morphism $\operatorname{w}:A\rightarrow B$ in $\SETW$ and for every
  morphism $f:C\rightarrow B$, there are an object $D$, a morphism $\operatorname{w}':D\rightarrow
  C$ in $\SETW$, a morphism $f':D\rightarrow A$ and an invertible $2$-morphism $\alpha:
  f\circ\operatorname{w}'\Rightarrow\operatorname{w}\circ f'$;
 \item\label{BF4}
 \begin{enumerate}[(a)]
  \item\label{BF4a} given any morphism $\operatorname{w}:B\rightarrow A$ in $\SETW$, any pair of
   morphisms $f^1,f^2:C\rightarrow B$ and any $\alpha:\operatorname{w}\circ f^1\Rightarrow
   \operatorname{w}\circ f^2$, there are an object $D$, a morphism $\operatorname{v}:D\rightarrow
   C$ in $\SETW$ and a $2$-morphism $\beta:f^1\circ\operatorname{v}\Rightarrow f^2\circ
   \operatorname{v}$, such that $\alpha\ast i_{\operatorname{v}}=i_{\operatorname{w}}\ast\beta$;
  \item\label{BF4b} if $\alpha$ in (a) is invertible, then so is $\beta$;
  \item\label{BF4c} if $(D',\operatorname{v}':D'\rightarrow C,\beta':f^1\circ\operatorname{v}'
   \Rightarrow f^2\circ\operatorname{v}')$ is another triple with the same properties of $(D,
   \operatorname{v},\beta)$ in (a), then there are an object $E$, a pair of morphisms
   $\operatorname{u}:E\rightarrow D$, $\operatorname{u}':E\rightarrow D'$ and an invertible
   $2$-morphism $\zeta:\operatorname{v}\circ\operatorname{u}\Rightarrow\operatorname{v}'\circ
   \operatorname{u}'$, such that $\operatorname{v}\circ\operatorname{u}$ belongs to $\SETW$ and
   
   \[\Big(\beta'\ast i_{\operatorname{u}'}\Big)\odot\Big(i_{f^1}\ast\zeta\Big)=\Big(i_{f^2}\ast\zeta
   \Big)\odot\Big(\beta\ast i_{\operatorname{u}}\Big);\]
  \end{enumerate}
 \item\label{BF5} if $\operatorname{w}:A\rightarrow B$ is a morphism in $\SETW$, $\operatorname{v}:A
  \rightarrow B$ is any morphism and if there is an invertible $2$-morphism $\alpha:
  \operatorname{v}\Rightarrow\operatorname{w}$, then also $\operatorname{v}$ belongs to $\SETW$.
\end{enumerate}
\end{defin}

We recall the following fundamental result:

\begin{theo}\label{theo-04}
\cite[Theorem~21]{Pr} Given any $2$-category or bicategory $\CATC$ and any class $\SETW$ as before, there are
a bicategory $\CATC\left[\SETWinv\right]$ \emphatic{(}called \emph{right bicategory of
fractions}\emphatic{)}
and a pseudofunctor $\functor{U}_{\SETW}:\CATC\rightarrow\CATC\left[\SETWinv\right]$ that
sends each element of $\SETW$ to an internal equivalence and that is universal with respect to such a
property.
\end{theo}

\begin{rem}\label{rem-03}
In the notations of~\cite{Pr}, $\functor{U}_{\SETW}$ is called \emph{bifunctor}, but this
notation is no more in use. For the precise meaning of ``universal'' we refer directly
to~\cite[\S~3.2]{Pr}.
In particular, the bicategory $\CATC\left[\SETWinv\right]$ is unique up equivalences of bicategories.
In~\cite{Pr} the theorem above is stated with (\hyperref[BF1]{BF1}) replaced by the slightly stronger
hypothesis

\begin{enumerate}[({BF}1)$'$:]
\item\label{BF1prime} all the equivalences of $\CATC$ are in $\SETW$.
\end{enumerate}

By looking carefully at the proofs in~\cite{Pr}, it is easy to see that the only part of axiom
(\hyperref[BF1prime]{BF1})$'$ that is really used in all the computations is (\hyperref[BF1]{BF1}),
so we are allowed to state the theorem of~\cite{Pr} under such less restrictive hypothesis.
\end{rem}

We refer to~\cite[\S~2.2, 2.3, 2.4]{Pr} and to our paper~\cite{T3}
for more details on the construction of bicategories of
fractions and to~\cite[\S~1.5]{Lei} for a general overview on bicategories and
pseudofunctors. Note that even if $\CATC$ is a $2$-category, in general $\CATC\left[\SETWinv\right]$
is only a bicategory (with trivial unitors but possibly non-trivial associators). In other terms,
in $\CATC\left[\SETWinv\right]$
in general the composition of morphisms and the horizontal compositions of $2$-morphisms are associative
only up to canonical invertible $2$-morphisms.\\

We recall also the following definition, that will be very useful soon.

\begin{defin}\label{def-12}
\cite[Definition~2.1]{T4}
Let us fix any bicategory $\CATC$ and any class $\SETW$ of morphisms in it
(not necessarily satisfying conditions (\hyperref[BF]{BF})).
Then we define a class $\SETWsat$ as the class of all morphisms $f:B
\rightarrow A$ in $\CATC$, such that there is a pair of objects $C,D$ and a pair of morphisms $g:C
\rightarrow B$, $h:D\rightarrow C$, such that both $f\circ g$ and $g\circ h$ belong to $\SETW$. We
call $\SETWsat$ the (right) \emph{saturation} of $\SETW$;
we say that $\SETW$ is (right) \emph{saturated} if it coincides with its saturation.
\end{defin}

We recall (see~\cite[\S~2.4]{M}) that a morphism $\groupoidmap{\psi}{}:(\groupname{X}{})\rightarrow
(\groupname{Y}{})$
between Lie groupoids is a \emph{Morita equivalence} (also known as \emph{weak equivalence} or
\emph{essential equivalence}) if and only if the following $2$ conditions hold:

\begin{enumerate}[({ME}1)]
 \item\label{ME1} the smooth map $t\circ\pi^1:\fiber{\groupoid{Y}_1}{s}{\psi_0}{\groupoid{X}_0}
  \rightarrow\groupoid{Y}_0$ is a surjective submersion (here $\pi^1$ is the projection
  $\fiber{\groupoid{Y}_1}{s}{\psi}{\groupoid{X}_0}\rightarrow\groupoid{Y}_1$ and the fiber product is
  a manifold since $s$ is a submersion);
 \item\label{ME2} the following square is cartesian (it is commutative by Definition~\ref{def-07}):
 
  \begin{equation}\label{eq-16}
  \begin{tikzpicture}[xscale=1.8,yscale=-0.8]
    \node (A0_0) at (0, 0) {$\groupoid{X}_1$};
    \node (A0_2) at (2, 0) {$\groupoid{Y}_1$};
    \node (A2_0) at (0, 2) {$\groupoid{X}_0\times\groupoid{X}_0$};
    \node (A2_2) at (2, 2) {$\groupoid{Y}_0\times\groupoid{Y}_0$.};
    \path (A0_0) edge [->]node [auto] {$\scriptstyle{\psi_1}$} (A0_2);
    \path (A0_0) edge [->]node [auto,swap] {$\scriptstyle{(s,t)}$} (A2_0);
    \path (A0_2) edge [->]node [auto] {$\scriptstyle{(s,t)}$} (A2_2);
    \path (A2_0) edge [->]node [auto] {$\scriptstyle{(\psi_0\times\psi_0)}$} (A2_2);
  \end{tikzpicture}
  \end{equation}
\end{enumerate}

Any $2$ Lie groupoids $\groupoidtot{X}{}$ and $\groupoidtot{Y}{}$ are said to be \emph{Morita 
equivalent} (or \emph{weakly equivalent} or \emph{essentially equivalent}) if and only if there
are a Lie groupoid $\groupoidtot{Z}{}$ and $2$ Morita equivalences as follows

\[\begin{tikzpicture}[xscale=2.5,yscale=-1.2]
    \node (A0_0) at (0, 0) {$\groupoidtot{X}{}$};
    \node (A0_1) at (1, 0) {$\groupoidtot{Z}{}$};
    \node (A0_2) at (2, 0) {$\groupoidtot{Y}{}$.};
    \path (A0_1) edge [->]node [auto,swap] {$\scriptstyle{\groupoidmaptot{\psi}{1}}$} (A0_0);
    \path (A0_1) edge [->]node [auto] {$\scriptstyle{\groupoidmaptot{\psi}{2}}$} (A0_2);
\end{tikzpicture}\]

This is actually an equivalence relation, see for example~\cite[Chapter~5]{MM}.
Given any Morita equivalence as above, the induced set map
$|\groupoidmaptot{\psi}{}|$ is surjective by (\hyperref[ME1]{ME1}) and it is injective as a
consequence of (\hyperref[ME2]{ME2}). If both $\groupoidtot{X}{}$ and $\groupoidtot{Y}{}$ are
\'etale, then by~\cite[Exercise~5.16(4)]{MM} the map $\psi_0$ is
\'etale, hence open. If in addition $\groupoidtot{X}{}$ and $\groupoidtot{Y}{}$ are also proper, then 
$\pr_{\groupoidtot{Y}{}}:\groupoid{Y}_0\twoheadrightarrow
|\groupoidtot{Y}{}|$ is also open and $|\groupoidmaptot{\psi}{}|$ is continuous
(see Remark~\ref{rem-06}). Since the projection
$\pr_{\groupoidtot{X}{}}:\groupoid{X}_0
\twoheadrightarrow|\groupoidtot{X}{}|$ is continuous and surjective, then diagram \eqref{eq-10}
proves that the induced map
$|\groupoidmaptot{\psi}{}|$ is open, i.e.\ that $|\groupoidmaptot{\psi}{}|^{-1}$ is continuous,
so we have:

\begin{lem}\label{lem-14}
Let us fix any Morita equivalence $\groupoidmaptot{\psi}{}:\groupoidtot{X}{}\rightarrow
\groupoidtot{Y}{}$ between proper,
\'etale groupoids. Then the induced continuous map $|\groupoidmaptot{\psi}{}|:
|\groupoidtot{X}{}|\rightarrow|\groupoidtot{Y}{}|$ is an homeomorphism.
\end{lem}

\begin{lem}\label{lem-20}
Let us fix any pair of proper, effective, \'etale groupoids $\groupoidtot{X}{},\groupoidtot{Y}{}$
and  any pair of Morita equivalences $\groupoidmaptot{\psi}{1},\groupoidmaptot{\psi}{2}:
\groupoidtot{X}{}\rightarrow
\groupoidtot{Y}{}$. Then the following facts are equivalent:

\begin{enumerate}[\emphatic{(}a\emphatic{)}]
 \item the topological maps $|\groupoidmaptot{\psi}{1}|$ and $|\groupoidmaptot{\psi}{2}|$
  \emphatic{(}see \emphatic{Remark~\ref{rem-06})} coincide;
 \item there exists a natural transformation $\alpha:\groupoidmaptot{\psi}{1}\Rightarrow
  \groupoidmaptot{\psi}{2}$ in $\PEEGpd$;
 \item there exists a unique natural transformation $\alpha:\groupoidmaptot{\psi}{1}\Rightarrow
  \groupoidmaptot{\psi}{2}$ in $\PEEGpd$.  
\end{enumerate}
\end{lem}

\begin{proof}
Let us assume (a) and let us prove (b).
Using~\cite[Exercise~5.16(4)]{MM} we get that both $\psi^1_0$ and $\psi^2_0$ are
\'etale maps, so for every point $x_0$ in $\groupoid{X}_0$ there is an open
neighborhood $W_{x_0}$ of $x_0$, such that both $\psi^1_0$ and $\psi^2_0$ are diffeomorphisms
if restricted to $W_{x_0}$. Then the map $f_{x_0}$ defined by

\[f_{x_0}:=\psi^2_0\circ\left(\psi^1_0|_{W_{x_0}}\right)^{-1}:\,\psi^1_0(W_{x_0})\longrightarrow
\psi^2_0(W_{x_0})\]
is a diffeomorphism from an open neighborhood of $\psi^1_0(x_0)$ to an open neighborhood
of $\psi^2_0(x_0)$. Moreover, since the topological maps
$|\groupoidmaptot{\psi}{1}|$ and $|\groupoidmaptot{\psi}{2}|$ (both defined from
$|\groupoidtot{X}{}|$ to $|\groupoidtot{Y}{}|$) coincide, then we get easily that $f_{x_0}$
commutes with the projection $\pr_{\groupoidtot{Y}{}}:\groupoidtot{Y}{}\twoheadrightarrow
|\groupoidtot{Y}{}|$ (see Remark~\ref{rem-06}). By Lemma~\ref{lem-21} (applied to
$\groupoidtot{Y}{}$), the set map
$\kappa_{\groupoidtot{Y}{}}(\psi^1_0(x_0),\psi^2_0(x_0),-)$ is a bijection.
Therefore, it makes sense to define

\[\alpha(x_0):=\kappa_{\groupoidtot{Y}{}}(\psi^1_0(x_0),\psi^2_0(x_0),-)^{-1}
(\germ_{\psi^1_0(x_0)}f_{x_0})\quad\forall\,x_0\in\groupoid{X}_0.\]

So we have defined a set map $\alpha:\groupoid{X}_0\rightarrow\groupoid{Y}_1$. Given any
point $x_0\in\groupoid{X}_0$, by
definition of $\kappa_{\groupoidtot{Y}{}}(\psi^1_0(x_0),\psi^2_0(x_0),-)$ (see \eqref{eq-40})
we have

\begin{equation}\label{eq-39}
s\circ\alpha(x_0)=\psi^1_0(x_0),\quad\quad t\circ\alpha(x_0)=\psi^2_0(x_0).
\end{equation}

Since both $s$ and $\psi^1_0$ are \'etale, the first identity
implies that $\alpha$ is an \'etale map; moreover \eqref{eq-39} proves that 
condition (\hyperref[NT1]{NT1}) holds for $\alpha$ (see Definition~\ref{def-08}).
We want to prove also that condition (\hyperref[NT2]{NT2}) is satisfied. In order
to prove that, let us fix any point $x_1\in\groupoid{X}_1$ and let us set 
$x_0:=s(x_1)$ and $x'_0:=t(x_1)$. Then we have:

\begin{gather*}
\germ_{\psi^1_0(x_0)}f_{x_0}=\germ_{x_0}\psi^2_0\cdot(\germ_{x_0}\psi^1_0)^{-1}= \\
=\germ_{x_0}\psi^2_0\cdot\germ_{x_1}s\cdot(\germ_{x_1}t)^{-1}\cdot\germ_{x_1}t
 \cdot(\germ_{x_1}s)^{-1}
 \cdot(\germ_{x_0}\psi^1_0)^{-1}\stackrel{(\ast)}{=} \\
\stackrel{(\ast)}{=}\germ_{\psi^2_1(x_1)}s\cdot\germ_{x_1}\psi^2_1\cdot
 (\germ_{x_1}t)^{-1}\cdot \\
\cdot\germ_{x_1}t\cdot(\germ_{x_1}\psi^1_1)^{-1}\cdot(\germ_{\psi^1_1(x_1)}s)^{-1}
 \stackrel{(\ast)}{=} \\
\stackrel{(\ast)}{=}\germ_{\psi^2_1(x_1)}s\cdot(\germ_{\psi^2_1(x_1)}t)^{-1}\cdot
 \germ_{x'_0}\psi^2_0\cdot \\
\cdot(\germ_{x'_0}\psi^1_0)^{-1}\cdot\germ_{\psi^1_1(x_1)}t\cdot
 (\germ_{\psi^1_1(x_1)}s)^{-1}= \\
=\left(\kappa_{\groupoidtot{Y}{}}(\psi^2_0(x_0),\psi^2_0(x'_0),\psi_1^2(x_1))\right)^{-1}
 \cdot\germ_{\psi^1_0(x'_0)}f_{x'_0}\cdot
 \kappa_{\groupoidtot{Y}{}}(\psi^1_0(x_0),\psi^1_0(x'_0),\psi^1_1(x_1)).
\end{gather*}
where all the identities denoted by $(\ast)$ are a consequence of Definition~\ref{def-07} for
$\groupoidmaptot{\psi}{1}$
and $\groupoidmaptot{\psi}{2}$. This implies that:

\begin{gather}
\nonumber \kappa_{\groupoidtot{Y}{}}(\psi^2_0(x_0),\psi^2_0(x'_0),\psi_1^2(x_1))\cdot
 \germ_{\psi^1_0(x_0)}f_{x_0}= \\
\label{eq-43} =\germ_{\psi^1_0(x'_0)}f_{x'_0}\cdot
 \kappa_{\groupoidtot{Y}{}}(\psi^1_0(x_0),\psi^1_0(x'_0),\psi^1_1(x_1)).
\end{gather}

Using the definition of $\kappa_{\groupoidtot{Y}{}}(-,-,-)$ (see Remark~\ref{rem-04} with
$\groupoidtot{X}{}$ replaced by $\groupoidtot{Y}{}$),
it is easy to prove that given any pair of objects
$\overline{y}_1,\widetilde{y}_1\in\groupoid{Y}_1$ with
$t(\overline{y}_1)=s(\widetilde{y}_1)$, we have

\[\kappa_{\groupoidtot{Y}{}}(s(\overline{y}_1),t(\widetilde{y}_1),
m(\overline{y}_1,\widetilde{y}_1))=\kappa_{\groupoidtot{Y}{}}(s(\widetilde{y}_1),
t(\widetilde{y}_1),\widetilde{y}_1)\cdot
\kappa_{\groupoidtot{Y}{}}(s(\overline{y}_1),t(\overline{y}_1),\overline{y}_1).\]

We apply this fact to \eqref{eq-43} together with the definition of $\alpha$
and the fact that the set maps $\kappa_{\groupoidtot{Y}{}}(y_0,y'_0,-)$ are
injective for every pair of points $y_0,y'_0$ in $\groupoid{Y}_0$ (see again
Lemma~\ref{lem-21}). So we conclude that:

\[m(\alpha(x_0),\psi_1^2(x_1))=m(\psi^1_1(x_1),\alpha(x'_0)).\]

In other terms, we have $m(\alpha\circ s(x_1),\psi_1^2(x_1))=m(\psi^1_1(x_1),\alpha\circ
t(x_1))$ for every point $x_1$ in $\groupoid{X}_1$.
This proves that condition (\hyperref[NT2]{NT2}) holds for $\alpha$, so
we have constructed a natural transformation $\alpha:\groupoidmaptot{\psi}{1}\Rightarrow
\groupoidmaptot{\psi}{2}$ as required in (b).\\

Now let us assume (b) and let us prove that (c) holds. So let us suppose that there is a pair of
natural transformations $\alpha,\beta:\groupoidmaptot{\psi}{1}\Rightarrow\groupoidmaptot{\psi}{2}$. We denote by
$\groupoid{X}_0^{\reg}$ the subset of $\groupoid{X}_0$ consisting of regular points, namely those points
$x_0$ such that $(s,t)^{-1}(x_0,x_0)=\{e(x_0)\}\subset\groupoid{X}_1$. Since $\groupoidtot{X}{}$ is effective
by hypothesis, then $\groupoid{X}_0^{\reg}$ is open and dense in $\groupoid{X}_0$. Let us fix any point
$x_0$ in $\groupoid{X}_0^{\reg}$; since $\groupoidmaptot{\psi}{1}$ is a Morita equivalence, then by
(\hyperref[ME2]{ME2}) (see Definition~\ref{def-12}) we get that $\psi^1_0(x_0)$ belongs to
$\groupoid{Y}_0^{\reg}$. Moreover, by (\hyperref[NT1]{NT1}) (see Definition~\ref{def-08}) we have

\[s\circ m(\alpha(x_0),i\circ\beta(x_0))=s\circ\alpha(x_0)=\psi^1_0(x_0)\]
and

\[t\circ m(\alpha(x_0),i\circ\beta(x_0))=t\circ i\circ\beta(x_0)=s\circ\beta(x_0)=\psi^1_0(x_0),\]
hence

\[m(\alpha(x_0),i\circ\beta(x_0))\in(s,t)^{-1}\left\{(\psi^1_0(x_0),\psi^1_0(x_0))\right\}=
\left\{e\circ\psi^1_0(x_0)\right\},\]
so $\alpha(x_0)=\beta(x_0)$ for each $x_0\in\groupoid{X}_0^{\reg}$. Now let us denote by $n$ the dimension
of the \'etale groupoid $\groupoidtot{X}{}$ (i.e.\ the dimension of $\groupoid{X}_0$, equivalently of
$\groupoid{X}_1$). If $n=0$, since $\groupoidtot{X}{}$ is effective we have that $\groupoid{X}_0=
\groupoid{X}_0^{\reg}$, so $\alpha=\beta$. If $n\geq 1$, then each connected component of $\groupoid{X}_0$
contains a dense subset where $\alpha$ and $\beta$ coincide. Since such functions are both continuous
(by Definition~\ref{def-08}), with target in the Hausdorff space $\groupoid{Y}_1$, then we conclude that
$\alpha$ and $\beta$ coincide everywhere, so (c) holds.\\

Lastly, let us assume that (c) holds and let us prove (a). Given any point $x_0\in\groupoid{X}_0$, using
condition (\hyperref[NT1]{NT1}) and Remark~\ref{rem-06} we get that

\begin{gather*}
|\groupoidmaptot{\psi}{1}|(\pr_{\groupoidtot{X}{}}(x_0))=
 \pr_{\groupoidtot{Y}{}}\circ\psi_0^1(x_0)=\pr_{\groupoidtot{Y}{}}\circ s\circ\alpha(x_0)= \\
=\pr_{\groupoidtot{Y}{}}\circ t\circ\alpha(x_0)=\pr_{\groupoidtot{Y}{}}\circ\psi_0^2(x_0)=
 |\groupoidmaptot{\psi}{2}|(\pr_{\groupoidtot{X}{}}(x_0)),
\end{gather*}
so (a) holds. 
\end{proof}

Note that the previous Lemma is false if we remove the hypothesis of effectiveness.\\

We denote by $\WEGpd$ the set of all Morita equivalences in $\EGpd$, i.e.\ the set of all Morita
equivalences between \'etale groupoids. Then we have
the following useful result (the second part of which was also proved in~\cite[Lemma~8.1]{PS}).

\begin{lem}\label{lem-17}
\cite[Corollary~4.2(b) and (c) and Proposition~2.11(ii)]{T4}
The class $\WEGpd$ is right saturated. In particular, it
satisfies the ``\,$2$-out-of-$3$-property'', i.e.\  given any
pair of morphisms $\groupoidmaptot{\phi}{}:\groupoidtot{X}{}\rightarrow\groupoidtot{Y}{}$ and
$\groupoidmaptot{\psi}{}:\groupoidtot{Y}{}
\rightarrow\groupoidtot{Z}{}$ between \'etale groupoids, if any $2$ out of $\{\groupoidmaptot{\phi}{},
\groupoidmaptot{\psi}{},\groupoidmaptot{\psi}{}\circ
\groupoidmaptot{\phi}{}\}$ are Morita equivalences, then so is the third one.
\end{lem}

Moreover we have:

\begin{prop}\label{prop-02}
\cite[\S~4.1]{Pr} The set $\WEGpd$ admits a right bicalculus of fractions, so there are a
bicategory $\EGpd\left[\WEGpdinv\right]$ and a pseudofunctor

\[\functor{U}_{\WEGpd}:\,\EGpd\longrightarrow\EGpd\left[\WEGpdinv\right]\]
that sends each Morita equivalence between \'etale groupoids to an internal equivalence and that is
universal with respect to such a property \emphatic{(}see \emphatic{Remark~\ref{rem-03})}.
\end{prop}

We denote by $\WPEGpd$, respectively by $\WPEEGpd$, the set of all Morita equivalences between
proper and \'etale groupoids, respectively between proper, effective and \'etale groupoids. Then
we have the following standard result:

\begin{prop}\label{prop-03}
The sets $\WPEGpd$ and $\WPEEGpd$ admit a right bicalculus of fractions \emphatic{(}in $\PEGpd$ and
$\PEEGpd$ respectively\emphatic{)}, so there are a pair of bicategories 
$\PEGpd\left[\WPEGpdinv\right]$ and $\PEEGpd\left[\WPEEGpdinv\right]$ and pseudofunctors

\[\functor{U}_{\WPEGpd}:\,\PEGpd\longrightarrow\PEGpd\left[\WPEGpdinv\right],\]

\begin{equation}\label{eq-11}
\functor{U}_{\WPEEGpd}:\,\PEEGpd\longrightarrow\PEEGpd\left[\WPEEGpdinv\right]
\end{equation}
that send each Morita equivalence between proper, \emphatic{(}effective\emphatic{)} \'etale
groupoids to an internal equivalence and that are universal with respect to such a property \emphatic{(}see
\mbox{\emphatic{Remark~\ref{rem-03}}\emphatic{)}.} Moreover, both $\PEGpd\left[\WPEGpdinv\right]$ and $\PEEGpd
\left[\WPEEGpdinv\right]$ are full bi-subcategories of $\EGpd\left[\WEGpdinv\right]$ and we have a
commutative diagram as follows:

\[\begin{tikzpicture}[xscale=2.2,yscale=-1.2]
    \node (A0_0) at (0, 0) {$\PEEGpd$};
    \node (A0_2) at (2, 0) {$\PEGpd$};
    \node (A0_4) at (4, 0) {$\EGpd$};
    \node (A2_4) at (4, 2) {$\EGpd\left[\WEGpdinv\right]$,};
    \node (A2_0) at (0, 2) {$\PEEGpd\left[\WPEEGpdinv\right]$};
    \node (A2_2) at (2, 2) {$\PEGpd\left[\WPEGpdinv\right]$};
    
    \node (A1_1) at (1.2, 1) {$\curvearrowright$};
    \node (A1_3) at (3.2, 1) {$\curvearrowright$};
    
    \path (A0_0) edge [->]node [auto] {$\scriptstyle{\functor{U}_{\WPEEGpd}}$} (A2_0);
    \path (A2_0) edge [right hook->]node [auto] {$\scriptstyle{}$} (A2_2);
    \path (A0_2) edge [->]node [auto] {$\scriptstyle{\functor{U}_{\WPEGpd}}$} (A2_2);
    \path (A0_0) edge [right hook->]node [auto] {$\scriptstyle{}$} (A0_2);
    \path (A0_4) edge [->]node [auto] {$\scriptstyle{\functor{U}_{\WEGpd}}$} (A2_4);
    \path (A0_2) edge [right hook->]node [auto] {$\scriptstyle{}$} (A0_4);
    \path (A2_2) edge [right hook->]node [auto] {$\scriptstyle{}$} (A2_4);
\end{tikzpicture}\]
where each map without a name denotes an embedding as full $2$-subcategory or as full
bi-subcategory.
\end{prop}

\begin{proof}
By~\cite[Proposition~5.26]{MM}, if $\groupoidtot{X}{}$ and $\groupoidtot{Y}{}$ are Morita equivalent
Lie groupoids, then the first Lie groupoid is proper if and only if the second one is so. Moreover,
by~\cite[Example~5.21(2)]{MM} if $\groupoidtot{X}{}$ and $\groupoidtot{Y}{}$ are both \'etale and
they are Morita equivalent, then the first Lie groupoid is effective if and only if the second one is
so (note that being effective is not preserved by Morita equivalences if we remove the \'etale
condition).\\

Now axioms (\hyperref[BF1]{BF1}), (\hyperref[BF2]{BF2}) and (\hyperref[BF5]{BF5}) are easily verified
for the set $\WPEEGpd$. Let us consider (\hyperref[BF3]{BF3}), so let us fix any triple of proper,
effective, \'etale groupoids and any pair of morphisms as follows

\[\begin{tikzpicture}[xscale=2.5,yscale=-1.2]
    \node (A0_0) at (0, 0) {$\groupoidtot{X}{}$};
    \node (A0_1) at (1, 0) {$\groupoidtot{Y}{}$};
    \node (A0_2) at (2, 0) {$\groupoidtot{Z}{}$,};
    \path (A0_0) edge [->]node [auto] {$\scriptstyle{\groupoidmaptot{\psi}{}}$} (A0_1);
    \path (A0_2) edge [->]node [auto,swap] {$\scriptstyle{\groupoidmaptot{\phi}{}}$} (A0_1);
\end{tikzpicture}\]
with $\groupoidmaptot{\psi}{}$ Morita equivalence. By Proposition~\ref{prop-02} we get that
(\hyperref[BF3]{BF3}) holds in $\EGpd$; therefore there are an \'etale groupoid
$\groupoidtot{U}{}$, a Morita equivalence $\groupoidmaptot{\psi}{\prime}$ and a morphism
$\groupoidmaptot{\phi}{\prime}$ as follows

\[\begin{tikzpicture}[xscale=2.5,yscale=-1.2]
    \node (A0_0) at (0, 0) {$\groupoidtot{X}{}$};
    \node (A0_1) at (1, 0) {$\groupoidtot{U}{}$};
    \node (A0_2) at (2, 0) {$\groupoidtot{Z}{}$};
    
    \path (A0_1) edge [->]node [auto,swap] {$\scriptstyle{\groupoidmaptot{\phi}{\prime}}$} (A0_0);
    \path (A0_1) edge [->]node [auto] {$\scriptstyle{\groupoidmaptot{\psi}{\prime}}$} (A0_2);
\end{tikzpicture}\]
and a natural transformation $\alpha:\groupoidmaptot{\psi}{}\circ\groupoidmaptot{\phi}{\prime}
\Rightarrow\groupoidmaptot{\phi}{}
\circ\groupoidmaptot{\psi}{\prime}$. Now $\groupoidtot{Z}{}$ and $\groupoidtot{U}{}$ are \'etale
groupoids that are weakly equivalent and the first one is proper and effective; so also the second
one is proper and effective. Therefore axiom (\hyperref[BF3]{BF3}) holds for the set $\WPEEGpd$. An
analogous proof shows
that also (\hyperref[BF4]{BF4}) holds for $\WPEEGpd$. The proofs for the set $\WPEGpd$ are
analogous.
Therefore there are bicategories and
pseudofunctors as in the claim. The last part of the claim is straightforward by looking at the
explicit construction of the bicategories of fractions in~\cite[\S~2.2 and~2.3]{Pr}
and using the remarks at the beginning of this proof.
\end{proof}

The bicategory $\PEGpd\left[\WPEGpdinv\right]$ is usually called in the literature the
\emph{bicategory of orbifolds} (from the point of view of Lie groupoids); its bi-subcategory
$\PEEGpd\left[\WPEEGpdinv\right]$ is usually called the \emph{bicategory of effective \emphatic{(}or
reduced\emphatic{)} orbifolds}. We refer to Description~\ref{descrip-02} below for
an explicit description of the last bicategory mentioned above.

\section{Weak equivalences, unit weak equivalences and refinements in $\RedAtl$}
In this section we introduce the notions of weak equivalences, unit weak equivalences and
refinements in $\RedAtl$. Using the $2$-functor
$\functor{F}^{\red}$, the definition of weak equivalences will match with the notion of Morita
equivalences between proper, effective, \'etale groupoids (see Proposition~\ref{prop-04} below).

\begin{defin}\label{def-10}
Let us fix any pair of reduced orbifold atlases $\atlas{X}=\{\unifX{i}\}_{i\in I}$ on $X$ and
$\atlas{Y}$ on $Y$ and any morphism 

\begin{equation}\label{eq-33}
\atlasmap{\operatorname{w}}{}:=\Big(\operatorname{w},\overline{\operatorname{w}},\left\{
\widetilde{\operatorname{w}}_i\right\}_{i\in I},\left[P_{\operatorname{w}},\nu_{\operatorname{w}}
\right]\Big):\,\atlas{X}\longrightarrow\atlas{Y}.
\end{equation}

Then we say that $\atlasmap{\operatorname{w}}{}$ is a \emph{refinement} if and only if the following
$2$ conditions hold:

\begin{enumerate}[({REF}1)]
 \item\label{REF1} $X=Y$ and the continuous map $\operatorname{w}:X\rightarrow X$ is equal to $\id_X$;
 \item\label{REF2} for each $i\in I$ the smooth map $\widetilde{\operatorname{w}}_i$ is an open
   embedding; assuming (\hyperref[REF1]{REF1}), this implies that $[\atlas{X}]=[\atlas{Y}]$.
\end{enumerate}

We say that $\atlasmap{\operatorname{w}}{}$ is a \emph{unit weak equivalence of reduced orbifold atlases}
if and only if it satisfies condition (\hyperref[REF1]{REF1}) and

\begin{enumerate}[({UWE})]
 \item\label{UWE} for each $i\in I$ the chart $(\tX_i,G_i,\pi_i)$ on $X$ is
  compatible with the atlas $\atlas{Y}$; assuming (\hyperref[REF1]{REF1}), this is equivalent
  to imposing that $[\atlas{X}]=[\atlas{Y}]$.
\end{enumerate}

We say that $\atlasmap{\operatorname{w}}{}$ is a \emph{weak equivalence of reduced orbifold
atlases} if and only if it satisfies the following conditions:

\begin{enumerate}[({WE}1)]
 \item\label{WE1} the continuous map $\operatorname{w}:X\rightarrow Y$ is an homeomorphism;
 \item\label{WE2} for each $i\in I$ the chart $(\tX_i,G_i,\operatorname{w}\circ\pi_i)$ on
  $Y$ is compatible with the atlas $\atlas{Y}$; assuming (\hyperref[WE1]{WE1}),
  this is equivalent to imposing that
  $[\operatorname{w}_{\ast}(\atlas{X})]=[\atlas{Y}]$ (see Definition~\ref{def-13}).
\end{enumerate}

We say that $\atlasmap{\operatorname{w}}{}$ is an \emph{open embedding of reduced orbifold
atlases} if and only if it satisfies the following conditions:

\begin{enumerate}[({OE}1)]
 \item\label{OE1} the continuous map $\operatorname{w}:X\rightarrow Y$ is a
  topological open embedding;
 \item\label{OE2} for each $i\in I$ the chart $(\tX_i,G_i,\operatorname{w}\circ\pi_i)$
  on $Y$ is compatible with the atlas $\atlas{Y}$; assuming (\hyperref[OE1]{OE1}),
  this is equivalent to imposing that $[\operatorname{w}_{\ast}(\atlas{X})]=
  [\atlas{Y}|_{\operatorname{w}(X)}]$, where $\atlas{Y}|_{\operatorname{w}(X)}$ is the reduced
  orbifold atlas induced by $\atlas{Y}$ on the open set $\operatorname{w}(X)\subseteq Y$.
\end{enumerate}
\end{defin}

So we have the following chain of inclusions:

\begin{gather*}
\{\textrm{refinements}\}\subset\{\textrm{unit weak equivalences}\}\subset \\
\subset\{\textrm{weak equivalences}\}\subset\{\textrm{open embeddings}\}.
\end{gather*}

\begin{lem}\label{lem-13}
Let us fix any pair of reduced orbifold atlases

\begin{equation}\label{eq-12}
\atlas{X}=\left\{\left(\tX_i,G_i,\pi_i\right)\right\}_{i\in I},\quad\quad
\atlas{Y}=\left\{\left(\tY_j,H_j,\chi_j\right)\right\}_{j\in J}
\end{equation}
and any open embedding $\atlasmap{\operatorname{w}}{}$ as in \eqref{eq-33}.
Then for each $i\in I$ the smooth map $\widetilde{\operatorname{w}}_i:\tX_i\rightarrow
\tY_{\overline{\operatorname{w}}(i)}$ is \'etale \emphatic{(}i.e.\ a local diffeomorphism\emphatic{)}.
\end{lem}

\begin{proof}
Let us fix any $i\in I$ and any $\tx_i\in\tX_i$. By definition of morphism in $\RedAtl$, we have 

\begin{equation}\label{eq-14}
\operatorname{w}\circ\pi_i=\chi_{\overline{\operatorname{w}}(i)}\circ
\widetilde{\operatorname{w}}_i,
\end{equation}
so $\operatorname{w}\circ\pi_i(\tx_i)$ belongs to $\chi_{\overline{\operatorname{w}}(i)}
(\tY_{\overline{\operatorname{w}}(i)})$. By (\hyperref[OE2]{OE2}), the chart
$(\tX_i,G_i,
\operatorname{w}\circ\pi_i)$ is compatible with the atlas $\atlas{Y}$, so in particular it is
compatible with $(\tY_{\overline{\operatorname{w}}(i)},H_{\overline{\operatorname{w}}(i)},
\chi_{\overline{\operatorname{w}}(i)})$. Therefore there exists a change of charts $\lambda$ from
$(\tX_i,G_i,\operatorname{w}\circ\pi_i)$ to $(\tY_{\overline{\operatorname{w}}(i)},
H_{\overline{\operatorname{w}}(i)}$, $\chi_{\overline{\operatorname{w}}(i)})$, such that $\tx_i
\in\dom\lambda$. By Definition~\ref{def-14}, we have

\begin{equation}\label{eq-15}
\chi_{\overline{\operatorname{w}}(i)}\circ\lambda=\left.\operatorname{w}\circ\pi_i\right|_{\dom
\lambda}.
\end{equation}

Then let us consider the map

\begin{equation}\label{eq-36}
\overline{\lambda}:=\widetilde{\operatorname{w}}_i\circ\lambda^{-1}:\,\cod\lambda\longrightarrow
\tY_{\overline{\operatorname{w}}(i)}.
\end{equation}

For each $\ty\in\cod\lambda$ we have

\[\chi_{\overline{\operatorname{w}}(i)}\circ\overline{\lambda}(\ty)=
\chi_{\overline{\operatorname{w}}(i)}\circ\widetilde{\operatorname{w}}_i\circ\lambda^{-1}(\ty)
\stackrel{\eqref{eq-14}}{=}\operatorname{w}\circ\pi_i\circ\lambda^{-1}(\ty)
\stackrel{\eqref{eq-15}}{=}\chi_{\overline{\operatorname{w}}(i)}(\ty).\]

So for each $\ty$ as before, there exists a (in general non-unique) $h\in
H_{\overline{\operatorname{w}}(i)}$ such that $\overline{\lambda}(\ty)=h(\ty)$. Since $\cod\lambda$
is connected, then by~\cite[Lemma~2.11]{MM} there 
is a unique $\overh\in H_{\overline{\operatorname{w}}(i)}$
such that $\overline{\lambda}=\overh|_{\cod\lambda}$. Therefore,

\[\left.\widetilde{\operatorname{w}}_i\right|_{\dom\lambda}\stackrel{\eqref{eq-36}}{=}
\overline{\lambda}\circ\lambda=\overh\circ\lambda.\]

So we have proved that for each $i\in I$ the map $\widetilde{\operatorname{w}}_i$ coincides
locally with a diffeomorphism.
\end{proof}

\begin{rem}
The previous lemma shows that the morphisms called ``lift of the identity''
in~\cite[Definition~5.8]{Po} coincide with the unit weak equivalences defined before.
\end{rem}

\begin{lem}\label{lem-15}
Let us fix the following data:

\begin{enumerate}[\emphatic{(}a\emphatic{)}]
 \item a pair of reduced orbifold atlases $\atlas{X}:=\{\unifX{i}\}_{i\in I}$ over $X$ and
  $\atlas{Y}:=\{\unifY{j}\}_{j\in J}$ over $Y$;
 \item a topological open embedding $\operatorname{w}:X\hookrightarrow Y$;
 \item a set map $\overline{\operatorname{w}}:I\rightarrow J$;
 \item for each $i\in I$ an \'etale map $\widetilde{\operatorname{w}}_i:\tX_i\rightarrow
 \tY_{\overline{\operatorname{w}}(i)}$ such that $\chi_{\overline{\operatorname{w}}(i)}
  \circ\widetilde{\operatorname{w}}_i=\operatorname{w}\circ\pi_i$.
\end{enumerate}

Then there is a \emph{unique} class $[P_{\operatorname{w}},\nu_{\operatorname{w}}]$, such that the
collection $\atlasmap{\operatorname{w}}{}:=(\operatorname{w},\overline{\operatorname{w}},
\{\widetilde{\operatorname{w}}_i\}_{i\in I},$ $[P_{\operatorname{w}},\nu_{\operatorname{w}}])$ is a
morphism of reduced orbifold atlases. Moreover, in this case $\atlasmap{\operatorname{w}}{}$
is actually an open embedding of reduced orbifold atlases.
\end{lem}

\begin{rem}
Combining this with Lemma~\ref{lem-13}, this means that \emph{each open embedding
$\atlasmap{\operatorname{w}}{}$ is completely determined by an underlying topological open embedding
and by a collection of \'etale local liftings}. In particular, each refinement is completely
determined by a collection of open embeddings from each chart of $\atlas{X}$
to some charts of $\atlas{Y}$, commuting with the
projections.
\end{rem}

\begin{proof}[Proof of Lemma~\ref{lem-15}.]
Let us fix any pair $i,i'\in I$, any $\lambda\in\change(\atlas{X},i,i')$ and any point $\tx_i\in
\dom\lambda$. Since $\widetilde{\operatorname{w}}_i$ and $\widetilde{\operatorname{w}}_{i'}$ are
both \'etale, then there are open neighborhoods $\tX_i(\tx_i)$ of $\tx_i$ in $\dom\lambda$,
respectively $\tX_{i'}(\lambda(\tx_i))$ of
$\lambda(\tx_i)$ in $\cod\lambda$, where $\widetilde{\operatorname{w}}_i$, respectively
$\widetilde{\operatorname{w}}_{i'}$, is invertible. Up to restricting $\tX_i(\tx_i)$ to
$\lambda^{-1}(\tX_{i'}(\lambda(\tx_i)))$, we can assume that $\lambda(\tX_i(\tx_i))=
\tX_{i'}(\lambda(\tx_i))$. So it makes sense to consider the diffeomorphism

\[\nu_{\operatorname{w}}(\lambda,\tx_i):=\widetilde{\operatorname{w}}_{i'}\circ\lambda\circ
\widetilde{\operatorname{w}}_i^{-1}:\,\,\widetilde{\operatorname{w}}_i(\tX_i(\tx_i))
\longrightarrow\widetilde{\operatorname{w}}_{i'}(\tX_{i'}(\lambda(\tx_i))).\]

By (d), this is a change of charts of $\atlas{Y}$. Now for each $\lambda\in\change(\atlas{X},i,i')$
we choose a collection of points $\{\tx_{i,t}\}_{t\in T(\lambda)}\subset\dom\lambda$, such that
the family $\{\tX_i(\tx_{i,t})\}_{t\in T(\lambda)}$ covers $\dom\lambda$ (and such that 
$\tX_i(\tx_{i,t})\neq\tX_i(\tx_{i,t'})$ for each $t\neq t'$). Then we define

\[P_{\operatorname{w}}:=\left\{\lambda|_{\tX_i(\tx_{i,t})}\quad\forall\,i\in I,\,\,\,
\forall\,\lambda\in\change(\atlas{X},i,-),\,\,\,\forall\,t\in T(\lambda)\right\};\]
for each $\lambda|_{\tX_i(\tx_{i,t})}$ in such a set, we define $\nu_{\operatorname{w}}
(\lambda|_{\tX_i(\tx_{i,t})}):=\nu_{\operatorname{w}}(\lambda,\tx_{i,t})$. Then it is easy to see that
the class $[P_{\operatorname{w}},\nu_{\operatorname{w}}]$ is such that
$(\operatorname{w},\overline{\operatorname{w}},
\{\widetilde{\operatorname{w}}_i\}_{i\in I},[P_{\operatorname{w}},\nu_{\operatorname{w}}])$ is a
morphism of reduced orbifold atlases. The fact that the class
$[P_{\operatorname{w}},\nu_{\operatorname{w}}]$ is unique is a direct consequence of
(\hyperref[M5c]{M5c}) (see Definition~\ref{def-11}): given any $\lambda$ in $P_{\operatorname{w}}(i,i')$, since
$\widetilde{\operatorname{w}}_i$ is a diffeomorphism if restricted enough in source and target,
then for each $\tx_i\in\dom\lambda$ the value of $\nu_{\operatorname{w}}(\lambda)$ around
$\widetilde{\operatorname{w}}_i(\tx_i)$ is completely determined
by (\hyperref[M5c]{M5c}). In other terms, the class $[P_{\operatorname{w}},\nu_{\operatorname{w}}]$
is uniquely determined.
\end{proof}

\begin{lem}\label{lem-18}
Let us fix the following data:

\begin{itemize}
 \item a finite number of reduced orbifold atlases $\atlas{X}^1,\cdots,\atlas{X}^r$
   over a topological space $X$, all belonging to the same orbifold structure $[\atlas{X}]$;
 \item a reduced orbifold atlas $\atlas{X}'$ over a topological space $X'$;
 \item an homeomorphism $\operatorname{w}:X'\stackrel{\sim}{\rightarrow}X$;
 \item a collection of weak equivalences of reduced orbifold atlases $\atlasmap{\operatorname{w}}{m}:
  \atlas{X}'\rightarrow\atlas{X}^m$ for $m=1,\cdots,r$, all defined over $\operatorname{w}$.
\end{itemize}

Then there are a reduced orbifold atlas $\overline{\atlas{X}}$ over $X$ and a weak equivalence
$\atlasmap{\operatorname{v}}{}:\overline{\atlas{X}}\rightarrow\atlas{X}'$, such that:

\begin{itemize}
 \item $\atlasmap{\operatorname{v}}{}$ is defined over the homeomorphism $\operatorname{w}^{-1}:X
  \rightarrow X'$;
 \item each local lift of $\atlasmap{\operatorname{v}}{}$ is an open embedding;
 \item for each $m=1,\cdots,r$, $\atlasmap{\operatorname{w}}{m}\circ\atlasmap{\operatorname{v}}{}$ 
  is a \emph{refinement}.
\end{itemize}

In particular, if $X'=X$ and $\operatorname{w}=\id_X$, then also $\atlasmap{\operatorname{v}}{}$
is a refinement.
\end{lem}

\begin{proof}
Let us suppose that $\atlas{X}'=\{(\tX'_i,G'_i,\pi'_i)\}_{i\in I}$ and that

\[\atlasmap{\operatorname{w}}{m}:=\Big(\operatorname{w},\overline{\operatorname{w}}^m,
\{\widetilde{\operatorname{w}}^m_i\}_{i\in I},\left[P_{\operatorname{w}^m},\nu_{\operatorname{w}^m}
\right]\Big):\,\atlas{X}'\longrightarrow\atlas{X}^m\quad\textrm{for}\,\,m=1,\cdots,r.\]

By Lemma~\ref{lem-13}, for each $i\in I$ and for each $m=1,\cdots,r$ the smooth map
$\widetilde{\operatorname{w}}_i^m$ is a local diffeomorphism, so for each $i\in I$ there exist a
(non-unique) open covering $\{\tX'_{i,a}\}_{a\in A(i)}$ of $\tX'_i$, such that for each $m=1,\cdots,r$,
the map $\widetilde{\operatorname{w}}_i^m$ is an open embedding if restricted to any $\tX'_{i,a}$. For
each $i\in I$ and for each $a\in A(i)$ there exists a (non-unique) open covering
$\{\tX'_{i,a,b}\}_{b\in B(i,a)}$ of $\tX'_{i,a}$, such that each $\tX'_{i,a,b}$ is the domain of a
chart $(\tX'_{i,a,b},G'_{i,a,b},\pi'_i|_{\tX'_{i,a,b}})$ compatible with $\atlas{X}'$.
Then we define a reduced orbifold atlas on $X$ as follows:

\[\overline{\atlas{X}}:=\left\{\left(\tX'_{i,a,b},G'_{i,a,b},\operatorname{w}\circ
\pi'_i|_{\tX'_{i,a,b}}\right)\right\}_{i\in I,\,a\in A(i),\,b\in B(i,a)}.\]

Now we consider the set map $\overline{\operatorname{v}}$ sending each triple $(i,a,b)$ to $i$;
for each such triple we define $\widetilde{\operatorname{v}}_{i,a,b}$ as the inclusion of
$\tX'_{i,a,b}$ in $\tX'_i$. Now let us fix any change of charts

\[\lambda\in\change\left(\overline{\atlas{X}},(i,a,b),(i',a',b')\right).\]
By definition of change of charts, we have $\operatorname{w}\circ\pi_{i'}\circ
\lambda=\operatorname{w}\circ\pi_i$; since $\operatorname{w}$ is an homeomorphism, this implies
that $\pi_{i'}\circ\lambda=\pi_i$, so $\lambda$ can be considered as a
change of charts in $\change(\atlas{X}',i,i')$; we denote by 
$\nu_{\operatorname{v}}(\lambda)$ such a change of charts.
Then we get easily that the collection

\[\atlasmap{\operatorname{v}}{}:=\Big(\operatorname{w}^{-1},\overline{\operatorname{v}},
\{\widetilde{\operatorname{v}}_{i,a,b}\}_{i,a,b},\left[\change(\overline{\atlas{X}}),
\nu_{\operatorname{v}}\right]\Big):\,\overline{\atlas{X}}\longrightarrow\atlas{X}'\]
is a morphism of reduced orbifold atlases. Now for each triple $(i,a,b)$ and for
each $m=1,\cdots,r$ the morphism
$\widetilde{\operatorname{w}}^m_i\circ\widetilde{\operatorname{v}}_{i,a,b}$ is an open embedding
because by construction $\widetilde{\operatorname{w}}^m_i$ is an open embedding if restricted
to $\tX'_{i,a,b}\subseteq\tX'_{i,a}$.
Moreover, the morphism $\atlasmap{\operatorname{w}}{m}\circ\atlasmap{\operatorname{v}}{}$ is defined
over $\operatorname{w}\circ\operatorname{w}^{-1}=\id_X$. Therefore, 
$\atlasmap{\operatorname{w}}{m}\circ\atlasmap{\operatorname{v}}{}$ is a refinement
for each $m=1,\cdots,r$.
\end{proof}

The following $2$ lemmas are on the same line of~\cite[Propositions~5.3 and~6.2]{Po}; the
significant difference is given by the fact that we consider all the weak equivalences
instead of restricting only the unit weak equivalences considered in~\cite{Po}.

\begin{lem}\label{lem-09}
If $\atlasmap{\operatorname{w}}{}:\atlas{X}\rightarrow\atlas{Y}$ is a weak equivalence of reduced
orbifold atlases, then $\functor{F}^{\red}_1(\atlasmap{\operatorname{w}}{})$ is a Morita equivalence of
proper, effective, \'etale groupoids.
\end{lem}

\begin{proof}
Let us use the notations of \eqref{eq-33} and \eqref{eq-12} and let us set:

\begin{equation}\label{eq-13}
\functor{F}_0^{\red}(\atlas{X}):=\left(\groupname{X}{}\right),\quad\functor{F}_0^{\red}(\atlas{Y}):=
\left(\groupname{Y}{}\right),\quad\functor{F}_1^{\red}(\atlasmap{\operatorname{w}}{}):=
\groupoidmap{\psi}{}.
\end{equation}

By Lemma~\ref{lem-13} each $\widetilde{\operatorname{w}}_i$ is \'etale; therefore also $\psi_0=
\coprod_{i\in I}\widetilde{\operatorname{w}}_i$ is \'etale. So also the induced projection
(see property (\hyperref[ME1]{ME1}))

\[\pi^1:\,\fiber{\groupoid{Y}_1}{s}{\psi_0}{\groupoid{X}_0}\longrightarrow\groupoid{Y}_1\]
is \'etale. Since $t$ is also \'etale, we conclude that $t\circ\pi^1$ is \'etale, so in
particular it is a submersion. Let us prove that it is also surjective. Let us fix any point
$\ty_j\in\tY_j\subseteq\groupoid{Y}_0$; since $\operatorname{w}$ is an homeomorphism
by (\hyperref[WE1]{WE1}), then it makes
sense to define $x:=\operatorname{w}^{-1}(\chi_j(\ty_j))$. Let us choose any chart
$\unifX{i}$ in
$\atlas{X}$ and any $\tx_i\in\tX_i$, such that $\pi_i(\tx_i)=x$. Then by definition of morphism 
in $\RedAtl$ we have:

\[\chi_{\overline{\operatorname{w}}(i)}\circ\widetilde{\operatorname{w}}_i(\tx_i)=
\operatorname{w}\circ\pi_i(\tx_i)=\operatorname{w}(x)=\chi_j(\ty_j).\]

Since $\atlas{Y}$ is a reduced orbifold atlas, there exists a change of charts $\omega$ from
$(\tY_{\overline{\operatorname{w}}(i)},H_{\overline{\operatorname{w}}(i)}$,
$\chi_{\overline{\operatorname{w}}(i)})$ to $\unifY{j}$, such that $\widetilde{\operatorname{w}}_i
(\tx_i)\in\dom\omega$ and $\omega(\widetilde{\operatorname{w}}_i(\tx_i))=\ty_j$. Then we set
$p:=(\germ_{\widetilde{\operatorname{w}}_i(\tx_i)}\omega,\tx_i)$ and we get that $p$ belongs to
the fiber product $\fiber{\groupoid{Y}_1}{s}{\psi_0}{\groupoid{X}_0}$; moreover
$t\circ\pi^1(p)=\ty_j$. So we have proved that $t\circ\pi^1$ is surjective, so
(\hyperref[ME1]{ME1}) holds.\\

In order to prove that $\groupoidmap{\psi}{}$ is a Morita equivalence , we have also to
prove (\hyperref[ME2]{ME2}), i.e.\ we have to show
that the commutative square \eqref{eq-16} has the universal property of fiber products. So let
us fix any smooth manifold $M$ together with any pair of smooth maps $a=(a_1,a_2):M\rightarrow
\groupoid{X}_0\times\groupoid{X}_0$
and $b:M\rightarrow\groupoid{Y}_1$, such that $(s,t)\circ b=(\psi_0\times\psi_0)\circ
(a_1,a_2)$. We have to prove that there is a unique smooth map $c:M\rightarrow\groupoid{X}_1$, making
the following diagram commute:

\begin{equation}\label{eq-17}
\begin{tikzpicture}[xscale=1.5,yscale=-1.0]
    \node (A0_0) at (1, 1) {$M$};
    \node (A2_2) at (2, 2) {$\groupoid{X}_1$};
    \node (A2_4) at (4, 2) {$\groupoid{Y}_1$};
    \node (A4_2) at (2, 4) {$\groupoid{X}_0\times\groupoid{X}_0$};
    \node (A4_4) at (4, 4) {$\groupoid{Y}_0\times\groupoid{Y}_0$.};

    \node (A3_3) at (3, 3) {$\curvearrowright$};
    \node (A1_2) at (2.2, 1.6) {$\curvearrowright$};
    \node (A2_1) at (1.6, 2.2) {$\curvearrowright$};
    
    \path (A4_2) edge [->]node [auto] {$\scriptstyle{(\psi_0\times\psi_0)}$} (A4_4);
    \path (A0_0) edge [->,dashed]node [auto] {$\scriptstyle{c}$} (A2_2);
    \path (A2_2) edge [->]node [auto] {$\scriptstyle{\psi_1}$} (A2_4);
    \path (A2_4) edge [->]node [auto] {$\scriptstyle{(s,t)}$} (A4_4);
    \path (A0_0) edge [->,bend left=15]node [auto,swap] {$\scriptstyle{(a_1,a_2)}$} (A4_2);
    \path (A0_0) edge [->,bend right=15]node [auto] {$\scriptstyle{b}$} (A2_4);
    \path (A2_2) edge [->]node [auto] {$\scriptstyle{(s,t)}$} (A4_2);
\end{tikzpicture}
\end{equation}

We have already said that $\psi_0$ is \'etale; moreover by definition of \'etale groupoid we have
$\psi_0\circ s=s\circ\psi_1$ and the $2$-morphisms $s$ in the previous identity are \'etale; hence $\psi_1$
is \'etale. So if we fix any point $x_1\in\groupoid{X}_1$ and we set $x_0:=s(x_1)$ and $x'_0:=t(x_1)$,
we have:

\begin{gather*}
\kappa_{\groupoidtot{X}{}}(x_0,x'_0,x_1)\stackrel{\eqref{eq-47}}{=}
 \germ_{x_1}t\cdot(\germ_{x_1}s)^{-1}= \\
=\germ_{x_1}t\cdot(\germ_{x_1}\psi_1)^{-1}\cdot
  \germ_{x_1}\psi_1\cdot(\germ_{x_1}s)^{-1}= \\
=(\germ_{x'_0}\psi_0)^{-1}\cdot\germ_{\psi_1(x_1)}t\cdot
 (\germ_{\psi_1(x_1)}s)^{-1}\cdot\germ_{x_0}\psi_0\stackrel{\eqref{eq-47}}{=} \\
\stackrel{\eqref{eq-47}}{=}(\germ_{x'_0}\psi_0)^{-1}\cdot
 \kappa_{\groupoidtot{Y}{}}(\psi_0(x_0),\psi_0(x'_0),\psi_1(x_1))\cdot\germ_{x_0}\psi_0.
\end{gather*}

By Lemma~\ref{lem-21} the set map $\kappa_{\groupoidtot{X}{}}(x_0,x'_0,-)$ is a bijection, so:

\[x_1=\kappa_{\groupoidtot{X}{}}(x_0,x'_0,-)^{-1}\left((\germ_{x'_0}\psi_0)^{-1}\cdot
\kappa_{\groupoidtot{Y}{}}(\psi_0(x_0),\psi_0(x'_0),\psi_1(x_1))\cdot\germ_{x_0}\psi_0\right).\]

So for every point $m\in M$, if $c(m)$ exists making \eqref{eq-17} commute, 
then it is necessarily equal to

\[c(m)=\kappa_{\groupoidtot{X}{}}(a_1(m),a_2(m),-)^{-1}(g(m)),\]
where $g(m)$ is the germ defined as follows:

\[g(m):=(\germ_{a_2(m)}\psi_0)^{-1}\cdot
\kappa_{\groupoidtot{Y}{}}(\psi_0\circ a_1(m),\psi_0\circ a_2(m),b(m))\cdot\germ_{a_1(m)}\psi_0.\]

This proves the uniqueness of a morphism $c$ as in \eqref{eq-17}. Moreover, 
using the definition of $\kappa_{\groupoidtot{X}{}}(-,-,-)$, the previous lines actually give rise
to a well-defined set map $c:M\rightarrow\groupoid{X}_1$, making \eqref{eq-17} commute. In particular,
we have $s\circ c=a_1$; since $a_1$ is smooth by hypothesis and $s$ is \'etale, this implies that $c$
is smooth, so we have proved that property (\hyperref[ME2]{ME2}) holds for $\groupoidmap{\psi}{}$.
\end{proof}

\begin{lem}\label{lem-10}
Let $\atlas{X}$ and $\atlas{Y}$ be reduced orbifold atlases \emphatic{(}for $X$ and $Y$
respectively\emphatic{)}. Let $\groupoidmap{\psi}{}:\functor{F}^{\red}_0(\atlas{X})\rightarrow
\functor{F}^{\red}_0(\atlas{Y})$ be a Morita equivalence and let
$\atlasmap{\operatorname{w}}{}:\atlas{X}\rightarrow\atlas{Y}$ be the unique morphism such that
$\functor{F}^{\red}_1(\atlasmap{\operatorname{w}}{})=\groupoidmap{\psi}{}$ \emphatic{(}see
\emphatic{Lemma~\ref{lem-07}}\emphatic{)}. Then $\atlasmap{\operatorname{w}}{}$ is a weak equivalence
of reduced orbifold atlases.
\end{lem}

\begin{proof}
Let us use the notations of \eqref{eq-12} and \eqref{eq-13} and let us denote by 

\[\atlasmap{\operatorname{w}}{}=\Big(\operatorname{w},\overline{\operatorname{w}},\left\{
\widetilde{\operatorname{w}}_i\right\}_{i\in I},\left[P_{\operatorname{w}},\nu_{\operatorname{w}}
\right]\Big):\,\atlas{X}\longrightarrow\atlas{Y}\]
the unique morphism obtained from $\groupoidmaptot{\psi}{}$ using Lemma~\ref{lem-07}.
We recall that in the
proof of such lemma, we defined the continuous map $\operatorname{w}:X\rightarrow Y$ (denoted by $f$
in the mentioned lemma), as $\operatorname{w}:=\varphi_{\atlas{Y}}\circ|\groupoidmaptot{\psi}{}|
\circ\varphi_{\atlas{X}}^{-1}$, where both $\varphi_{\atlas{X}}$ and $\varphi_{\atlas{Y}}$
are homeomorphisms. Since $\groupoidmaptot{\psi}{}$ is a Morita equivalence of \'etale groupoids, then by
Lemma~\ref{lem-14} we have that also $|\groupoidmaptot{\psi}{}|$ is an homeomorphism, so property
(\hyperref[WE1]{WE1}) is satisfied.\\

In order to prove that $\atlasmap{\operatorname{w}}{}$ is a weak equivalence, we need also to prove
that for each $i\in I$ the chart $(\tX_i,G_i,\operatorname{w}\circ\pi_i)$ on $Y$ is compatible
with the atlas $\atlas{Y}$. So let us fix any index $j\in J$ and any pair of points $(\tx_i,\ty_j)\in\tX_i
\times\tY_j$ such that $\operatorname{w}\circ\pi_i(\tx_i)=\chi_j(\ty_j)$. Then we have

\[\chi_{\overline{\operatorname{w}}(i)}\circ\widetilde{\operatorname{w}}_i(\tx_i)=\operatorname{w}
\circ\pi_i(\tx_i)=\chi_j(\ty_j).\]

Since $\atlas{Y}$ is a reduced orbifold atlas, then
there exists a change of charts $\omega$ from $\unifY{\overline{\operatorname{w}}(i)}$ to $\unifY{j}$,
such that $\widetilde{\operatorname{w}}_i(\tx_i)\in\dom\omega$. Since $\omega$ is a change
of charts, then $\chi_j\circ\omega=\chi_{\overline{\operatorname{w}}(i)}$. Moreover, since $\psi_0$ is
\'etale (see~\cite[Exercise 5.16(4)]{MM}),
then the map $\widetilde{\operatorname{w}}_i=\psi_0|_{\tX_i}$
is locally a diffeomorphism. Therefore there exists
an open neighborhood $\tX$ of $\tx_i$, contained in $\widetilde{\operatorname{w}}_i^{-1}(\dom
\omega)$, such that $\widetilde{\operatorname{w}}_i$ is an embedding if restricted to $\tX$. Then
$\overline{\omega}:=\omega\circ\widetilde{\operatorname{w}}_i|_{\tX}$ is a smooth embedding from $\tX
\subseteq\tX_i$ to $\tY_j$. Up to restricting $\tX$ to a smaller neighborhood of $\tx_i$, we can
always assume that $\tX$ is the domain of a change of charts of $\atlas{X}$. Moreover, we have

\[\chi_j\circ\overline{\omega}=\chi_j\circ\omega\circ\widetilde{\operatorname{w}}_i|_{\tX}=
\chi_{\overline{\operatorname{w}}(i)}\circ\widetilde{\operatorname{w}}	_i|_{\tX}=
\operatorname{w}\circ\pi_i|_{\tX},\]
so $\overline{\omega}$ is a change of charts from $(\tX_i,G_i,\operatorname{w}\circ\pi_i)$
to $\unifY{j}$ with $\tx_i\in\dom\overline{\omega}$. So we have proved that the chart
$(\tX_i,G_i,\operatorname{w}\circ\pi_i)$ is compatible with $\atlas{Y}$ for every $i\in I$,
i.e.\ condition (\hyperref[WE2]{WE2}) holds.
\end{proof}

By combining Lemmas~\ref{lem-09} and~\ref{lem-10} we get:

\begin{prop}\label{prop-04}
Given any $2$ reduced orbifold atlases $\atlas{X},\atlas{Y}$, the bijection

\begin{gather*}
\{\textrm{morphisms}\,\,\atlasmap{f}{}:\atlas{X}\rightarrow\atlas{Y}\,\,\textrm{in}\,\,\RedAtl\}
\longrightarrow \\
\longrightarrow\{\textrm{morphisms}\,\,\groupoidmaptot{\phi}{}:\functor{F}^{\red}_0(\atlas{X})
\rightarrow\functor{F}^{\red}_0(\atlas{Y})\,\,\textrm{in}\,\,\PEEGpd\}
\end{gather*}
of \emphatic{Lemma~\ref{lem-07}} induces a bijection between weak equivalences of reduced
orbifold atlases and Morita equivalences of proper, effective, \'etale groupoids.
\end{prop}

Then we are able to compute the right saturation (see
Definition~\ref{def-12}) of the class $\WRedAtl$ of all refinements.

\begin{lem}\label{lem-19}
The right saturation $\WRedAtlsat$ is the class of all weak equivalences of reduced orbifold atlases.
\end{lem}

\begin{proof}
Let us fix any morphism $\atlasmap{f}{}:\atlas{Y}\rightarrow\atlas{X}$ in $\WRedAtlsat$. 
By Definition~\ref{def-12}, this
implies that there are a pair of reduced orbifold atlases
$\atlas{U},\atlas{Z}$ and a pair of morphisms
$\atlasmap{h}{}:\atlas{U}\rightarrow\atlas{Z}$ and $\atlasmap{g}{}:\,\atlas{Z}
\rightarrow\atlas{Y}$,
such that both $\atlasmap{f}{}\circ\atlasmap{g}{}$ and $\atlasmap{g}{}\circ\atlasmap{h}{}$ are
refinements (hence weak equivalences of reduced orbifold atlases).  So by
Proposition~\ref{prop-04} we have that both $\functor{F}_2^{\red}(\atlasmap{f}{})\circ
\functor{F}_2^{\red}
(\atlasmap{g}{})$ and $\functor{F}_2^{\red}(\atlasmap{g}{})\circ\functor{F}_2^{\red}(\atlasmap{h}{})$
are Morita equivalences of \'etale groupoids. In other terms, the morphism 
$\functor{F}_2^{\red}(\atlasmap{f}{})$ belongs to the right saturation of the class
$\WEGpd$ of Morita equivalences between \'etale groupoids. So by Lemma~\ref{lem-17}
we conclude that actually $\functor{F}_2^{\red}(\atlasmap{f}{})$ is a Morita equivalence. Again by
Proposition~\ref{prop-04}, this implies that $\atlasmap{f}{}$ is a weak equivalence of reduced
orbifold atlases. So we have proved
that $\WRedAtlsat$ is contained in the set of all weak equivalences of reduced orbifold atlases.
Conversely, let us suppose that $\atlasmap{f}{}:\atlas{Y}\rightarrow\atlas{X}$ is a
weak equivalence. Then by Lemma~\ref{lem-18} there are a reduced orbifold atlas $\atlas{Z}$
and a weak equivalence $\atlasmap{g}{}:\atlas{Z}\rightarrow\atlas{Y}$ such that
$\atlasmap{f}{}\circ\atlasmap{g}{}$ is a refinement. Applying 
Lemma~\ref{lem-18} a second time on $\atlasmap{g}{}$,
there are a reduced orbifold atlas $\atlas{U}$
and a weak equivalence $\atlasmap{h}{}:\atlas{U}\rightarrow\atlas{Z}$ such that
$\atlasmap{g}{}\circ\atlasmap{h}{}$ is a refinement.
Therefore, the morphism $\atlasmap{f}{}$ belongs to the right saturation
$\WRedAtlsat$. This suffices to conclude.
\end{proof}

\begin{lem}\label{lem-11}
Let us fix any proper, effective, \'etale groupoid $(\groupname{X}{})$. Then there are a reduced
orbifold atlas $\atlas{X}$ and a Morita equivalence $\groupoidmap{\psi}{}:
\functor{F}^{\red}_0(\atlas{X})\rightarrow(\groupname{X}{})$.
\end{lem}

\begin{proof}
Given $(\groupname{X}{})$, the reduced orbifold atlas $\atlas{X}$ is obtained as in the last part of the
proof of Theorem 4.1 in~\cite{MP}. In~\cite[Lemmas~4.7, 4.8 and 4.9]{T1} we proved that we can
define a weak equivalence as required. The proofs in~\cite{T1} were done in the category of complex
manifolds, but they can be easily adapted to the case of smooth manifolds, so we omit the details.
\end{proof}

\begin{lem}\label{lem-12}
Let us fix any pair of reduced orbifold atlases $\atlas{X},\atlas{Y}$ and any pair of open
embeddings $\atlasmap{\operatorname{w}}{m}:\atlas{X}\rightarrow\atlas{Y}$ for $m=1,2$. Then the
following facts are equivalent:

\begin{enumerate}[\emphatic{(}a\emphatic{)}]
 \item the underlying topological maps $\operatorname{w}^1$ and $\operatorname{w}^2$ coincide;
 \item there exists a $2$-morphism $[\delta]:\atlasmap{\operatorname{w}}{1}\Rightarrow
  \atlasmap{\operatorname{w}}{2}$ in $\RedAtl$;
 \item there exists a \emph{unique} $2$-morphism $[\delta]:\atlasmap{\operatorname{w}}{1}\Rightarrow
  \atlasmap{\operatorname{w}}{2}$ in $\RedAtl$.
\end{enumerate}

In particular, \emphatic{(}c\emphatic{)} holds if we consider any pair of refinements
\emphatic{(}or more generally, any pair of unit weak equivalences\emphatic{)}.
\end{lem}

\begin{proof}
Clearly (b) implies (a) by Definitions~\ref{def-05} and~\ref{def-06} of $2$-morphism of reduced
orbifold atlases, so
let us prove that (a) implies (b). Let

\[\atlas{X}:=\left\{\left(\tX_i,G_i,\pi_i\right)\right\}_{i\in I},\quad\atlas{Y}:=\left\{\left(
\tY_j,H_j,\chi_j\right)\right\}_{j\in J}\]
and let

\[\atlasmaprep{\operatorname{w}}{m}:=\Big(\operatorname{w}^m,\overline{\operatorname{w}}^m,
\{\widetilde{\operatorname{w}}^m_i\}_{i\in I},P_{\operatorname{w}^m},\nu_{\operatorname{w}^m}\Big)
\quad\textrm{for}\,\,m=1,2\]
be representatives for $\atlasmap{\operatorname{w}}{1}$ and $\atlasmap{\operatorname{w}}{2}$
respectively, with $\operatorname{w}^1=\operatorname{w}^2$. By Lemma~\ref{lem-13}, for each $i\in I$
both $\widetilde{\operatorname{w}}_i^1$ and $\widetilde{\operatorname{w}}_i^2$ are \'etale. Therefore
for each $i\in I$ there exists an open covering $\{\tX_i^a\}_{a\in A(i)}$ of $\tX_i$ such that
both $\widetilde{\operatorname{w}}_i^1$ and $\widetilde{\operatorname{w}}_i^2$ are diffeomorphisms
if restricted to any $\tX_i^a$. Up to replacing $\{\tX_i^a\}_{a\in A(i)}$ by a finer covering, we
can assume that each $\tX_i^a$ is the domain of a change of charts of $\atlas{X}$. Since both
$\atlasmap{\operatorname{w}}{1}$ and $\atlasmap{\operatorname{w}}{2}$ are morphisms in
$\RedAtl$, then for each $i\in I$ we have

\begin{equation}\label{eq-20}
\chi_{\overline{\operatorname{w}}^1(i)}\circ\widetilde{\operatorname{w}}_i^1=\operatorname{w}^1
\circ\pi_i\quad\textrm{and}\quad\chi_{\overline{\operatorname{w}}^2(i)}\circ
\widetilde{\operatorname{w}}_i^2=\operatorname{w}^2\circ\pi_i=\operatorname{w}^1\circ\pi_i.
\end{equation}

Then for each $i\in I$ and for each $a\in A(i)$ we set:

\[\delta_i^a:=\widetilde{\operatorname{w}}^2_i\circ\left(\widetilde{
\operatorname{w}}_i^1|_{\tX_i^a}\right)^{-1}.\]

By construction, each $\delta_i^a$ is a diffeomorphisms; moreover, by \eqref{eq-20} each
$\delta_i^a$ is a change of charts in $\change(\atlas{Y},\overline{\operatorname{w}}^1(i),
\overline{\operatorname{w}}^2(i))$. Then it is easy to see that the family $\delta:=\{(\tX_i^a,
\delta_i^a)\}_{i\in I,a\in A(i)}$ satisfies properties (\hyperref[2Ma]{2Ma}) -- (\hyperref[2Me]{2Me}), so
$[\delta]$ is a $2$-morphism from $\atlasmap{\operatorname{w}}{1}$ to
$\atlasmap{\operatorname{w}}{2}$. Therefore (a) implies (b).\\

Since (c) implies (b), we need only to prove the opposite implication. Let us suppose that there
exists another $2$-morphism $[\,\overline{\delta}\,]:\atlasmap{\operatorname{w}}{1}\Rightarrow
\atlasmap{\operatorname{w}}{2}$ with representative

\[\overline{\delta}=\left\{\left(\tX_i^{\overline{a}},\overline{\delta}^{\overline{a}}_i\right)
\right\}_{i\in I,\overline{a}\in\overline{A}(i)}.\]

Let us fix any $i\in I$ and any pair $(a,\overline{a})\in A(i)\times\overline{A}(i)$ such that
$\tX_i^a\cap\tX_i^{\overline{a}}\neq\varnothing$. Then by property (\hyperref[2Mc]{2Mc}) for $\delta$
and $\overline{\delta}$, we get that $\delta_i^a$ coincides with $\overline{\delta}_i^{\overline{a}}$
on the set $\widetilde{\operatorname{w}}_i^1(\tX_i^a\cap\tX_i^{\overline{a}})$; such a set is open
because $\widetilde{\operatorname{w}}_i^1$ is \'etale by Lemma~\ref{lem-13}.
Therefore, for each $\tx_i\in\tX_i^a\cap\tX_i^{\overline{a}}$ we have

\[\germ_{\widetilde{\operatorname{w}}^1_i(\tx_i)}\delta_i^a=
\germ_{\widetilde{\operatorname{w}}^1_i(\tx_i)}\overline{\delta}_i^{\operatorname{a}},\]
so $[\delta]=[\,\overline{\delta}\,]$ by Definition~\ref{def-06}.
\end{proof}

\section{The bicategories $\RedOrb$ and $\PEEGpd\left[\WPEEGpdinv\right]$}
In this section we
will prove that the pair $(\RedAtl,\WRedAtl)$ admits a right bicalculus of fractions
and we will give a simple description of the associated bicategory of fractions
$\RedOrb$. We will also recall briefly the description of the bicategory
$\PEEGpd\left[\WPEEGpdinv\right]$. In the next section we will prove that such 
$2$ bicategories are equivalent.

\begin{prop}\label{prop-05}
The pair $(\RedAtl,\WRedAtl)$ admits a right bicalculus of fractions, so there are
a bicategory $\RedOrb:=\RedAtl\left[\WRedAtlinv\right]$ and a pseudofunctor

\[\functor{U}_{\WRedAtl}:\,\RedAtl\longrightarrow\RedOrb\]
that sends every refinement of reduced orbifold atlases to an internal equivalence and that is
universal with respect to such a property \emphatic{(}see \emphatic{Remark~\ref{rem-03})}.
\end{prop}

\begin{proof}
Condition (\hyperref[BF1]{BF1}) is obviously satisfied and (\hyperref[BF2]{BF2}) is an easy
consequence of the definition of compositions (see Construction~\ref{cons-01}).
Let us
consider (\hyperref[BF3]{BF3}), so let us fix any triple of reduced orbifold atlases $\atlas{X},
\atlas{Y},\atlas{Z}$, any refinement $\atlasmap{\operatorname{w}}{}:\atlas{X}\rightarrow
\atlas{Y}$ and any morphism $\atlasmap{f}{}:\atlas{Z}\rightarrow\atlas{Y}$. Using
Lemmas~\ref{lem-09} and~\ref{lem-03} we
have that $\functor{F}^{\red}_1(\atlasmap{\operatorname{w}}{})$ is a Morita equivalence between proper,
effective and \'etale
groupoids; moreover we have proved in Proposition~\ref{prop-03} that the set $\WPEEGpd$ satisfies
(\hyperref[BF3]{BF3}). Therefore there exist a proper, effective, \'etale Lie groupoid
$\groupoidtot{U}{}$, a Morita equivalence $\groupoidmaptot{\psi}{\prime}$, a morphism
$\groupoidmaptot{\phi}{\prime}$ and a natural transformation $\alpha$ as follows:

\[
\begin{tikzpicture}[xscale=2.8,yscale=-0.8]
    \node (A0_1) at (1, 0) {$\groupoidtot{U}{}$};
    \node (A2_0) at (0, 2) {$\functor{F}^{\red}_0(\atlas{Z})$};
    \node (A2_1) at (1, 2) {$\functor{F}^{\red}_0(\atlas{Y})$};
    \node (A2_2) at (2, 2) {$\functor{F}^{\red}_0(\atlas{X})$.};

    \node (A1_1) at (1, 1.4) {$\Rightarrow$};
    \node (B1_1) at (1, 1) {$\alpha$};
    
    \path (A0_1) edge [->]node [auto,swap] {$\scriptstyle{\groupoidmaptot{\psi}{\prime}}$} (A2_0);
    \path (A2_2) edge [->]node [auto] {$\scriptstyle{\functor{F}_1^{\red}
      (\atlasmap{\operatorname{w}}{})}$} (A2_1);
    \path (A2_0) edge [->]node [auto,swap] {$\scriptstyle{\functor{F}_1^{\red}
      (\atlasmap{f}{})}$} (A2_1);
    \path (A0_1) edge [->]node [auto] {$\scriptstyle{\groupoidmaptot{\phi}{\prime}}$} (A2_2);
\end{tikzpicture}
\]

By Lemma~\ref{lem-11} there exist a reduced orbifold atlas $\atlas{U}$ and a Morita
equivalence $\groupoidmaptot{\psi}{\prime\prime}:\functor{F}^{\red}_0(\atlas{U})
\rightarrow\groupoidtot{U}{}$. Lemmas~\ref{lem-07} and~\ref{lem-08} prove that there are a pair
of morphisms $\atlasmap{\operatorname{v}}{}$, $\atlasmap{g}{}$ and a $2$-morphism $[\theta]$
in $\RedAtl$ as follows

\[
\begin{tikzpicture}[xscale=2.8,yscale=-0.8]
    \node (A0_1) at (1, 0) {$\atlas{U}$};
    \node (A2_0) at (0, 2) {$\atlas{Z}$};
    \node (A2_1) at (1, 2) {$\atlas{Y}$};
    \node (A2_2) at (2, 2) {$\atlas{X}$,};

    \node (A1_1) at (1, 1.4) {$\Rightarrow$};
    \node (B1_1) at (1, 0.9) {$[\theta]$};
    
    \path (A0_1) edge [->]node [auto,swap] {$\scriptstyle{\atlasmap{\operatorname{v}}{}}$} (A2_0);
    \path (A2_2) edge [->]node [auto] {$\scriptstyle{\atlasmap{\operatorname{w}}{}}$} (A2_1);
    \path (A2_0) edge [->]node [auto,swap] {$\scriptstyle{\atlasmap{f}{}}$} (A2_1);
    \path (A0_1) edge [->]node [auto] {$\scriptstyle{\atlasmap{g}{}}$} (A2_2);
\end{tikzpicture}
\]
such that the $2$-functor $\functor{F}^{\red}$ maps such a diagram to

\[
\begin{tikzpicture}[xscale=2.8,yscale=-0.8]
    \node (A0_1) at (1, 0) {$\functor{F}^{\red}_0(\atlas{U})$};
    \node (A2_0) at (0, 2) {$\functor{F}^{\red}_0(\atlas{Z})$};
    \node (A2_1) at (1, 2) {$\functor{F}^{\red}_0(\atlas{Y})$};
    \node (A2_2) at (2, 2) {$\functor{F}^{\red}_0(\atlas{X})$.};

    \node (A1_1) at (1, 1.4) {$\Rightarrow$};
    \node (B1_1) at (1, 1) {$\alpha\ast i_{\groupoidmaptot{\psi}{\prime\prime}}$};
    
    \path (A0_1) edge [->]node [auto,swap] {$\scriptstyle{\groupoidmaptot{\psi}{\prime}
      \circ\groupoidmaptot{\psi}{\prime\prime}}$} (A2_0);
    \path (A2_2) edge [->]node [auto] {$\scriptstyle{\functor{F}_1^{\red}
      (\atlasmap{\operatorname{w}}{})}$} (A2_1);
    \path (A2_0) edge [->]node [auto,swap] {$\scriptstyle{\functor{F}_1^{\red}
      (\atlasmap{f}{})}$} (A2_1);
    \path (A0_1) edge [->]node [auto] {$\scriptstyle{\groupoidmaptot{\phi}{\prime}
      \circ\groupoidmaptot{\psi}{\prime\prime}}$} (A2_2);
\end{tikzpicture}
\]

Now $\groupoidmaptot{\psi}{\prime}\circ\groupoidmaptot{\psi}{\prime\prime}$ is a Morita equivalence
(see condition (\hyperref[BF2]{BF2}) for the class $\WEGpd$).
So by Lemma~\ref{lem-10} we have that
$\atlasmap{\operatorname{v}}{}$ is a weak equivalence of reduced orbifold atlases. By
Lemma~\ref{lem-18} there is a reduced orbifold atlas $\atlas{V}$ and a weak equivalence
$\atlasmap{\operatorname{u}}{}:\atlas{V}\rightarrow\atlas{U}$, such that 
$\atlasmap{\operatorname{v}}{}\circ\atlasmap{\operatorname{u}}{}$ is a refinement.
Then we set

\[\atlasmap{\operatorname{w}}{\prime}:=\atlasmap{\operatorname{v}}{}\circ
\atlasmap{\operatorname{u}}{},\quad\quad\atlasmap{f}{\prime}:=\atlasmap{g}{}\circ
\atlasmap{\operatorname{u}}{},\quad\quad[\delta]:=
[\theta]\ast i_{\atlasmap{\operatorname{u}}{}}.\]

By Lemma~\ref{lem-01}(a) we get that $[\delta]$ is an invertible $2$-morphism, so the data $(\atlas{V},
\atlasmap{\operatorname{w}}{\prime},\atlasmap{f}{\prime},$ $[\delta])$ prove that
(\hyperref[BF3]{BF3}) holds for $\WRedAtl$.\\

The proof that (\hyperref[BF4]{BF4}) holds follows the same ideas described for (\hyperref[BF3]{BF3}),
so we omit it. Lastly, let us prove condition (\hyperref[BF5]{BF5}), so let us fix any pair of
reduced orbifold atlases

\[\atlas{X}:=\Big\{\Big(\tX_i,G_i,\pi_i\Big)\Big\}_{i\in I},\quad\atlas{Y}:=\Big\{\Big(
\tY_j,H_j,\chi_j\Big)\Big\}_{j\in J}\]
over $X$ and $Y$ respectively, any pair of morphisms

\[\atlasmap{\operatorname{w}}{m}:=\Big(\operatorname{w}^m,\overline{\operatorname{w}}^m,
\{\widetilde{\operatorname{w}}^m_i\}_{i\in I},\left[P_{\operatorname{w}^m},
\nu_{\operatorname{w}^m}\right]\Big):\,\,\atlas{X}\longrightarrow\atlas{Y},\quad m=1,2\]
and any $2$-morphism

\[[\delta]:=\left[\left\{\left(\tX_i^a,\delta_i^a\right)\right\}_{i\in I,a\in A(i)}\right]:
\atlasmap{\operatorname{w}}{1}\Longrightarrow
\atlasmap{\operatorname{w}}{2}\]
in $\RedAtl$. Moreover, let us suppose that 
$\atlasmap{\operatorname{w}}{2}$ is a refinement. This implies that $X=Y$,
$\operatorname{w}^2=\id_X$ and that every smooth map $\widetilde{\operatorname{w}}^2_i$ is an
open embedding. Now let us fix any $i\in I$, any $a\in A(i)$ and any
point $\tx_i$ in the open set $\tX_i^a$. By (\hyperref[2Mc]{2Mc}) we have that

\[\widetilde{\operatorname{w}}^1_i(\tx_i)=(\delta_i^a)^{-1}\circ\widetilde{\operatorname{w}}^2_i
(\tx_i)\]
so $\widetilde{\operatorname{w}}^1_i$ locally coincides with
an open embedding, hence $\widetilde{\operatorname{w}}^1_i$ is an \'etale map. Again by 
(\hyperref[2Mc]{2Mc}) we get

\[\tx_i=(\widetilde{\operatorname{w}}^2_i)^{-1}\circ\delta_i^a\circ
\widetilde{\operatorname{w}}^1_i(\tx_i).\]

If we fix any other index $a'\in A(i)$ and any other point $\tx'_i\in\tX_i^{a'}$, then we have also

\[\tx'_i=(\widetilde{\operatorname{w}}^2_i)^{-1}\circ\delta_i^{a'}\circ
\widetilde{\operatorname{w}}^1_i(\tx'_i).\]

Therefore, if $\widetilde{\operatorname{w}}^1_i(\tx_i)=\widetilde{\operatorname{w}}^1_i(\tx'_i)$,
then $\tx_i=\tx'_i$, i.e.\ $\widetilde{\operatorname{w}}^1_i$ is injective. So we conclude that
the smooth map $\widetilde{\operatorname{w}}^1_i$ is an open embedding for each $i\in I$,
i.e.\ condition (\hyperref[REF2]{REF2}) holds for $\atlasmap{\operatorname{w}}{1}$.\\

By Lemma~\ref{lem-12} the topological map $\operatorname{w}^1$ coincides with $\operatorname{w}^2=
\id_X$, so (\hyperref[REF1]{REF1}) holds. Therefore, we have proved that
(\hyperref[BF5]{BF5}) holds for $\WRedAtl$.
\end{proof}

\begin{rem}
In a similar way we can prove that also the pairs

\begin{enumerate}[(a)]
 \item $\RedAtl$ together with the class of all unit weak equivalences,
 \item $\RedAtl$ together with the class of all weak equivalences
\end{enumerate}
satisfy axioms (\hyperref[BF1]{BF1}) --
(\hyperref[BF5]{BF5}) (actually, (b) satisfies also the stronger axiom
(\hyperref[BF1prime]{BF1})$'$, see Remark~\ref{rem-03}), so a bicalculus of fractions exists also
for (a) and (b), and the resulting bicategories of fractions are equivalent to $\RedOrb$
(this is an easy consequence of~\cite[Proposition~2.10]{T3} and Lemma~\ref{lem-19}).
We prefer to use the class $\WRedAtl$ of refinements
instead of (a) or (b) because this leads to a bicategory $\RedOrb$ that is easier
to describe and that is more close to the geometric intuition (see below for details).
\end{rem}

\begin{lem}\label{lem-24}
The set $\WRedAtl$ satisfies a ``weak $2$-out-of-$3$-property'', i.e.\ given any pair of morphisms
$\atlasmap{f}{}:\atlas{X}\rightarrow\atlas{Y}$ and $\atlasmap{g}{}:\atlas{Y}\rightarrow\atlas{Z}$
in $\RedAtl$, we have:

\begin{enumerate}[\emphatic{(}i\emphatic{)}]
 \item if $\atlasmap{f}{}$ and $\atlasmap{g}{}$ belong to $\WRedAtl$, then so does
  $\atlasmap{g}{}\circ\atlasmap{f}{}$;
 \item if $\atlasmap{g}{}$ and $\atlasmap{g}{}\circ\atlasmap{f}{}$ belong to $\WRedAtl$, then so does
  $\atlasmap{f}{}$.
\end{enumerate}

The sets of unit weak equivalences and of all weak equivalences satisfy both
the \emphatic{(}strong\emphatic{)} ``\,$2$-out-of-$3$'' property.
\end{lem}

Note that if both $\atlasmap{f}{}$ and $\atlasmap{g}{}\circ\atlasmap{f}{}$ belong to $\WRedAtl$, then
in general it is not true that $\atlasmap{g}{}$ belong to $\WRedAtl$ (in general one can only prove
that $\atlasmap{g}{}$ is a unit weak equivalence).

\begin{proof}
(i) is simply (\hyperref[BF2]{BF2}) for the class $\WRedAtl$ (see Proposition~\ref{prop-05}). For
(ii), let us suppose that $\atlasmap{f}{}$ and $\atlasmap{g}{}$ are as in \eqref{eq-45}. Since
$\atlasmap{g}{}$ is a refinement, then $Y=Z$, $g=\id_Z$ and $\tg_j$ is an open embedding for each
$j\in J$. Since $\atlasmap{g}{}\circ\atlasmap{f}{}$ is a refinement,
then $X=Z$, $g\circ f=\id_Z$ and
$\tg_{\overline{f}(i)}\circ\tf_i$ is an open embedding for each $i\in I$. From this we get
that $f=\id_X$ and that $\tf_i$ is an open embedding for each $i\in I$, i.e.\ $\atlasmap{f}{}$
is a refinement.\\

The last part of the statement can be easily verified directly. For the class of all weak
equivalences, one can also prove it using Lemmas~\ref{lem-17} and~\ref{lem-19}.
\end{proof}

\begin{cor}\label{cor-01}
Let us fix any triple of reduced orbifold atlases $\atlas{X},\atlas{Y},\atlas{Z}$ over \emph{the
same topological space} $X$ and any pair of refinements

\[\begin{tikzpicture}[xscale=2.4,yscale=-1.2]
    \node (A0_0) at (0, 0) {$\atlas{X}$};
    \node (A0_1) at (1, 0) {$\atlas{Y}$};
    \node (A0_2) at (2, 0) {$\atlas{Z}$.};
    \path (A0_0) edge [->]node [auto] {$\scriptstyle{\atlasmap{\operatorname{w}}{1}}$} (A0_1);
    \path (A0_2) edge [->]node [auto,swap] {$\scriptstyle{\atlasmap{\operatorname{w}}{2}}$} (A0_1);
\end{tikzpicture}\]

Then there exists data in $\RedAtl$ as in the following diagram

\begin{equation}\label{eq-34}
\begin{tikzpicture}[xscale=2.4,yscale=-0.8]
    \node (A0_1) at (1, 0) {$\atlas{U}$};
    \node (A2_0) at (0, 2) {$\atlas{X}$};
    \node (A2_1) at (1, 2) {$\atlas{Y}$};
    \node (A2_2) at (2, 2) {$\atlas{Z}$.};

    \node (A1_1) at (1, 1.4) {$\Rightarrow$};
    \node (B1_1) at (1, 0.9) {$[\delta]$};
    
    \path (A0_1) edge [->]node [auto,swap] {$\scriptstyle{\atlasmap{\operatorname{v}}{1}}$} (A2_0);
    \path (A2_2) edge [->]node [auto] {$\scriptstyle{\atlasmap{\operatorname{w}}{2}}$} (A2_1);
    \path (A2_0) edge [->]node [auto,swap] {$\scriptstyle{\atlasmap{\operatorname{w}}{1}}$} (A2_1);
    \path (A0_1) edge [->]node [auto] {$\scriptstyle{\atlasmap{\operatorname{v}}{2}}$} (A2_2);
\end{tikzpicture}
\end{equation}
such that also $\atlasmap{\operatorname{v}}{1}$ and $\atlasmap{\operatorname{v}}{2}$ are refinements.
\end{cor}

\begin{proof}
Let us apply (\hyperref[BF3]{BF3}) for $\WRedAtl$ to the pair $(\atlasmap{\operatorname{w}}{1},
\atlasmap{\operatorname{w}}{2})$ (where $\atlasmap{\operatorname{w}}{1}$ is simply considered as a
morphism of reduced orbifold atlases). Then we get a diagram as \eqref{eq-34}, with 
$\atlasmap{\operatorname{v}}{1}$ refinement. Using (\hyperref[BF2]{BF2}) for $(\RedAtl,\WRedAtl)$
we get that $\atlasmap{\operatorname{w}}{1}\circ\atlasmap{\operatorname{v}}{1}$ is a refinement,
hence by (\hyperref[BF5]{BF5}) (applied to $[\delta]^{-1}$) we get that also
$\atlasmap{\operatorname{w}}{2}\circ\atlasmap{\operatorname{v}}{2}$ is a refinement. By
Lemma~\ref{lem-24}(ii), this implies that also $\atlasmap{\operatorname{v}}{2}$ is a refinement.
\end{proof}

\begin{cor}\label{cor-02}
Given any pair of reduced orbifold atlases $\atlas{X}^1$ and $\atlas{X}^2$ \emph{over the same
topological space $X$}, the following facts are equivalent:

\begin{enumerate}[\emphatic{(}1\emphatic{)}]
 \item there are a topological space $Y$, a reduced orbifold atlas $\atlas{X}$ over $Y$ and
  weak equivalences

  \begin{equation}\label{eq-35}
  \begin{tikzpicture}[xscale=2.4,yscale=-1.2]
    \node (A0_0) at (0, 0) {$\atlas{X}^1$};
    \node (A0_1) at (1, 0) {$\atlas{X}$};
    \node (A0_2) at (2, 0) {$\atlas{X}^2$,};
    \path (A0_1) edge [->]node [auto,swap] {$\scriptstyle{\atlasmap{\operatorname{w}}{1}}$} (A0_0);
    \path (A0_1) edge [->]node [auto] {$\scriptstyle{\atlasmap{\operatorname{w}}{2}}$} (A0_2);
  \end{tikzpicture}
  \end{equation}
  such that the underlying topological maps $\operatorname{w}^1$ and $\operatorname{w}^2$ 
  coincide;
 \item there are a reduced orbifold atlas $\atlas{X}$ over $X$ and \emph{refinements}
  as in \eqref{eq-35};
 \item $\atlas{X}^1$ is equivalent to $\atlas{X}^2$.
\end{enumerate}
\end{cor}

\begin{proof}
Let us suppose that (1) holds; then (2) holds by Lemma~\ref{lem-18}. If (2) holds, then by
(\hyperref[REF2]{REF2}) we have $[\atlas{X}^1]=
[\atlas{X}]=[\atlas{X}^2]$, so (3) holds.\\

Now let us assume that (3) holds and let us prove (1). So let us assume that
$\atlas{X}^1$ and $\atlas{X}^2$ are equivalent atlases. Therefore, for $m=1,2$ we can
consider the inclusion $\iota_{\atlas{X}^m}$ of $\atlas{X}^m$ in the common maximal
atlas $\atlas{X}^{\max}$ (see Definition~\ref{def-04}). Both maps are refinements,
therefore by Corollary~\ref{cor-01} there are a reduced orbifold atlas $\atlas{X}$ and a pair of
refinements (hence, of weak equivalences) as in \eqref{eq-35}, so (1) holds.
\end{proof}

\begin{rem}
For the explicit description of the bicategory of fractions $\RedOrb$ and of the pseudofunctor
$\functor{U}_{\WRedAtl}$, we refer mainly to the original construction in~\cite{Pr} or to
our previous paper~\cite{T3}, where we have explained how to simplify the construction
of associators and compositions of $2$-morphisms in any bicategory of fractions.
As it is stated in~\cite{Pr}, bicategories
of fractions are unique only up to weak equivalences of bicategories. In order to describe
explicitly one such bicategory, one has to make some choices as in the following description.
By~\cite[Theorem~21]{Pr}, \emph{different choices will give equivalent bicategories of
fractions where objects, $1$-morphisms and $2$-morphisms are the same, but compositions of
$1$-morphisms and $2$-morphisms are \emphatic{(}possibly\emphatic{)} different.}
\end{rem}

\begin{descrip}\label{descrip-01}
Following~\cite[\S~2.2, 2.3 and~2.4]{Pr}, the bicategory $\RedOrb$ and
the pseudofunctor $\functor{U}_{\WRedAtl}$ can be described as follows.

\begin{itemize}
 \item The \textbf{objects} of $\RedOrb$ are exactly the objects of $\RedAtl$, i.e.\ all the
  reduced orbifold atlases according to Definition~\ref{def-01}.
 \item Given any pair of reduced orbifold atlases $\atlas{X},\atlas{Y}$, the \textbf{$1$-morphisms}
  in $\RedOrb$ from the first atlas to the second one consist of all the triples $(\atlas{X}',
  \atlasmap{\operatorname{w}}{},\atlasmap{f}{})$ where $\atlas{X}'$ is any reduced orbifold atlas,
  $\atlasmap{\operatorname{w}}{}$ is any \emph{refinement} (see Definition~\ref{def-10})
  and $\atlasmap{f}{}$ is any morphism of reduced orbifold atlases (see Definition~\ref{def-03}), as
  follows

  \[
  \begin{tikzpicture}[xscale=2.4,yscale=-1.2]
    \node (A0_0) at (0, 0) {$\atlas{X}$};
    \node (A0_1) at (1, 0) {$\atlas{X}'$};
    \node (A0_2) at (2, 0) {$\atlas{Y}$};
    
    \path (A0_1) edge [->]node [auto,swap] {$\scriptstyle{\atlasmap{\operatorname{w}}{}}$} (A0_0);
    \path (A0_1) edge [->]node [auto] {$\scriptstyle{\atlasmap{f}{}}$} (A0_2);
  \end{tikzpicture}
  \]
  (in particular, using Corollary~\ref{cor-02}, we have that $\atlas{X}'$ \emph{is
  equivalent to} $\atlas{X}$). In other terms, a morphism in $\RedOrb$ from $\atlas{X}$ to $\atlas{Y}$
  consists firstly in replacing $\atlas{X}$ with a ``refined'' atlas $\atlas{X}'$ (\emph{keeping track of
  the refinement}), then by considering a usual morphism of $\RedAtl$ from $\atlas{X}'$ to $\atlas{Y}$.
  
 \item Given any pair of objects $\atlas{X},\atlas{Y}$ and any pair of morphisms
  $(\atlas{X}^m,\atlasmap{\operatorname{w}}{m},\atlasmap{f}{m}):\atlas{X}\rightarrow\atlas{Y}$ for
  $m=1,2$, according to the construction of bicategories of fractions in~\cite[\S~2.3]{Pr}, a
  \textbf{$2$-morphism} from $(\atlas{X}^1,\atlasmap{\operatorname{w}}{1},\atlasmap{f}{1})$
  to $(\atlas{X}^2,\atlasmap{\operatorname{w}}{2},\atlasmap{f}{2})$ is an equivalence class of data
  $(\atlas{X}^3,\atlasmap{\operatorname{v}}{1},\atlasmap{\operatorname{v}}{2},[\mu],[\delta])$ in
  $\RedAtl$ as follows:
  
  \begin{equation}\label{eq-51}
  \begin{tikzpicture}[xscale=1.8,yscale=-0.8]
    \node (A2_0) at (0, 2) {$\atlas{X}$};
    \node (A0_2) at (2, 0) {$\atlas{X}^1$};
    \node (A2_2) at (2, 2) {$\atlas{X}^3$};
    \node (A2_4) at (4, 2) {$\atlas{Y}$,};
    \node (A4_2) at (2, 4) {$\atlas{X}^2$};
    
    \node (A2_3) at (2.8, 2) {$\Downarrow\,[\delta]$};
    \node (A2_1) at (1.2, 2) {$\Downarrow\,[\mu]$};

    \path (A4_2) edge [->]node [auto] {$\scriptstyle{\atlasmap{\operatorname{w}}{2}}$} (A2_0);
    \path (A0_2) edge [->]node [auto,swap] {$\scriptstyle{\atlasmap{\operatorname{w}}{1}}$} (A2_0);
    \path (A4_2) edge [->]node [auto,swap] {$\scriptstyle{\atlasmap{f}{2}}$} (A2_4);
    \path (A0_2) edge [->]node [auto] {$\scriptstyle{\atlasmap{f}{1}}$} (A2_4);
    \path (A2_2) edge [->]node [auto,swap] {$\scriptstyle{\atlasmap{\operatorname{v}}{1}}$} (A0_2);
    \path (A2_2) edge [->]node [auto] {$\scriptstyle{\atlasmap{\operatorname{v}}{2}}$} (A4_2);
  \end{tikzpicture}
  \end{equation}
  such that $\atlasmap{\operatorname{w}}{1}\circ\atlasmap{\operatorname{v}}{1}$ is a refinement
  (in~\cite{Pr} it is also required that $[\mu]$ is invertible, but this property is automatically
  satisfied by Lemma~\ref{lem-01}). Any other set of data
  
  \[
  \begin{tikzpicture}[xscale=1.8,yscale=-0.8]
    \node (A2_0) at (0, 2) {$\atlas{X}$};
    \node (A0_2) at (2, 0) {$\atlas{X}^1$};
    \node (A2_2) at (2, 2) {$\atlas{X}^{\prime 3}$};
    \node (A2_4) at (4, 2) {$\atlas{Y}$,};
    \node (A4_2) at (2, 4) {$\atlas{X}^2$};
    
    \node (A2_3) at (2.8, 2) {$\Downarrow\,[\delta']$};
    \node (A2_1) at (1.2, 2) {$\Downarrow\,[\mu']$};

    \path (A4_2) edge [->]node [auto] {$\scriptstyle{\atlasmap{\operatorname{w}}{2}}$} (A2_0);
    \path (A0_2) edge [->]node [auto,swap] {$\scriptstyle{\atlasmap{\operatorname{w}}{1}}$} (A2_0);
    \path (A4_2) edge [->]node [auto,swap] {$\scriptstyle{\atlasmap{f}{2}}$} (A2_4);
    \path (A0_2) edge [->]node [auto] {$\scriptstyle{\atlasmap{f}{1}}$} (A2_4);
    \path (A2_2) edge [->]node [auto,swap] {$\scriptstyle{\atlasmap{\operatorname{v}}{\prime 1}}$} (A0_2);
    \path (A2_2) edge [->]node [auto]
      {$\scriptstyle{\atlasmap{\operatorname{v}}{\prime 2}}$} (A4_2);
  \end{tikzpicture}
  \]
  (with $\atlasmap{\operatorname{w}}{1}\circ\atlasmap{\operatorname{v}}{\prime 1}$ refinement)
  represents the same $2$-morphism in $\RedOrb$ if and only if there is a set of data $(\atlas{X}^4,
  \atlasmap{\operatorname{z}}{},\atlasmap{\operatorname{z}}{\prime},[\sigma^1],[\sigma^2])$ as
  follows:
    
  \begin{equation}\label{eq-21}
  \begin{tikzpicture}[xscale=1.8,yscale=-0.8]

    \node (A0_4) at (4, 0) {$\atlas{X}^1$};
    \node (A2_2) at (3, 2) {$\atlas{X}^{\prime 3}$};
    \node (A2_4) at (4, 2) {$\atlas{X}^4$};
    \node (A2_6) at (5, 2) {$\atlas{X}^3$,};
    \node (A4_4) at (4, 4) {$\atlas{X}^2$};
    
    \node (A3_4) at (4, 3.1) {$\Leftarrow$};
    \node (B3_4) at (4, 2.6) {$[\sigma^2]$};
    \node (A1_4) at (4, 0.9) {$\Rightarrow$};
    \node (B1_4) at (4, 1.4) {$[\sigma^1]$};
    
    \path (A2_2) edge [->]node [auto,swap]
      {$\scriptstyle{\atlasmap{\operatorname{v}}{\prime 2}}$} (A4_4);
    \path (A2_4) edge [->]node [auto] {$\scriptstyle{\atlasmap{\operatorname{z}}{}}$} (A2_6);
    \path (A2_6) edge [->]node [auto] {$\scriptstyle{\atlasmap{\operatorname{v}}{2}}$} (A4_4);
    \path (A2_2) edge [->]node [auto] {$\scriptstyle{\atlasmap{\operatorname{v}}{\prime 1}}$} (A0_4);
    \path (A2_4) edge [->]node [auto,swap]
      {$\scriptstyle{\atlasmap{\operatorname{z}}{\prime}}$} (A2_2);
    \path (A2_6) edge [->]node [auto,swap] {$\scriptstyle{\atlasmap{\operatorname{v}}{1}}$} (A0_4);
  \end{tikzpicture}
  \end{equation}
  
  such that $\atlasmap{\operatorname{w}}{1}\circ\atlasmap{\operatorname{v}}{1}\circ
  \atlasmap{\operatorname{z}}{}$ is a refinement,
  
  \begin{equation}\label{eq-23}
  \Big(i_{\atlasmap{f}{2}}\ast[\sigma^2]\Big)\odot\Big([\delta]\ast i_{\atlasmap{\operatorname{z}}{}}
  \Big)\odot\Big(i_{\atlasmap{f}{1}}\ast[\sigma^1]\Big)=[\delta']\ast
  i_{\atlasmap{\operatorname{z}}{\prime}}
  \end{equation}
  and
  
  \begin{equation}\label{eq-53}
  \Big(i_{\atlasmap{\operatorname{w}}{2}}\ast[\sigma^2]\Big)\odot\Big([\mu]\ast
  i_{\atlasmap{\operatorname{z}}{}}\Big)\odot\Big(i_{\atlasmap{\operatorname{w}}{1}}
  \ast[\sigma^1]\Big)=[\mu']\ast i_{\atlasmap{\operatorname{z}}{\prime}}.
  \end{equation}
  
  We denote by
 
  \begin{equation}\label{eq-28}
  \Big[\atlas{X}^3,\atlasmap{\operatorname{v}}{1},\atlasmap{\operatorname{v}}{2},[\mu],[\delta]
  \Big]:\Big(\atlas{X}^1,\atlasmap{\operatorname{w}}{1},\atlasmap{f}{1}\Big)\Longrightarrow\Big(
  \atlas{X}^2,\atlasmap{\operatorname{w}}{2},\atlasmap{f}{2}\Big)
  \end{equation}
  the class of any such data (we refer to Lemma~\ref{lem-25} below for a slightly simplified
  description of $2$-morphisms).
  
 \item For the \textbf{composition of $1$-morphisms} in $\RedOrb$ we have to do a preliminary
  step as follows: for every pair of morphisms in $\RedAtl$

  \begin{equation}\label{eq-27}
  \begin{tikzpicture}[xscale=1.5,yscale=-1.2]
    \node (A0_0) at (0, 0) {$\atlas{X}'$};
    \node (A0_1) at (1, 0) {$\atlas{Y}$};
    \node (A0_2) at (2, 0) {$\atlas{Y}'$};
    \path (A0_0) edge [->]node [auto] {$\scriptstyle{\atlasmap{f}{}}$} (A0_1);
    \path (A0_2) edge [->]node [auto,swap] {$\scriptstyle{\atlasmap{\operatorname{v}}{}}$} (A0_1);
  \end{tikzpicture}
  \end{equation}
  with $\atlasmap{\operatorname{v}}{}$ refinement, using the axiom of choice
  we \emph{choose} any reduced orbifold atlas
  $\atlas{X}''$, any pair of morphisms $\atlasmap{\operatorname{v}}{\prime}$, $\atlasmap{f}{\prime}$
  in $\RedAtl$ with $\atlasmap{\operatorname{v}}{\prime}$ refinement and any $2$-morphism $[\delta]$
  in $\RedAtl$ as follows

  \begin{equation}\label{eq-32}
  \begin{tikzpicture}[xscale=1.5,yscale=-0.8]
    \node (A0_1) at (1, 0) {$\atlas{X}''$};
    \node (A1_0) at (0, 2) {$\atlas{X}'$};
    \node (A1_2) at (2, 2) {$\atlas{Y}'$.};
    \node (A2_1) at (1, 2) {$\atlas{Y}$};

    \node (A1_1) at (1, 0.9) {$[\delta]$};
    \node (B1_1) at (1, 1.4) {$\Rightarrow$};
    
    \path (A1_2) edge [->]node [auto] {$\scriptstyle{\atlasmap{\operatorname{v}}{}}$} (A2_1);
    \path (A0_1) edge [->]node [auto] {$\scriptstyle{\atlasmap{f}{\prime}}$} (A1_2);
    \path (A1_0) edge [->]node [auto,swap] {$\scriptstyle{\atlasmap{f}{}}$} (A2_1);
    \path (A0_1) edge [->]node [auto,swap]
      {$\scriptstyle{\atlasmap{\operatorname{v}}{\prime}}$} (A1_0);
  \end{tikzpicture} 
  \end{equation}
  Such a choice is always possible by (\hyperref[BF3]{BF3}) (see
  Proposition~\ref{prop-05}) but in general it is not unique. By~\cite[\S~2.2]{Pr} we only have to
  force such a choice in the following special cases:

  \begin{enumerate}[(a)]
  \item whenever \eqref{eq-27} is such that $\atlas{Y}=\atlas{X}'$ and $\atlasmap{f}{}=
   \id_{\atlas{Y}}$, then we have to choose $\atlas{X}'':=
   \atlas{Y}'$, $\atlasmap{f}{\prime}:=\id_{\atlas{Y}'}$, $\atlasmap{\operatorname{v}}{\prime}:=
   \atlasmap{\operatorname{v}}{}$ and $[\delta]:=i_{\atlasmap{\operatorname{v}}{}}$;
  \item whenever \eqref{eq-27} is such that $\atlas{Y}=\atlas{Y}'$ and
   $\atlasmap{\operatorname{v}}{}=\id_{\atlas{Y}}$, then we have to choose 
   $\atlas{X}'':=\atlas{X}'$, $\atlasmap{f}{\prime}:=\atlasmap{f}{}$,
   $\atlasmap{\operatorname{v}}{\prime}:=\id_{\atlas{X}'}$ and $[\delta]:=i_{\atlasmap{f}{}}$.
  \end{enumerate}
  
  Having fixed any set of such choices, given any pair of morphisms in $\RedOrb$ as follows:

  \[\begin{tikzpicture}[xscale=1.5,yscale=-1.2]
    \node (A0_0) at (0, 0) {$\atlas{X}$};
    \node (A0_1) at (1, 0) {$\atlas{X}'$};
    \node (A0_2) at (2, 0) {$\atlas{Y}$};
    \path (A0_1) edge [->]node [auto,swap] {$\scriptstyle{\atlasmap{\operatorname{w}}{}}$} (A0_0);
    \path (A0_1) edge [->]node [auto] {$\scriptstyle{\atlasmap{f}{}}$} (A0_2);
    \node (B0_0) at (3, 0) {$\atlas{Y}$};
    \node (B0_1) at (4, 0) {$\atlas{Y}'$};
    \node (B0_2) at (5, 0) {$\atlas{Z}$};
    \path (B0_1) edge [->]node [auto,swap] {$\scriptstyle{\atlasmap{\operatorname{v}}{}}$} (B0_0);
    \path (B0_1) edge [->]node [auto] {$\scriptstyle{\atlasmap{g}{}}$} (B0_2);
  \end{tikzpicture}\]
  (with both $\atlasmap{\operatorname{w}}{}$ and $\atlasmap{\operatorname{v}}{}$ refinements),
  we use the fixed choice \eqref{eq-32} and we set

  \[\Big(\atlas{Y}',\atlasmap{\operatorname{v}}{},\atlasmap{g}{}\Big)\circ\Big(\atlas{X}',
  \atlasmap{\operatorname{w}}{},\atlasmap{f}{}\Big):=\Big(\atlas{X}'',\atlasmap{\operatorname{w}}{}
  \circ\atlasmap{\operatorname{v}}{\prime},\atlasmap{g}{}\circ\atlasmap{f}{\prime}\Big):\atlas{X}
  \longrightarrow\atlas{Z}.\]
  In this way in general the composition of morphisms in $\RedOrb$ is associative only up to
  canonical $2$-isomorphisms, so $\RedOrb$ is a bicategory but not a $2$-category.
 \item We omit the construction of the \textbf{vertical and horizontal compositions for
  $2$-morphisms} (for details we refer to the original constructions in any bicategory of fractions, as
  described in~\cite[\S~2.3]{Pr}, or to the simplified version given in our
  previous paper~\cite{T3}).
  A priori the construction of such
  compositions depends on some additional choices involving axiom (\hyperref[BF4]{BF4});
  by~\cite[Theorem~0.5]{T3} actually
  the choices of \eqref{eq-32} completely determine all the structure of $\RedOrb$.
  We only remark that since each $2$-morphism is invertible
  in $\RedAtl$, then it is not difficult to prove that the same property holds in $\RedOrb$.
  In particular, the inverse of any $2$-morphism as in \eqref{eq-28} is given by $[\atlas{X}^3,
  \atlasmap{\operatorname{v}}{2},\atlasmap{\operatorname{v}}{1},[\mu]^{-1},[\delta]^{-1}]$.
 \item The pseudofunctor $\atlas{U}_{\WRedAtl}$ sends each reduced orbifold atlas $\atlas{X}$ to
  the same object in $\RedOrb$. For every morphism $\atlasmap{f}{}:\atlas{X}\rightarrow\atlas{Y}$
  we have $\atlas{U}_{\WRedAtl}(\atlasmap{f}{})=(\atlas{X},\id_{\atlas{X}},\atlasmap{f}{})$. For
  every pair of morphisms $\atlasmap{f}{m}:\atlas{X}\rightarrow\atlas{Y}$ for $m=1,2$ and for
  every $2$-morphism $[\delta]:\atlasmap{f}{1}\Rightarrow\atlasmap{f}{2}$ in $\RedAtl$ we have
  
  \[\atlas{U}_{\WRedAtl}([\delta])=\Big[\atlas{X},\id_{\atlas{X}},\id_{\atlas{X}},i_{\id_{\atlas{X}}},
  [\delta]\Big]:\,\Big(\atlas{X},\id_{\atlas{X}},\atlasmap{f}{1}\Big)\Longrightarrow
  \Big(\atlas{X},\id_{\atlas{X}},\atlasmap{f}{2}\Big).\]
\end{itemize}
\end{descrip}
  
As we mentioned above, we can simplify a bit the description of $2$-morphisms in $\RedOrb$
as follows.

\begin{lem}\label{lem-25}
Let us fix any pair of reduced orbifold atlases $\atlas{X},\atlas{Y}$ and any pair of morphisms
$\underline{f}^m:=(\atlas{X}^m,\atlasmap{\operatorname{w}}{m},\atlasmap{f}{m}):\atlas{X}
\rightarrow\atlas{Y}$ in $\RedOrb$ for $m=1,2$. Then any $2$-morphism from $\underline{f}^1$ to
$\underline{f}^2$ in $\RedOrb$ is completely determined by a set of data as follows:

\begin{enumerate}[\emphatic{(}a\emphatic{)}]
 \item a reduced orbifold atlas $\atlas{X}^3$,
 \item a pair of refinements $\atlasmap{\operatorname{v}}{m}:\atlas{X}^3\rightarrow\atlas{X}^m$
   for $m=1,2$,
 \item a $2$-morphism $[\delta]$ in $\RedAtl$ as follows:

  \begin{equation}\label{eq-22}
  \begin{tikzpicture}[xscale=1.8,yscale=-0.8]
    \node (A0_2) at (2, 0) {$\atlas{X}^1$};
    \node (A2_2) at (2, 2) {$\atlas{X}^3$};
    \node (A2_4) at (4, 2) {$\atlas{Y}$.};
    \node (A4_2) at (2, 4) {$\atlas{X}^2$};

    \node (A2_3) at (2.8, 2) {$\Downarrow\,[\delta]$};

    \path (A4_2) edge [->]node [auto,swap] {$\scriptstyle{\atlasmap{f}{2}}$} (A2_4);
    \path (A0_2) edge [->]node [auto] {$\scriptstyle{\atlasmap{f}{1}}$} (A2_4);
    \path (A2_2) edge [->]node [auto] {$\scriptstyle{\atlasmap{\operatorname{v}}{1}}$} (A0_2);
    \path (A2_2) edge [->]node [auto,swap] {$\scriptstyle{\atlasmap{\operatorname{v}}{2}}$} (A4_2);
  \end{tikzpicture}
  \end{equation}
\end{enumerate}

Moreover, any set of data as above determines a $2$-morphism in $\RedOrb$; any other set of data 

\begin{equation}\label{eq-46}
\begin{tikzpicture}[xscale=1.8,yscale=-0.8]
    \node (A0_2) at (2, 0) {$\atlas{X}^1$};
    \node (A2_2) at (2, 2) {$\atlas{X}^{\prime 3}$};
    \node (A2_4) at (4, 2) {$\atlas{Y}$,};
    \node (A4_2) at (2, 4) {$\atlas{X}^2$};
    
    \node (A2_3) at (2.8, 2) {$\Downarrow\,[\delta']$};

    \path (A4_2) edge [->]node [auto,swap] {$\scriptstyle{\atlasmap{f}{2}}$} (A2_4);
    \path (A0_2) edge [->]node [auto] {$\scriptstyle{\atlasmap{f}{1}}$} (A2_4);
    \path (A2_2) edge [->]node [auto] {$\scriptstyle{\atlasmap{\operatorname{v}}{\prime 1}}$} (A0_2);
    \path (A2_2) edge [->]node [auto,swap]
      {$\scriptstyle{\atlasmap{\operatorname{v}}{\prime 2}}$} (A4_2);
\end{tikzpicture}
\end{equation}
\emphatic{(}with $\atlasmap{\operatorname{v}}{\prime 1}$ and $\atlasmap{\operatorname{v}}{\prime 2}$
refinements\emphatic{)}
determines the same $2$-morphism as \eqref{eq-22} if and only if there are a
reduced orbifold atlas $\atlas{X}^4$ and a pair of refinements 
  
\begin{equation}\label{eq-48}
\begin{tikzpicture}[xscale=2.4,yscale=-1.2]
    \node (A0_0) at (0, 0) {$\atlas{X}^{\prime 3}$};
    \node (A0_1) at (1, 0) {$\atlas{X}^4$};
    \node (A0_2) at (2, 0) {$\atlas{X}^3$,};
    
    \path (A0_1) edge [->]node [auto,swap]
      {$\scriptstyle{\atlasmap{\operatorname{z}}{\prime}}$} (A0_0);
    \path (A0_1) edge [->]node [auto] {$\scriptstyle{\atlasmap{\operatorname{z}}{}}$} (A0_2);
\end{tikzpicture}
\end{equation}
such that

\begin{equation}\label{eq-52}
\Big(i_{\atlasmap{f}{2}}\ast[\sigma^2]\Big)\odot\Big([\delta]\ast i_{\atlasmap{\operatorname{z}}{}}
\Big)\odot\Big(i_{\atlasmap{f}{1}}\ast[\sigma^1]\Big)=[\delta']\ast
i_{\atlasmap{\operatorname{z}}{\prime}},
\end{equation}
where $[\sigma^1],[\sigma^2]$ are the unique $2$-morphisms filling diagram \eqref{eq-21}
\emphatic{(}existence and uniqueness are a consequence of \emphatic{Lemma~\ref{lem-12})}.
Therefore, from now on we will denote each $2$-morphism in $\RedOrb$ from $\underline{f}^1$
to $\underline{f}^2$ as a class $[\atlas{X}^3,\atlasmap{\operatorname{v}}{1},
\atlasmap{\operatorname{v}}{2},[\delta]]$,
where the class of equivalence is the one induced by saying that

\[\Big(\atlas{X}^3,\atlasmap{\operatorname{v}}{1},\atlasmap{\operatorname{v}}{2},[\delta]\Big)
\sim\Big(\atlas{X}^{\prime 3},\atlasmap{\operatorname{v}}{\prime 1},
\atlasmap{\operatorname{v}}{\prime 2},[\delta']\Big)\]
if and only if there are data $(\atlas{X}^4,\atlasmap{\operatorname{z}}{},
\atlasmap{\operatorname{z}}{\prime})$ as above, such that \eqref{eq-52} holds.
\end{lem}

\begin{proof}
Let us fix any $2$-morphism in $\RedOrb$ represented by a set of data as in \eqref{eq-51}. Since
each $\underline{f}^m$ is a morphism in $\RedOrb$, then $\atlasmap{\operatorname{w}}{1}$
and $\atlasmap{\operatorname{w}}{2}$
are refinements. Moreover also $\atlasmap{\operatorname{w}}{1}\circ\atlasmap{\operatorname{v}}{1}$
is a refinement,
so by Lemma~\ref{lem-24} we have that $\atlasmap{\operatorname{v}}{1}$ is a refinement. Using
(\hyperref[BF5]{BF5}) for $\RedAtl$ on $[\mu]^{-1}$, also $\atlasmap{\operatorname{w}}{2}\circ
\atlasmap{\operatorname{v}}{2}$ is also a refinement, so again by Lemma~\ref{lem-24} we conclude that
$\atlasmap{\operatorname{v}}{2}$ is a refinement, so we have obtained a set of data as in (a) -- (c).
Conversely, let us fix any set of data as in
(a) -- (c). Then by (\hyperref[BF2]{BF2}) for $\RedAtl$ 
we get that $\atlasmap{\operatorname{w}}{m}\circ\atlasmap{\operatorname{v}}{m}$
is a refinement for each $m=1,2$. So by Lemma~\ref{lem-12} there is a unique $2$-morphism

\[[\mu]:\,\,\atlasmap{\operatorname{w}}{1}\circ\atlasmap{\operatorname{v}}{1}\Longrightarrow
\atlasmap{\operatorname{w}}{2}\circ\atlasmap{\operatorname{v}}{2}.\]

This proves that each set of data (a) -- (c) completely determines a set of data as in \eqref{eq-51},
hence a $2$-morphism from $\underline{f}^1$ to $\underline{f}^2$ in $\RedOrb$.\\

Now let us fix any other set of data as in \eqref{eq-46} and let us denote by 

\[[\mu']:\,\,\atlasmap{\operatorname{w}}{1}\circ\atlasmap{\operatorname{v}}{\prime 1}\Longrightarrow
\atlasmap{\operatorname{w}}{2}\circ\atlasmap{\operatorname{v}}{\prime 2}\]
the unique $2$-morphism in $\RedAtl$ determined by Lemma~\ref{lem-12}. By
Description~\ref{descrip-01}, the $2$-morphisms

\[\Big[\atlas{X}^3,\atlasmap{\operatorname{v}}{1},
\atlasmap{\operatorname{v}}{2},[\mu],[\delta]\Big]\quad\textrm{and}\quad\Big[\atlas{X}^{\prime 3},
\atlasmap{\operatorname{v}}{\prime 1},\atlasmap{\operatorname{v}}{\prime 2},[\mu'],[\delta']\Big]\]
coincide if and only if there are data $(\atlas{X}^4,\atlasmap{\operatorname{z}}{},
\atlasmap{\operatorname{z}}{\prime},[\sigma^1],[\sigma^2])$ as in \eqref{eq-21}, such that 
$\atlasmap{\operatorname{w}}{1}\circ\atlasmap{\operatorname{v}}{1}\circ\atlasmap{\operatorname{z}}{}$
is a refinement and \eqref{eq-23} and \eqref{eq-53} hold. Since 
$\atlasmap{\operatorname{w}}{1}\circ\atlasmap{\operatorname{v}}{1}$ is a refinement, by
Lemma~\ref{lem-24}
we conclude that $\atlasmap{\operatorname{z}}{}$ is a refinement. Using (\hyperref[BF5]{BF5})
for $\RedAtl$ on $[\sigma^1]$, we conclude that also $\atlasmap{\operatorname{v}}{\prime 1}
\circ\atlasmap{\operatorname{z}}{\prime}$ is a refinement. Again by Lemma~\ref{lem-24}, this implies
that $\atlasmap{\operatorname{z}}{\prime}$ is a refinement.\\

Conversely, let us suppose that there are a reduced orbifold atlas $\atlas{X}^4$ and a pair
of refinements $\atlasmap{\operatorname{z}}{}$ and $\atlasmap{\operatorname{z}}{\prime}$ as in
\eqref{eq-48}, such that \eqref{eq-52} holds. Then \eqref{eq-53} is automatically satisfied:
indeed it is an equality between $2$-morphisms having sources and target given by refinements, and
Lemma~\ref{lem-12} applies. This suffices to conclude.
\end{proof}

\begin{descrip}\label{descrip-02}
A description analogous to Description~\ref{descrip-01} holds for the bicategory $\PEEGpd\left[
\WPEEGpdinv\right]$ and for the pseudofunctor $\functor{U}_{\WPEEGpd}$ defined in \eqref{eq-11}.
The objects of that bicategory
are proper, effective, \'etale Lie groupoids; given $\groupoidtot{X}{}$ and
$\groupoidtot{Y}{}$, a morphism from the first object to the second one is given by any data
as follows, with $\groupoidmaptot{\psi}{}$ Morita equivalence:

\[\begin{tikzpicture}[xscale=2.9,yscale=-1.2]
    \node (A0_0) at (0, 0) {$\groupoidtot{X}{}$};
    \node (A0_1) at (1, 0) {$\groupoidtot{X}{\prime}$};
    \node (A0_2) at (2, 0) {$\groupoidtot{Y}{}$.};
    
    \path (A0_1) edge [->]node [auto,swap] {$\scriptstyle{\groupoidmaptot{\psi}{}}$} (A0_0);
    \path (A0_1) edge [->]node [auto] {$\scriptstyle{\groupoidmaptot{\phi}{}}$} (A0_2);
\end{tikzpicture}\]

Given any pair of morphisms in $\PEEGpd\left[\WPEEGpdinv\right]$

\[\Big(\groupoidtot{X}{m},\groupoidmaptot{\psi}{m},\groupoidmaptot{\phi}{m}\Big):\,
\groupoidtot{X}{}\longrightarrow\groupoidtot{Y}{}\quad\textrm{for}\,\,m=1,2,\]
using~\cite[Lemma~8.1]{PS} a
$2$-morphism from the first morphism to the second one is any equivalence class of data
$(\groupoidtot{X}{3},\groupoidmaptot{\xi}{1},\groupoidmaptot{\xi}{2},\mu,\delta)$ in
$\PEEGpd$ as follows

\begin{equation}\label{eq-50}
\begin{tikzpicture}[xscale=1.8,yscale=-0.8]
    \node (A2_0) at (0, 2) {$\groupoidtot{X}{}$};
    \node (A0_2) at (2, 0) {$\groupoidtot{X}{1}$};
    \node (A2_2) at (2, 2) {$\groupoidtot{X}{3}$};
    \node (A2_4) at (4, 2) {$\groupoidtot{Y}{}$,};
    \node (A4_2) at (2, 4) {$\groupoidtot{X}{2}$};

    \node (A2_3) at (2.8, 2) {$\Downarrow\,\delta$};
    \node (A2_1) at (1.2, 2) {$\Downarrow\,\mu$};

    \path (A4_2) edge [->]node [auto,swap] {$\scriptstyle{\groupoidmaptot{\phi}{2}}$} (A2_4);
    \path (A0_2) edge [->]node [auto] {$\scriptstyle{\groupoidmaptot{\phi}{1}}$} (A2_4);
    \path (A2_2) edge [->]node [auto,swap] {$\scriptstyle{\groupoidmaptot{\xi}{1}}$} (A0_2);
    \path (A2_2) edge [->]node [auto] {$\scriptstyle{\groupoidmaptot{\xi}{2}}$} (A4_2);
    \path (A0_2) edge [->]node [auto,swap] {$\scriptstyle{\groupoidmaptot{\psi}{1}}$} (A2_0);
    \path (A4_2) edge [->]node [auto] {$\scriptstyle{\groupoidmaptot{\psi}{2}}$} (A2_0);
\end{tikzpicture}
\end{equation}
such that both $\groupoidmaptot{\xi}{1}$ and $\groupoidmaptot{\xi}{2}$ are Morita equivalences
(a priori we should also impose that $\mu$ is invertible, but this is always verified since
each natural transformation is invertible in $\LieGpd$). The equivalence relation on the set of data of the form
$(\groupoidtot{X}{3},\groupoidmaptot{\xi}{1},\groupoidmaptot{\xi}{2},\mu,\delta)$ is analogous to the
one given in Description~\ref{descrip-01}, so we omit it.
\end{descrip}

In the next pages we will also need the following simplified description of $2$-morphisms in
$\PEEGpd\left[\WPEEGpdinv\right]$.

\begin{lem}\label{lem-16}
Let us fix any $2$ proper, effective, \'etale groupoids $\groupoidtot{X}{},
\groupoidtot{Y}{}$ and any pair of morphisms
$\underline{g}^m:=(\groupoidtot{X}{m},\groupoidmaptot{\psi}{m},\groupoidmaptot{\phi}{m}):
\groupoidtot{X}{}
\rightarrow\groupoidtot{Y}{}$ in $\PEEGpd\left[\WPEEGpdinv\right]$ for $m=1,2$.
Then any $2$-morphism from $\underline{g}^1$ to
$\underline{g}^2$ in $\PEEGpd\left[\WPEEGpdinv\right]$
is completely determined by a set of data as follows:

\begin{enumerate}[\emphatic{(}a\emphatic{)}]
 \item a proper, effective and \'etale groupoid $\groupoidtot{X}{3}$,
 \item a pair of Morita equivalences $\groupoidmaptot{\xi}{m}:\groupoidtot{X}{3}
  \rightarrow\groupoidtot{X}{m}$ for $m=1,2$, such that $|\groupoidmaptot{\psi}{1}|\circ
  |\groupoidmaptot{\xi}{1}|=|\groupoidmaptot{\psi}{2}|\circ
  |\groupoidmaptot{\xi}{2}|$ \emphatic{(}see \emphatic{Remark~\ref{rem-06}}\emphatic{)},
 \item a $2$-morphism $\alpha$ in $\PEEGpd$ as follows:

  \begin{equation}\label{eq-44}
  \begin{tikzpicture}[xscale=1.8,yscale=-0.8]
    \node (A0_2) at (2, 0) {$\groupoidtot{X}{1}$};
    \node (A2_2) at (2, 2) {$\groupoidtot{X}{3}$};
    \node (A2_4) at (4, 2) {$\groupoidtot{Y}{}$.};
    \node (A4_2) at (2, 4) {$\groupoidtot{X}{}$};

    \node (A2_3) at (2.8, 2) {$\Downarrow\,\alpha$};

    \path (A4_2) edge [->]node [auto,swap] {$\scriptstyle{\groupoidmaptot{\phi}{2}}$} (A2_4);
    \path (A0_2) edge [->]node [auto] {$\scriptstyle{\groupoidmaptot{\phi}{1}}$} (A2_4);
    \path (A2_2) edge [->]node [auto] {$\scriptstyle{\groupoidmaptot{\xi}{1}}$} (A0_2);
    \path (A2_2) edge [->]node [auto,swap] {$\scriptstyle{\groupoidmaptot{\xi}{2}}$} (A4_2);
  \end{tikzpicture}
  \end{equation}
\end{enumerate}

Moreover, any set of data as above determines a $2$-morphism in the bicategory $\PEEGpd\left[\WPEEGpdinv
\right]$; any other set of data 

  \[
  \begin{tikzpicture}[xscale=1.8,yscale=-0.8]
    \node (A0_2) at (2, 0) {$\groupoidtot{X}{1}$};
    \node (A2_2) at (2, 2) {$\groupoidtot{X}{\prime 3}$};
    \node (A2_4) at (4, 2) {$\groupoidtot{Y}{}$.};
    \node (A4_2) at (2, 4) {$\groupoidtot{X}{}$};

    \node (A2_3) at (2.8, 2) {$\Downarrow\,\alpha'$};

    \path (A4_2) edge [->]node [auto,swap] {$\scriptstyle{\groupoidmaptot{\phi}{2}}$} (A2_4);
    \path (A0_2) edge [->]node [auto] {$\scriptstyle{\groupoidmaptot{\phi}{1}}$} (A2_4);
    \path (A2_2) edge [->]node [auto] {$\scriptstyle{\groupoidmaptot{\xi}{\prime 1}}$} (A0_2);
    \path (A2_2) edge [->]node [auto,swap] {$\scriptstyle{\groupoidmaptot{\xi}{\prime 2}}$} (A4_2);
  \end{tikzpicture}
  \]
\emphatic{(}with $\groupoidmaptot{\xi}{\prime 1}$ and $\groupoidmaptot{\xi}{\prime 2}$
Morita equivalences such that
$|\groupoidmaptot{\psi}{1}|\circ
|\groupoidmaptot{\xi}{\prime 1}|=|\groupoidmaptot{\psi}{2}|\circ
|\groupoidmaptot{\xi}{\prime 2}|$\emphatic{)}
determines the same $2$-morphism as \eqref{eq-44} if and only if there are a
proper, effective and \'etale groupoid $\groupoidtot{X}{4}$ and a pair of Morita equivalences
  
\begin{equation}\label{eq-37}
\begin{tikzpicture}[xscale=2.4,yscale=-1.2]
    \node (A0_0) at (0, 0) {$\groupoidtot{X}{\prime 3}$};
    \node (A0_1) at (1, 0) {$\groupoidtot{X}{4}$};
    \node (A0_2) at (2, 0) {$\groupoidtot{X}{3}$,};
    
    \path (A0_1) edge [->]node [auto,swap]
      {$\scriptstyle{\groupoidmaptot{\gamma}{\prime}}$} (A0_0);
    \path (A0_1) edge [->]node [auto] {$\scriptstyle{\groupoidmaptot{\gamma}{}}$} (A0_2);
\end{tikzpicture}
\end{equation}
such that $|\groupoidmaptot{\xi}{1}|\circ|\groupoidmaptot{\gamma}{}|=
|\groupoidmaptot{\xi}{\prime 1}|\circ|\groupoidmaptot{\gamma}{\prime}|$ and

\begin{equation}\label{eq-38}
\Big(i_{\groupoidmaptot{\phi}{2}}\ast\beta^2\Big)\odot\Big(\alpha\ast
i_{\groupoidmaptot{\gamma}{}}
\Big)\odot\Big(i_{\groupoidmaptot{\phi}{1}}\ast\beta^1\Big)=\alpha\ast
i_{\groupoidmaptot{\gamma}{\prime}},
\end{equation}
where each $\beta^1$ is the unique natural transformation from $\groupoidmaptot{\xi}{\prime 1}\circ
\groupoidmaptot{\gamma}{\prime}$ to $\groupoidmaptot{\xi}{1}\circ\groupoidmaptot{\gamma}{}$ and
$\beta^2$ is the unique natural transformation from $\groupoidmaptot{\xi}{2}\circ
\groupoidmaptot{\gamma}{}$ to $\groupoidmaptot{\xi}{\prime 2}\circ\groupoidmaptot{\gamma}{\prime}$.
Therefore, from now on we will denote each $2$-morphism in $\PEEGpd\left[\WPEEGpdinv\right]$
from $\underline{g}^1$
to $\underline{g}^2$ as the class $[\groupoidtot{X}{3},\groupoidmaptot{\xi}{1},
\groupoidmaptot{\xi}{2},\alpha]$,
where the class of equivalence is the one induced by saying that

\[\Big(\groupoidtot{X}{3},\groupoidmaptot{\xi}{1},\groupoidmaptot{\xi}{2},\alpha\Big)
\sim\Big(\groupoidtot{X}{\prime 3},\groupoidmaptot{\xi}{\prime 1},
\groupoidmaptot{\xi}{\prime 2},\alpha'\Big)\]
if and only if there are data $(\groupoidtot{X}{4},\groupoidmaptot{\gamma}{},
\groupoidmaptot{\gamma}{\prime})$ as above,
such that \eqref{eq-38} holds.
\end{lem}

\begin{proof}
The proof of this result follows the same lines of the proof of Lemma~\ref{lem-25}.
The only significant difference is that we use Lemma~\ref{lem-20} instead of Lemma~\ref{lem-12}: this
allows to prove the existence and uniqueness of a $2$-morphism $\beta^1:\groupoidmaptot{\xi}{\prime 1}\circ
\groupoidmaptot{\gamma}{\prime }\Rightarrow\groupoidmaptot{\xi}{1}\circ\groupoidmaptot{\gamma}{}$
as desired. For the existence and uniqueness of the $2$-morphism $\beta^2$, one proceeds as follows:
using Lemma~\ref{lem-20} and (b), we have that

\[|\groupoidmaptot{\psi}{2}|\circ|\groupoidmaptot{\xi}{2}|\circ|\groupoidmaptot{\gamma}{}|=
|\groupoidmaptot{\psi}{1}|\circ|\groupoidmaptot{\xi}{1}|\circ|\groupoidmaptot{\gamma}{}|;\]
moreover, using the hypothesis we have

\[|\groupoidmaptot{\psi}{1}|\circ|\groupoidmaptot{\xi}{1}|\circ|\groupoidmaptot{\gamma}{}|=
|\groupoidmaptot{\psi}{1}|\circ|\groupoidmaptot{\xi}{\prime 1}|\circ|\groupoidmaptot{\gamma}{\prime}|\]
and

\[|\groupoidmaptot{\psi}{1}|\circ|\groupoidmaptot{\xi}{\prime 1}|\circ|\groupoidmaptot{\gamma}{\prime}|=
|\groupoidmaptot{\psi}{2}|\circ|\groupoidmaptot{\xi}{\prime 2}|\circ|\groupoidmaptot{\gamma}{\prime}|.\]

Since $|\groupoidmaptot{\psi}{2}|$ is an homeomorphism (see Lemma~\ref{lem-14}), then the previous $3$
identities imply that $|\groupoidmaptot{\xi}{2}|\circ|\groupoidmaptot{\gamma}{}|=
|\groupoidmaptot{\xi}{\prime 2}|\circ|\groupoidmaptot{\gamma}{\prime}|$; then the existence and uniqueness
of $\beta^2$ is again a consequence of Lemma~\ref{lem-20}.
\end{proof}

\section{The pseudofunctor $\functor{G}^{\red}$}
Now we are almost ready to describe the pseudofunctor $\functor{G}^{\red}$ mentioned in the Introduction.
For that, we will only need the following result.

\begin{theo}\label{theo-02}
\cite[Theorem~0.3 and Remark~3.2 in the case when $\CATA$ and $\CATB$ are $2$-categories]{T4}
Let us fix any pair of $2$-categories $\CATA,\CATB$ and any pair of classes $\SETWA$ and $\SETWB$
of morphisms in $\CATA$ and $\CATB$ respectively, such that both $(\CATA,\SETWA)$ and
$(\CATB,\SETWB)$ satisfy conditions
\emphatic{(\hyperref[BF]{BF})}. Let us also fix any pseudofunctor $\functor{F}:\CATA\rightarrow
\CATB$ such that $\functor{F}_1(\SETWA)\subseteq\SETWBsat$ and let us 
assume the axiom of choice. Then there is a pseudofunctor

\[\widetilde{\functor{G}}:\CATA\Big[\SETWAinv\Big]\longrightarrow\CATB\Big[\SETWBsatinv\Big]\]
 such that:

\begin{itemize}
 \item $\functor{U}_{\SETWBsat}\circ\functor{F}=\widetilde{\functor{G}}\circ\functor{U}_{\SETWA}$;
 
 \item for each object $A_{\CATA}$, we have $\widetilde{\functor{G}}_0(A_{\CATA})=
  \functor{F}_0(A_{\CATA})$;

 \item for each morphism $(A'_{\CATA},\operatorname{w}_{\CATA},f_{\CATA}):A_{\CATA}\rightarrow
  B_{\CATA}$ in $\CATA\left[\SETWAinv\right]$, we have 
  
  \[\widetilde{\functor{G}}_1\Big(A'_{\CATA},\operatorname{w}_{\CATA},f_{\CATA}\Big)=\Big(
  \functor{F}_0(A'_{\CATA}),\functor{F}_1(\operatorname{w}_{\CATA}),\functor{F}_1(f_{\CATA})
  \Big);\]
  
 \item for each $2$-morphism

  \[\Big[A^3_{\CATA},\operatorname{v}^1_{\CATA},\operatorname{v}^2_{\CATA},\mu_{\CATA},
  \delta_{\CATA}\Big]:\Big(A^1_{\CATA},\operatorname{w}^1_{\CATA},f^1_{\CATA}\Big)\Longrightarrow
  \Big(A^2_{\CATA},\operatorname{w}^2_{\CATA},f^2_{\CATA}\Big)\]
  in $\CATA\left[\SETWAinv\right]$, we have
  
  \[\widetilde{\functor{G}}_2\Big(\Big[A^3_{\CATA},\operatorname{v}^1_{\CATA},
  \operatorname{v}^2_{\CATA},\mu_{\CATA},\delta_{\CATA}\Big]\Big)=
  \Big[\functor{F}_0(A^3_{\CATA}),\functor{F}_1(\operatorname{v}^1_{\CATA}),\functor{F}_1
  (\operatorname{v}^2_{\CATA}),\functor{F}_2(\mu_{\CATA}),
  \functor{F}_2(\delta_{\CATA})\Big].\]
\end{itemize}  
\end{theo}

Then we have:

\begin{prop}\label{prop-06}
If we assume the axiom of choice, there is a pseudofunctor

\[\functor{G}^{\red}:\RedOrb\longrightarrow\PEEGpd\left[\WPEEGpdinv\right]\]
such that:

\begin{enumerate}[\emphatic{(}1\emphatic{)}]
 \item for each reduced orbifold atlas $\atlas{X}$, $\functor{G}^{\red}_0(\atlas{X})=
  \functor{F}^{\red}_0(\atlas{X})$;
 \item for each morphism $(\atlas{X}',\atlasmap{\operatorname{w}}{},\atlasmap{f}{}):\atlas{X}
  \rightarrow\atlas{Y}$ in $\RedOrb$, we have
  
  \[\functor{G}^{\red}_1\Big(\atlas{X}',\atlasmap{\operatorname{w}}{},\atlasmap{f}{}\Big)=
  \Big(\functor{F}_0^{\red}(\atlas{X}),\functor{F}^{\red}_1(\atlasmap{\operatorname{w}}{}),
  \functor{F}_1^{\red}(\atlasmap{f}{})\Big);\]

 \item for each $2$-morphism

  \[\Big[\atlas{X}^3,\atlasmap{\operatorname{v}}{1},\atlasmap{\operatorname{v}}{2},
  [\delta]\Big]:\Big([\atlas{X}^1,\atlasmap{\operatorname{w}}{1},\atlasmap{f}{1}\Big)
  \Longrightarrow\Big(\atlas{X}^2,\atlasmap{\operatorname{w}}{2},\atlasmap{f}{2}\Big)\]
  in $\RedOrb$, we have
  
  \[\functor{G}^{\red}_2\Big(\Big[[\atlas{X}^3,\atlasmap{\operatorname{v}}{1},
   \atlasmap{\operatorname{v}}{2},[\delta]\Big]\Big)=
   \Big[\functor{F}_0^{\red}(\atlas{X}^3),\functor{F}_1^{\red}(\atlasmap{\operatorname{v}}{1}),
   \functor{F}_1^{\red}(\atlasmap{\operatorname{v}}{2}),\functor{F}_2^{\red}([\delta])\Big].\]
\end{enumerate}

Moreover, we have $\functor{U}_{\WPEEGpd}\circ\functor{F}^{\red}=\functor{G}^{\red}\circ
\functor{U}_{\WRedAtl}$.
\end{prop}

\begin{proof}
Let us apply Theorem~\ref{theo-02} with $\CATA:=\RedAtl$, $\SETWA:=\WRedAtl$ (i.e.\ all refinements
of reduced orbifold atlases), $\CATB:=\PEEGpd$, $\SETWB:=\WPEEGpd$ (i.e.\ all Morita equivalences of
proper, effective, \'etale groupoids) and $\functor{F}:=\functor{F}^{\red}$. We recall that
by Lemma~\ref{lem-17} we have $\SETWBsat=\WPEEGpd$. Given any refinement
$\atlasmap{\operatorname{w}}{}$, by Proposition~\ref{prop-04} we have that
$\functor{F}_1^{\red}(\atlasmap{\operatorname{w}}{})$ is a Morita equivalence, so we are in the
hypothesis of Theorem~\ref{theo-02}. Then the claim follows at once
using Lemmas~\ref{lem-25} and~\ref{lem-16} for the description of $\functor{G}_2^{\red}$.
\end{proof}

In addition, we recall the following result. For the more general form of this statement,
we refer to~\cite[Theorem~0.2]{T5}. We state such a result here only in the special framework where:

\begin{itemize}
 \item $\CATA$ and $\CATB$ are $2$-categories and $\functor{F}$ is a $2$-functor (also known as
  strict pseudofunctor), i.e.\ it preserves compositions and identities;
 \item $\functor{U}_{\SETWB}\circ\functor{F}=\functor{G}\circ\functor{U}_{\SETWA}$ and the
  natural equivalence $\kappa$ appearing in~\cite[Theorem~0.2]{T5} is the $2$-identity of
  $\functor{U}_{\SETWB}\circ\functor{F}$.
\end{itemize}

\begin{theo}\cite{T5}
Let us fix any pair of $2$-categories $\CATA$, $\CATB$ and any pair of classes of morphisms $\SETWA$,
$\SETWB$, such that both $(\CATA,\SETWA)$ and $(\CATB,\SETWB)$ satisfy conditions
\emphatic{(\hyperref[BF]{BF})}. Moreover, let us fix any $2$-functor $\functor{F}:\CATA
\rightarrow\CATB$, such that $\functor{F}_1(\SETWA)\subseteq\SETWBsat$.
In addition, let us suppose that there is a pseudofunctor $\functor{G}:\CATA\left[\SETWAinv\right]
\rightarrow\CATB\left[\SETWBinv\right]$ such that
$\functor{U}_{\SETWB}\circ\functor{F}=\functor{G}\circ\functor{U}_{\SETWA}$, and let us assume the
axiom of choice. Then $\functor{G}$ is an equivalence of bicategories if and
only if $\functor{F}$ satisfies the following $5$ conditions.

\begin{enumerate}[\emphatic{(}{A}1\emphatic{)}]\label{A}
 \item\label{A1} For any object $A_{\CATB}$, there are a pair of objects $A_{\CATA}$ and
  $A'_{\CATB}$ and a pair of morphisms $\operatorname{w}^1_{\CATB}$ in $\SETWB$ and
  $\operatorname{w}^2_{\CATB}$ in $\SETWBsat$, as follows:
  
  \[
  \begin{tikzpicture}[xscale=2.4,yscale=-1.2]
    \node (A0_0) at (-0.2, 0) {$\functor{F}_0(A_{\CATA})$};
    \node (A0_1) at (1, 0) {$A'_{\CATB}$};
    \node (A0_2) at (2, 0) {$A_{\CATB}$.};
    \path (A0_1) edge [->]node [auto,swap] {$\scriptstyle{\operatorname{w}^1_{\CATB}}$} (A0_0);
    \path (A0_1) edge [->]node [auto] {$\scriptstyle{\operatorname{w}^2_{\CATB}}$} (A0_2);
  \end{tikzpicture}
  \]

 \item\label{A2} Let us fix any triple of objects $A^1_{\CATA},A^2_{\CATA},A_{\CATB}$ and any pair of
  morphisms $\operatorname{w}^1_{\CATB}$ in $\SETWB$ and $\operatorname{w}^2_{\CATB}$ in
  $\SETWBsat$ as follows
  
  \[
  \begin{tikzpicture}[xscale=2.4,yscale=-1.2]
    \node (A0_0) at (0, 0) {$\functor{F}_0(A^1_{\CATA})$};
    \node (A0_1) at (1, 0) {$A_{\CATB}$};
    \node (A0_2) at (2, 0) {$\functor{F}_0(A^2_{\CATA})$.};
    \path (A0_1) edge [->]node [auto,swap] {$\scriptstyle{\operatorname{w}^1_{\CATB}}$} (A0_0);
    \path (A0_1) edge [->]node [auto] {$\scriptstyle{\operatorname{w}^2_{\CATB}}$} (A0_2);
  \end{tikzpicture}
  \]
  Then there are an object $A^3_{\CATA}$, a pair of morphisms $\operatorname{w}^1_{\CATA}$ in
  $\SETWA$ and $\operatorname{w}^2_{\CATA}$ in $\SETWAsat$ as follows
  
  \[
  \begin{tikzpicture}[xscale=2.4,yscale=-1.2]
    \node (A0_0) at (0, 0) {$A^1_{\CATA}$};
    \node (A0_1) at (1, 0) {$A^3_{\CATA}$};
    \node (A0_2) at (2, 0) {$A^2_{\CATA}$};
    
    \path (A0_1) edge [->]node [auto,swap] {$\scriptstyle{\operatorname{w}^1_{\CATA}}$} (A0_0);
    \path (A0_1) edge [->]node [auto] {$\scriptstyle{\operatorname{w}^2_{\CATA}}$} (A0_2);
  \end{tikzpicture}
  \]
  and a set of data $(A'_{\CATB},\operatorname{z}^1_{\CATB},\operatorname{z}^2_{\CATB},
  \gamma^1_{\CATB},\gamma^2_{\CATB})$ as follows
  
  \[
  \begin{tikzpicture}[xscale=2.2,yscale=-0.8]
    \node (A0_2) at (2, 0) {$A_{\CATB}$};
    \node (A2_2) at (2, 2) {$A'_{\CATB}$};
    \node (A2_0) at (0, 2) {$\functor{F}_0(A^1_{\CATA})$};
    \node (A2_4) at (4, 2) {$\functor{F}_0(A^2_{\CATA})$,};
    \node (A4_2) at (2, 4) {$\functor{F}_0(A^3_{\CATA})$};
    
    \node (A2_3) at (2.8, 2) {$\Downarrow\,\gamma^2_{\CATB}$};
    \node (A2_1) at (1.2, 2) {$\Downarrow\,\gamma^1_{\CATB}$};
    
    \path (A4_2) edge [->]node [auto,swap]
      {$\scriptstyle{\functor{F}_1(\operatorname{w}^2_{\CATA})}$} (A2_4);
    \path (A0_2) edge [->]node [auto] {$\scriptstyle{\operatorname{w}^2_{\CATB}}$} (A2_4);
    \path (A2_2) edge [->]node [auto,swap] {$\scriptstyle{\operatorname{z}^1_{\CATB}}$} (A0_2);
    \path (A2_2) edge [->]node [auto] {$\scriptstyle{\operatorname{z}^2_{\CATB}}$} (A4_2);
    \path (A4_2) edge [->]node [auto]
      {$\scriptstyle{\functor{F}_1(\operatorname{w}^1_{\CATA})}$} (A2_0);
    \path (A0_2) edge [->]node [auto,swap] {$\scriptstyle{\operatorname{w}^1_{\CATB}}$} (A2_0);
  \end{tikzpicture}
  \]
  such that $\operatorname{z}^1_{\CATB}$ belongs to $\SETWB$ and both $\gamma^1_{\CATB}$ and
  $\gamma^2_{\CATB}$ are invertible.

 \item\label{A3} Let us fix any pair of objects $B_{\CATA},A_{\CATB}$ and any morphism $f_{\CATB}:
  A_{\CATB}\rightarrow\functor{F}_0(B_{\CATA})$. Then there are an object $A_{\CATA}$, a morphism
  $f_{\CATA}:A_{\CATA}\rightarrow B_{\CATA}$ and data $(A'_{\CATB},\operatorname{v}^1_{\CATB},
  \operatorname{v}^2_{\CATB},\alpha_{\CATB})$ as follows
  
  \[
  \begin{tikzpicture}[xscale=1.8,yscale=-0.6]
    \node (B0_0) at (-1, 0) {}; 
    \node (B1_1) at (5, 0) {}; 
    
    \node (A0_2) at (2, 0) {$A_{\CATB}$};
    \node (A2_2) at (0, 1) {$A'_{\CATB}$};
    \node (A2_4) at (4, 1) {$\functor{F}_0(B_{\CATA})$,};
    \node (A4_2) at (2, 2) {$\functor{F}_0(A_{\CATA})$};
    
    \node (A2_3) at (2, 1) {$\Downarrow\,\alpha_{\CATB}$};
        
    \path (A4_2) edge [->,bend left=15] node [auto,swap]
      {$\scriptstyle{\functor{F}_1(f_{\CATA})}$} (A2_4);
    \path (A0_2) edge [->,bend right=15] node [auto] {$\scriptstyle{f_{\CATB}}$} (A2_4);
    \path (A2_2) edge [->,bend right=15] node [auto]
      {$\scriptstyle{\operatorname{v}^1_{\CATB}}$} (A0_2);
    \path (A2_2) edge [->,bend left=15] node [auto,swap]
      {$\scriptstyle{\operatorname{v}^2_{\CATB}}$} (A4_2);
  \end{tikzpicture}
  \]
  with $\operatorname{v}^1_{\CATB}$ in $\SETWB$, $\operatorname{v}^2_{\CATB}$ in $\SETWBsat$ and
  $\alpha_{\CATB}$ invertible.

 \item\label{A4} Let us fix any pair of objects $A_{\CATA},B_{\CATA}$, any pair of morphisms
  $f^1_{\CATA},f^2_{\CATA}:A_{\CATA}\rightarrow B_{\CATA}$ and any pair of $2$-morphisms
  $\gamma^1_{\CATA},\gamma^2_{\CATA}:f^1_{\CATA}\Rightarrow f^2_{\CATA}$. Moreover, let us fix
  any object $A'_{\CATB}$ and any morphism $\operatorname{z}_{\CATB}:A'_{\CATB}
  \rightarrow\functor{F}_0(A_{\CATA})$ in $\SETWB$. If $\functor{F}_2(\gamma^1_{\CATA})
  \ast i_{\operatorname{z}_{\CATB}}=\functor{F}_2(\gamma^2_{\CATA})\ast
  i_{\operatorname{z}_{\CATB}}$, then there are an object $A'_{\CATA}$ and a morphism
  $\operatorname{z}_{\CATA}:A'_{\CATA}\rightarrow A_{\CATA}$ in $\SETWA$, such that
  $\gamma^1_{\CATA}\ast i_{\operatorname{z}_{\CATA}}=\gamma^2_{\CATA}\ast
  i_{\operatorname{z}_{\CATA}}$.

 \item\label{A5} Let us fix any triple of objects $A_{\CATA},B_{\CATA},A_{\CATB}$, any pair of
  morphisms $f^1_{\CATA},f^2_{\CATA}:A_{\CATA}\rightarrow B_{\CATA}$, any morphism
  $\operatorname{v}_{\CATB}:A_{\CATB}\rightarrow\functor{F}_0(A_{\CATA})$ in $\SETWB$ and any
  $2$-morphism
  
  \[
  \begin{tikzpicture}[xscale=1.8,yscale=-0.6]
    \node (B0_0) at (-1, 0) {}; 
    \node (B1_1) at (5, 0) {}; 
    
    \node (A0_1) at (2, 0) {$\functor{F}_0(A_{\CATA})$};
    \node (A1_0) at (0, 1) {$A_{\CATB}$};
    \node (A1_2) at (4, 1) {$\functor{F}_0(B_{\CATA})$.};
    \node (A2_1) at (2, 2) {$\functor{F}_0(A_{\CATA})$};
    
    \node (A1_1) at (2, 1) {$\Downarrow\,\alpha_{\CATB}$};
        
    \path (A1_0) edge [->,bend right=20]node [auto]
       {$\scriptstyle{\operatorname{v}_{\CATB}}$} (A0_1);
    \path (A1_0) edge [->,bend left=20]node [auto,swap]
      {$\scriptstyle{\operatorname{v}_{\CATB}}$} (A2_1);
    \path (A0_1) edge [->,bend right=20]node [auto]
      {$\scriptstyle{\functor{F}_1(f^1_{\CATA})}$} (A1_2);
    \path (A2_1) edge [->,bend left=20]node [auto,swap]
      {$\scriptstyle{\functor{F}_1(f^2_{\CATA})}$} (A1_2);
  \end{tikzpicture}
  \]
  Then there are a pair of objects
  $A'_{\CATA},A'_{\CATB}$, a triple of morphisms $\operatorname{v}_{\CATA}:A'_{\CATA}\rightarrow
  A_{\CATA}$ in $\SETWA$, $\operatorname{z}_{\CATB}:A'_{\CATB}\rightarrow\functor{F}_0(A'_{\CATA})$
  in $\SETWB$ and $\operatorname{z}'_{\CATB}:A'_{\CATB}\rightarrow A_{\CATB}$,
  a $2$-morphism 
  
  \[
  \begin{tikzpicture}[xscale=1.8,yscale=-0.6]
    \node (B0_0) at (-1, 0) {}; 
    \node (B1_1) at (5, 0) {}; 
    
    \node (A0_1) at (2, 0) {$A_{\CATA}$};
    \node (A1_0) at (0, 1) {$A'_{\CATA}$};
    \node (A1_2) at (4, 1) {$B_{\CATA}$};
    \node (A2_1) at (2, 2) {$A_{\CATA}$};

    \node (A1_1) at (2, 1) {$\Downarrow\,\alpha_{\CATA}$};

    \path (A1_0) edge [->,bend right=20]node [auto] {$\scriptstyle{\operatorname{v}_{\CATA}}$} (A0_1);
    \path (A1_0) edge [->,bend left=20]node [auto,swap]
      {$\scriptstyle{\operatorname{v}_{\CATA}}$} (A2_1);
    \path (A0_1) edge [->,bend right=20]node [auto] {$\scriptstyle{f^1_{\CATA}}$} (A1_2);
    \path (A2_1) edge [->,bend left=20]node [auto,swap] {$\scriptstyle{f^2_{\CATA}}$} (A1_2);
   \end{tikzpicture}
   \]
  and an invertible $2$-morphism
  
  \[
  \begin{tikzpicture}[xscale=1.8,yscale=-0.6]
    \node (B0_0) at (-1, 0) {}; 
    \node (B1_1) at (5, 0) {}; 
    
    \node (A0_1) at (2, 0) {$\functor{F}_0(A'_{\CATA})$};
    \node (A1_0) at (0, 1) {$A'_{\CATB}$};
    \node (A1_2) at (4, 1) {$\functor{F}_0(A_{\CATA})$,};
    \node (A2_1) at (2, 2) {$A_{\CATB}$};
    
    \node (A1_1) at (2, 1) {$\Downarrow\,\sigma_{\CATB}$};

    \path (A1_0) edge [->,bend right=20]node [auto]
      {$\scriptstyle{\operatorname{z}_{\CATB}}$} (A0_1);
    \path (A1_0) edge [->,bend left=20]node [auto,swap]
      {$\scriptstyle{\operatorname{z}'_{\CATB}}$} (A2_1);
    \path (A0_1) edge [->,bend right=20]node [auto] {$\scriptstyle{\functor{F}_1
      (\operatorname{v}_{\CATA})}$} (A1_2);
    \path (A2_1) edge [->,bend left=20]node [auto,swap]
      {$\scriptstyle{\operatorname{v}_{\CATB}}$} (A1_2);
   \end{tikzpicture}
   \]
  such that $\alpha_{\CATB}\ast i_{\operatorname{z}'_{\CATB}}$
  coincides with the following composition:
  
  \[
  \begin{tikzpicture}[xscale=2.5,yscale=-1.3]
    \node (A0_1) at (1, 0) {$A_{\CATB}$};
    \node (A0_2) at (2, 0) {$\functor{F}_0(A_{\CATA})$};
    \node (A1_0) at (0, 1) {$A'_{\CATB}$};
    \node (A1_1) at (1, 1) {$\functor{F}_0(A'_{\CATA})$};
    \node (A1_3) at (3, 1) {$\functor{F}_0(B_{\CATA})$.};
    \node (A2_1) at (1, 2) {$A_{\CATB}$};
    \node (A2_2) at (2, 2) {$\functor{F}_0(A_{\CATA})$};

    \node (A1_2) at (2.3, 1) {$\Downarrow\,\functor{F}_2(\alpha_{\CATA})$};
    \node (A2_0) at (1, 1.5) {$\Downarrow\,\sigma_{\CATB}$};
    \node (A0_0) at (1, 0.5) {$\Downarrow\,(\sigma_{\CATB})^{-1}$};
    
    \path (A1_0) edge [->]node [auto,swap] {$\scriptstyle{\operatorname{z}'_{\CATB}}$} (A2_1);
    \path (A2_1) edge [->]node [auto,swap] {$\scriptstyle{\operatorname{v}_{\CATB}}$} (A2_2);
    \path (A1_1) edge [->]node [auto,swap] {$\scriptstyle{\functor{F}_1
      (\operatorname{v}_{\CATA})}$} (A0_2);
    \path (A1_0) edge [->]node [auto] {$\scriptstyle{\operatorname{z}'_{\CATB}}$} (A0_1);
    \path (A1_0) edge [->]node [auto] {$\scriptstyle{\operatorname{z}_{\CATB}}$} (A1_1);
    \path (A2_2) edge [->]node [auto,swap] {$\scriptstyle{\functor{F}_1(f^2_{\CATA})}$} (A1_3);
    \path (A1_1) edge [->]node [auto] {$\scriptstyle{\functor{F}_1
      (\operatorname{v}_{\CATA})}$} (A2_2);
    \path (A0_2) edge [->]node [auto] {$\scriptstyle{\functor{F}_1(f^1_{\CATA})}$} (A1_3);
    \path (A0_1) edge [->]node [auto] {$\scriptstyle{\operatorname{v}_{\CATB}}$} (A0_2);
  \end{tikzpicture}
  \]
\end{enumerate}
\end{theo}

Then we have:

\begin{theo}\label{theo-05}
The pseudofunctor $\functor{G}^{\red}$ described in \emphatic{Proposition~\ref{prop-06}}
\emphatic{(}using the axiom
of choice\emphatic{)} is an equivalence of bicategories.
\end{theo}

\begin{proof}
Let us verify condition (\hyperref[A1]{A1}), so let us fix any $\groupoidtot{X}{}$ in $\PEEGpd$; using
Lemma~\ref{lem-11} there are a reduced orbifold atlas $\atlas{X}$ and a Morita equivalence
$\groupoidmaptot{\psi}{}:\functor{F}_0^{\red}(\atlas{X})\rightarrow\groupoidtot{X}{}$. Therefore,
(\hyperref[A1]{A1}) holds if we choose the following set of data:

\[\begin{tikzpicture}[xscale=2.6,yscale=-1.2]
    \node (A0_0) at (-0.2, 0) {$\functor{F}_0^{\red}(\atlas{X})$};
    \node (A0_1) at (1, 0) {$\functor{F}_0^{\red}(\atlas{X})$};
    \node (A0_2) at (2, 0) {$\groupoidtot{X}{}$.};
    \path (A0_1) edge [->]node [auto,swap]
      {$\scriptstyle{\id_{\functor{F}_0^{\red}(\atlas{X})}}$} (A0_0);
    \path (A0_1) edge [->]node [auto] {$\scriptstyle{\groupoidmaptot{\psi}{}}$} (A0_2);
\end{tikzpicture}\]

Let us consider (\hyperref[A2]{A2}), so let us fix any pair of reduced orbifold atlases
$\atlas{X}^1,\atlas{X}^2$ and any $\groupoidtot{X}{}$ in $\PEEGpd$, together with any pair of Morita
equivalences as follows:

\[\begin{tikzpicture}[xscale=2.9,yscale=-1.2]
    \node (A0_0) at (0, 0) {$\functor{F}_0^{\red}(\atlas{X}^1)$};
    \node (A0_1) at (1, 0) {$\groupoidtot{X}{}$};
    \node (A0_2) at (2, 0) {$\functor{F}_0^{\red}(\atlas{X}^2)$.};
    \path (A0_1) edge [->]node [auto,swap] {$\scriptstyle{\groupoidmaptot{\psi}{1}}$} (A0_0);
    \path (A0_1) edge [->]node [auto] {$\scriptstyle{\groupoidmaptot{\psi}{2}}$} (A0_2);
\end{tikzpicture}\]

By Lemma~\ref{lem-11} there are a reduced orbifold atlas $\atlas{Y}$ and a Morita equivalence
$\groupoidmaptot{\phi}{}:\functor{F}_0^{\red}(\atlas{Y})\rightarrow\groupoidtot{X}{}$. By
Proposition~\ref{prop-04} there is a weak equivalence $\atlasmap{\operatorname{v}}{}:\atlas{Y}
\rightarrow\atlas{X}^1$, such that
$\functor{F}_1^{\red}(\atlasmap{\operatorname{v}}{})=\groupoidmaptot{\psi}{1}\circ
\groupoidmaptot{\phi}{}$. Since 
$\atlasmap{\operatorname{v}}{}$ a weak equivalence, then by Lemma~\ref{lem-18} there are a reduced
orbifold atlas $\atlas{X}^3$ and a weak equivalence $\atlasmap{\operatorname{u}}{}:\atlas{X}^3
\rightarrow\atlas{Y}$, such that the morphism

\[\atlasmap{\operatorname{w}}{1}=\atlasmap{\operatorname{v}}{}\circ\atlasmap{\operatorname{u}}{}:\,
\atlas{X}^3\longrightarrow\atlas{X}^1\]
is a refinement. We set $\groupoidmaptot{\xi}{}:=\functor{F}_1(\atlasmap{\operatorname{u}}{})$; this
morphism is a Morita equivalence by Lemma~\ref{lem-09} and we have

\[\functor{F}_1^{\red}(\atlasmap{\operatorname{w}}{1})=\groupoidmaptot{\psi}{1}\circ
\groupoidmaptot{\phi}{}\circ\groupoidmaptot{\xi}{}.\]

Again by Proposition~\ref{prop-04} there is a unique weak equivalence $\atlasmap{\operatorname{w}}{2}:
\atlas{X}^3\rightarrow\atlas{X}^2$, such that

\[\functor{F}_1^{\red}(\atlasmap{\operatorname{w}}{2})=\groupoidmaptot{\psi}{2}\circ
\groupoidmaptot{\phi}{}\circ\groupoidmaptot{\xi}{}.\]

By Lemma~\ref{lem-19}, we have that $\atlasmap{\operatorname{w}}{2}$ belongs to the right saturation
of $\WRedAtl$. Then (\hyperref[A2]{A2}) is satisfied by the following set of data

\[\begin{tikzpicture}[xscale=2.4,yscale=-0.8]
    \node (A0_2) at (2, 0) {$\groupoidtot{X}{}$};
    \node (A2_2) at (2, 2) {$\functor{F}_0^{\red}(\atlas{X}^3)$};
    \node (A2_0) at (0, 2) {$\functor{F}_0^{\red}(\atlas{X}^1)$};
    \node (A2_4) at (4, 2) {$\functor{F}_0^{\red}(\atlas{X}^2)$.};
    \node (A4_2) at (2, 4) {$\functor{F}_0^{\red}(\atlas{X}^3)$};

    \node (A2_3) at (3, 2) {$\Downarrow\,i_{\functor{F}_1^{\red}(\atlasmap{\operatorname{w}}{2})}$};
    \node (A2_1) at (1, 2) {$\Downarrow\,i_{\functor{F}_1^{\red}(\atlasmap{\operatorname{w}}{1})}$};
    
    \path (A4_2) edge [->]node [auto,swap]
      {$\scriptstyle{\functor{F}_1^{\red}(\atlasmap{\operatorname{w}}{2})}$} (A2_4);
    \path (A0_2) edge [->]node [auto] {$\scriptstyle{\groupoidmaptot{\psi}{2}}$} (A2_4);
    \path (A2_2) edge [->]node [auto,swap]
      {$\scriptstyle{\groupoidmaptot{\phi}{}\circ\groupoidmaptot{\xi}{}}$} (A0_2);
    \path (A2_2) edge [->]node [auto]
      {$\scriptstyle{\id_{\functor{F}_0^{\red}(\atlas{X}^3)}}$} (A4_2);
    \path (A4_2) edge [->]node [auto]
      {$\scriptstyle{\functor{F}_1^{\red}(\atlasmap{\operatorname{w}}{1})}$} (A2_0);
    \path (A0_2) edge [->]node [auto,swap] {$\scriptstyle{\groupoidmaptot{\psi}{1}}$} (A2_0);
\end{tikzpicture}\]

Let us consider (\hyperref[A3]{A3}), so let us fix any reduced orbifold atlas $\atlas{Y}$, any object
$\groupoidtot{X}{}$ in $\PEEGpd$ and any morphism $\groupoidmaptot{\phi}{}:\groupoidtot{X}{}
\rightarrow\functor{F}_0^{\red}
(\atlas{Y})$. By Lemma~\ref{lem-11} there are a reduced orbifold atlas $\atlas{X}$ and a Morita
equivalence $\groupoidmaptot{\psi}{}:\functor{F}_0^{\red}(\atlas{X})\rightarrow\groupoidtot{X}{}$. By
Lemma~\ref{lem-07} there is a unique morphism $\atlasmap{f}{}:\atlas{X}\rightarrow\atlas{Y}$, such
that $\functor{F}_0^{\red}(\atlasmap{f}{})=\groupoidmaptot{\phi}{}\circ\groupoidmaptot{\psi}{}$. Then
(\hyperref[A3]{A3}) is easily verified with $A'_{\CATB}:=\functor{F}_0^{\red}(\atlas{X})$,
$\operatorname{v}^1_{\CATB}:=\groupoidmaptot{\psi}{}$ and $\operatorname{v}^2_{\CATB}:=
\id_{\functor{F}_0^{\red}(\atlas{X})}$.\\

Let us consider (\hyperref[A4]{A4}), so let us fix any pair of reduced orbifold atlases $\atlas{X},
\atlas{Y}$, any pair of morphisms $\atlasmap{f}{1},\atlasmap{f}{2}:\atlas{X}\rightarrow\atlas{Y}$ and
any pair of $2$-morphisms $[\gamma^1],[\gamma^2]:\atlasmap{f}{1}\Rightarrow\atlasmap{f}{2}$
in $\RedAtl$. Moreover, let us fix any object $\groupoidtot{X}{}$ in $\PEEGpd$ and any Morita
equivalence $\groupoidmaptot{\psi}{}:\groupoidtot{X}{}\rightarrow\functor{F}_0^{\red}(\atlas{X})$,
such that

\begin{equation}\label{eq-25}
\functor{F}_2^{\red}([\gamma^1])\ast i_{\groupoidmaptot{\psi}{}}=\functor{F}_2^{\red}([\gamma^2])\ast
i_{\groupoidmaptot{\psi}{}}.
\end{equation}

By Lemma~\ref{lem-11} there are a reduced orbifold atlas $\atlas{Z}$ and a Morita equivalence
$\groupoidmaptot{\phi}{}:\functor{F}_0^{\red}(\atlas{Z})\rightarrow\groupoidtot{X}{}$.
By Proposition~\ref{prop-04} there is a unique weak equivalence $\atlasmap{\operatorname{u}}{}:
\atlas{Z}\rightarrow\atlas{X}$ such that $\functor{F}_1^{\red}(\atlasmap{\operatorname{u}}{})=
\groupoidmaptot{\psi}{}\circ\groupoidmaptot{\phi}{}$. By Lemma~\ref{lem-18} there are a reduced
orbifold atlas $\atlas{U}$ and a weak equivalence $\atlasmap{\operatorname{v}}{}:
\atlas{U}\rightarrow\atlas{Z}$,
such that $\atlasmap{\operatorname{z}}{}:=\atlasmap{\operatorname{u}}{}\circ
\atlasmap{\operatorname{v}}{}$ is a refinement. So:

\begin{gather*}
\functor{F}_2^{\red}\left([\gamma^1]\ast i_{\atlasmap{\operatorname{z}}{}}\right)=\functor{F}_2^{\red}
 ([\gamma^1])\ast i_{\groupoidmaptot{\psi}{}\circ\groupoidmaptot{\phi}{}\circ\functor{F}_1^{\red}
 (\atlasmap{\operatorname{v}}{})}\stackrel{\eqref{eq-25}}{=} \\
\stackrel{\eqref{eq-25}}{=}\functor{F}_2^{\red}([\gamma^2])\ast i_{\groupoidmaptot{\psi}{}\circ
 \groupoidmaptot{\phi}{}\circ
 \functor{F}_1^{\red}(\atlasmap{\operatorname{v}}{})}=\functor{F}_2^{\red}\left([\gamma^2]\ast
 i_{\atlasmap{\operatorname{z}}{}}\right).
\end{gather*}

By Lemma~\ref{lem-08}, this implies that $[\gamma^1]\ast i_{\atlasmap{\operatorname{z}}{}}=
[\gamma^2]\ast i_{\atlasmap{\operatorname{z}}{}}$, so (\hyperref[A4]{A4}) holds.\\

Lastly, let us prove (\hyperref[A5]{A5}), so let us fix any pair of reduced orbifold atlases
$\atlas{X},\atlas{Y}$, any object $\groupoidtot{X}{}$ in $\PEEGpd$, any pair of morphisms
$\atlasmap{f}{1},
\atlasmap{f}{2}:\atlas{X}\rightarrow\atlas{Y}$, any Morita equivalence $\groupoidmaptot{\psi}{}:
\groupoidtot{X}{}\rightarrow
\functor{F}_0^{\red}(\atlas{X})$ and any natural transformation $\alpha:\functor{F}_1^{\red}(\atlasmap{f}{1})
\circ\groupoidmaptot{\psi}{}\Rightarrow\functor{F}_1^{\red}(\atlasmap{f}{2})\circ
\groupoidmaptot{\psi}{}$. By
Lemma~\ref{lem-11} there are a reduced orbifold atlas $\atlas{Z}$ and a Morita equivalence
$\groupoidmaptot{\phi}{}:
\functor{F}_0^{\red}(\atlas{Z})\rightarrow\groupoidtot{X}{}$. By Proposition~\ref{prop-04} there is a
unique weak equivalence $\atlasmap{\operatorname{u}}{}:\atlas{Z}\rightarrow\atlas{X}$ such that
$\functor{F}_1^{\red}(\atlasmap{\operatorname{u}}{})=\groupoidmaptot{\psi}{}\circ
\groupoidmaptot{\phi}{}$. By Lemma~\ref{lem-18}, there are a reduced orbifold atlas $\atlas{X}'$
and a weak equivalence $\atlasmap{\operatorname{v}}{}:\atlas{X}'\rightarrow\atlas{Z}$, such that
$\atlasmap{\operatorname{u}}{}\circ\atlasmap{\operatorname{v}}{}$ is a refinement.
Then let us consider the $2$-morphism

\begin{equation}\label{eq-24}
\alpha\ast i_{\groupoidmaptot{\phi}{}\circ\functor{F}_1^{\red}(\atlasmap{\operatorname{v}}{})}:\,
\functor{F}_1^{\red}(\atlasmap{f}{1}\circ\atlasmap{\operatorname{u}}{}\circ
\atlasmap{\operatorname{v}}{})
\Longrightarrow\functor{F}_1^{\red}(\atlasmap{f}{2}\circ\atlasmap{\operatorname{u}}{}\circ
\atlasmap{\operatorname{v}}{}).
\end{equation}

By Lemma~\ref{lem-08} there is a unique $2$-morphism

\[[\delta]:\,\atlasmap{f}{1}\circ\atlasmap{\operatorname{u}}{}\circ\atlasmap{\operatorname{v}}{}
\Longrightarrow
\atlasmap{f}{2}\circ\atlasmap{\operatorname{u}}{}\circ\atlasmap{\operatorname{v}}{}\]
in $\RedAtl$, such that $\functor{F}_2^{\red}([\delta])$ is equal to \eqref{eq-24}. Then
(\hyperref[A5]{A5}) is satisfied if we choose $A'_{\CATA}:=\atlas{X}'$, $A'_{\CATB}:=
\functor{F}_0^{\red}(\atlas{X}')$, $\operatorname{v}_{\CATA}:=\atlasmap{\operatorname{u}}{}\circ
\atlasmap{\operatorname{v}}{}:\atlas{X}'\longrightarrow\atlas{X}$, $\operatorname{z}_{\CATB}:=
\id_{\functor{F}_0^{\red}(\atlas{X}')},$

\[\operatorname{z}'_{\CATB}:=\groupoidmaptot{\phi}{}\circ\functor{F}_1^{\red}
(\atlasmap{\operatorname{v}}{}):\,\functor{F}_0(\atlas{X}')\longrightarrow\groupoidtot{X}{}\]
(this is a Morita equivalence because composition of Morita  equivalences), $\alpha_{\CATA}:=
[\delta]$ and if we define $\sigma_{\CATB}$ as the $2$-identity of

\[\functor{F}_1^{\red}(\atlasmap{\operatorname{u}}{}\circ\atlasmap{\operatorname{v}}{})=
\groupoidmaptot{\psi}{}\circ\groupoidmaptot{\phi}{}\circ\functor{F}_1^{\red}
(\atlasmap{\operatorname{v}}{}).\]
\end{proof}

\begin{rem}
As we said in the Introduction, $\PEEGpd\left[\WPEEGpdinv\right]$ is the bicategory of reduced
differentiable
orbifolds in the language of Lie groupoids; so the previous theorem proves that the bicategory
$\RedOrb$ just defined is the first known bicategory of reduced orbifolds in the language of
reduced orbifold atlases. Compared to $\PEEGpd\left[\WPEEGpdinv\right]$, the
main advantage of $\RedOrb$ for differential geometers is the fact that all the definitions 
used for the construction of such a bicategory
do not require any knowledge of Lie groupoids or differentiable stacks, but they use only
the notion of reduced orbifold atlases, local lifts and changes of charts.
\end{rem}

\section{An equivalence between $\RedOrb$ and the 2-category of effective orbifolds
described in terms of differentiable stacks}
As we mentioned in the introduction, a very convenient way to define a $2$-category of orbifolds
is by exhibiting it as a full $2$-subcategory of the $2$-category of $C^{\infty}$-stacks (these
are called ``differentiable stacks'' in several papers, see for example~\cite{Pr}). For the
Grothendieck topology used for such stacks, we refer to~\cite[Definition~8.1]{J2}. A
$C^{\infty}$-stack is called an \emph{orbifold} (see~\cite[Definition~9.25]{J2}) if it is
equivalent to the stack $[\groupname{X}{}]$ associated to a proper, \'etale groupoid
$(\groupname{X}{})$. In
particular (see again~\cite[Definition~9.25]{J2}) every orbifold is a separated, locally finitely
presented Deligne-Mumford $C^{\infty}$-stack. An orbifold $\stack{X}$ is called \emph{effective}
or \emph{reduced} (see~\cite[Definition~1.9.4]{J1}) if for every point $[x]\in
\stack{X}_{\operatorname{top}}$ there exists a linear effective action of $G:=
\operatorname{Iso}_{\stack{X}}([x])$ on some $\mathbb{R}^n$, a $G$-invariant open neighborhood $\tX$
of $0$ in $\mathbb{R}^n$ and a $1$-morphism $i:[\tX/G]\rightarrow\stack{X}$, which is an
equivalence with an open neighborhood of $x$ in $\stack{X}$ with $i_{\operatorname{top}}(0)=[x]$
(if $\stack{X}$ is not effective, we are in the same setup but the action of each $G$ is not
required to be effective). Equivalently, an orbifold is effective if and only if it
is associated to a proper, \'etale, \emph{effective} groupoid.\\

According to~\cite{J2} we write $\Orb$ and $\OrbEff$ for the full
$2$-subcategories of orbifolds, respectively of effective orbifolds, in the $2$-category of
$C^{\infty}$-stacks (or, equivalently, in the $2$-category of Deligne-Mumford $C^{\infty}$-stacks).
We recall that by~\cite[Corollary~43]{Pr} there is an equivalence of bicategories

\[\widetilde{\functor{H}}:\,\EGpd\left[\WEGpdinv\right]\longrightarrow\DiffStacks\]
and that by~\cite[Theorem~9.26]{J2} there is an equivalence of bicategories induced by
$\widetilde{\functor{H}}$:

\[\functor{H}:\,\PEGpd\left[\WPEGpdinv\right]\longrightarrow\Orb.\]

Therefore we get easily that there is also is an equivalence of bicategories induced by
$\functor{H}$:

\[\functor{H}^{\red}:\,\PEEGpd\left[\WPEEGpdinv\right]\longrightarrow\OrbEff.\]

By considering the composition:

\[\begin{tikzpicture}[xscale=4.9,yscale=-1.2]
    \node (A0_0) at (0, 0) {$\RedOrb$};
    \node (A0_1) at (1, 0) {$\PEEGpd\left[\WPEEGpdinv\right]$};
    \node (A0_2) at (2, 0) {$\OrbEff$};
    \path (A0_0) edge [->]node [auto] {$\scriptstyle{\functor{G}^{\red}}$} (A0_1);
    \path (A0_1) edge [->]node [auto] {$\scriptstyle{\functor{H}^{\red}}$} (A0_2);
\end{tikzpicture}\]
we conclude by Theorem~\ref{theo-05} that:

\begin{theo}\label{theo-03}
Assuming the axiom of choice, there is an equivalence between the bicategory
$\RedOrb$ and the $2$-category $\OrbEff$ of
effective orbifolds described as a full $2$-subcategory of the $2$-category of
$C^{\infty}$-Deligne-Mumford stacks.
\end{theo}


\end{document}